\documentclass{amsart}
\usepackage[latin9]{inputenc}
\usepackage{color}
\usepackage{verbatim}
\usepackage{mathtools}
\usepackage{amstext}
\usepackage{amsthm}
\usepackage{amssymb}
\usepackage[unicode=true,
 bookmarks=false,
 breaklinks=false,pdfborder={0 0 1},backref=section,colorlinks=true]
 {hyperref}
\hypersetup{
 urlcolor=blue,citecolor=red}

\makeatletter
\numberwithin{equation}{section}
\theoremstyle{plain}
\newtheorem{thm}{\protect\theoremname}[section]
\theoremstyle{plain}
\newtheorem{lem}[thm]{\protect\lemmaname}
\theoremstyle{plain}
\newtheorem{prop}[thm]{\protect\propositionname}
\theoremstyle{remark}
\newtheorem{rem}[thm]{\protect\remarkname}
\theoremstyle{plain}
\newtheorem{cor}[thm]{\protect\corollaryname}

\usepackage{tikz}

\providecommand{\corollaryname}{Corollary}

\providecommand{\lemmaname}{Lemma}
\providecommand{\propositionname}{Proposition}
\providecommand{\remarkname}{Remark}
\providecommand{\theoremname}{Theorem}

\global\long\def\R{\mathbb{R}}%
\global\long\def\Z{\mathbb{Z}}%
\global\long\def\N{\mathbb{N}}%
\global\long\def\T{\mathbb{T}}%

\providecommand{\corollaryname}{Corollary}
\providecommand{\lemmaname}{Lemma}
\providecommand{\propositionname}{Proposition}
\providecommand{\remarkname}{Remark}
\providecommand{\theoremname}{Theorem}

\makeatother

\providecommand{\corollaryname}{Corollary}
\providecommand{\lemmaname}{Lemma}
\providecommand{\propositionname}{Proposition}
\providecommand{\remarkname}{Remark}
\providecommand{\theoremname}{Theorem}

\makeatletter
\@namedef{subjclassname@2020}{%
  \textup{2020} Mathematics Subject Classification}
\makeatother

\begin{document}
\title[Local well-posedness for generalized Benjamin-Ono]{Low regularity well-posedness for generalized Benjamin-Ono equations
on the circle}
\author{Kihyun Kim}
\email{khyun1215@kaist.ac.kr}
\address{Department of Mathematical Sciences, Korea Advanced Institute of Science
and Technology, 291 Daehak-ro, Yuseong-gu, Daejeon 34141, Korea}
\author{Robert Schippa}
\email{robert.schippa@kit.edu}
\address{Fakult\"at f\"ur Mathematik, Karlsruher Institut f\"ur Technologie,
Englerstrasse 2, 76131 Karlsruhe, Germany}

\begin{abstract}
New low regularity well-posedness results for the generalized Ben\-ja\-min-\-Ono
equations with quartic or higher nonlinearity and periodic boundary
conditions are shown. We use the short-time Fourier transform restriction
method and modified energies to overcome the derivative loss. Previously,
Molinet--Ribaud established local well-posedness in $H^{1}(\T,\R)$
via gauge transforms. We show local existence and a priori estimates
in $H^{s}(\T,\R)$, $s>1/2$, and local well-posedness in $H^{s}(\T,\R)$,
$s\geq3/4$ without using gauge transforms. In case of quartic nonlinearity we prove global existence of solutions conditional upon small initial data. 
\end{abstract}

\maketitle

\section{Introduction}

In this article we improve the well-posedness theory for the \emph{$k$-generalized
periodic Benjamin-Ono equation} in $L^{2}$-based Sobolev spaces 
\begin{equation}
\left\{ \begin{array}{cl}
\partial_{t}u+\mathcal{H}\partial_{xx}u & =\mp\partial_{x}(u^{k})\quad(t,x)\in\mathbb{R}\times\mathbb{T},\\
u(0) & =u_{0}\in H^{s}(\mathbb{T},\R),
\end{array}\right.\label{eq:GeneralizedBenjaminOnoEquation}
\end{equation}
where $k\geq4$ and $\mathbb{T}=\R/(2\pi\mathbb{Z})$. Throughout
this article, $\mathcal{H}$ denotes the \emph{Hilbert transform},
i.e., 
\[
\mathcal{H}:L^{2}(\mathbb{T})\rightarrow L^{2}(\mathbb{T}),\quad(\mathcal{H}f)^{\wedge}(\xi)=-i\text{sgn}(\xi)\hat{f}(\xi).
\]

Note that real-valued initial data give rise to real-valued solutions.
We shall implicitly consider real-valued initial data in the following,
unless stated otherwise.

By \emph{local well-posedness} we refer to the following: the data-to-solution
mapping $S_{T}^{\infty}:H^{\infty}(\mathbb{T})\rightarrow C([0,T],H^{\infty}(\T))$
assigning smooth initial data to smooth solutions admits a continuous
extension $S_{T}^{s}:H^{s}\rightarrow C_{T}H^{s}$ with $T=T(\Vert u_{0}\Vert_{H^{s}})$,
which can be chosen continuously on $\Vert u_{0}\Vert_{H^{s}}$. Existence
and continuity of $S_{T}^{s}:H^{s}\rightarrow C_{T}H^{s}$ for $s>3/2$
follows from the classical energy method (cf. \cite{BonaSmith1975,AbdelouhabBonaFellandSaut1989}). Solutions to \eqref{eq:GeneralizedBenjaminOnoEquation}
on the real line admit the scaling symmetry 
\[
u(t,x)\rightarrow\lambda^{-\frac{1}{k-1}}u(\lambda^{-2}t,\lambda^{-1}x).
\]
This leads to the scaling critical space $\dot{H}^{s_{c}}(\R)$, $s_{c}(k)=\frac{1}{2}-\frac{1}{k-1}$,
which is the largest $L^{2}$-Sobolev space for which local well-posedness
can be expected.

Conserved quantities of solutions to \eqref{eq:GeneralizedBenjaminOnoEquation}
are the \emph{mass}, i.e., the $L^{2}$-norm, 
\[
M(u_{0})=\int_{\T}u_{0}^{2}dx,
\]
and the energy, related with the $H^{1/2}$-norm, 
\[
E(u_{0})=\int_{\T}\frac{u_{0}\mathcal{H}\partial_{x}u_{0}}{2}\pm\frac{u_{0}^{k+1}}{k+1}dx.
\]
The $\pm$ signs correspond to \eqref{eq:GeneralizedBenjaminOnoEquation}.

When $k$ is even, there is no difference between the dynamics of
\eqref{eq:GeneralizedBenjaminOnoEquation} with $\pm$ signs in front
of the nonlinearity, because if $u$ is a solution to \eqref{eq:GeneralizedBenjaminOnoEquation}
with $+$ sign, then $-u$ is a solution to \eqref{eq:GeneralizedBenjaminOnoEquation}
with $-$ sign, and vice versa. However, when $k$ is odd, there is
a big difference between the dynamics. \eqref{eq:GeneralizedBenjaminOnoEquation}
with a minus sign is referred to as \emph{defocusing} equation and
with a plus sign as \emph{focusing} equation. The energy is positive
definite, and a local well-posedness result in $H^{1/2}$ can be extended
globally in the defocusing case. On the contrary, in the focusing
case, Martel--Pilod \cite{MartelPilod2017} recently proved the existence
of minimal blow-up solutions in the energy space for $k=3$ on the
real line (see also \cite{KenigMartelRobbiano2011}). This indicates
blow-up in the periodic case for focusing nonlinearities.

Equations \eqref{eq:GeneralizedBenjaminOnoEquation} have mostly been
studied on the real line, where the dispersive effects are stronger
and the solutions are easier to handle. We digress for a moment to
review the results on the real line to highlight key-points of the
local well-posedness on the real line. Some transpire to the periodic
case. We shall refer to the most recent results and the references
therein.

The \emph{Benjamin-Ono equation} $(k=2)$ is completely integrable
and has been studied extensively. We first note that the High$\times$Low$\rightarrow$High-interaction
\begin{equation}
\partial_{x}(P_{N}uP_{K}u),\label{eq:HighLowHighInteraction}
\end{equation}
with $P_{L}$, $L\in2^{\mathbb{N}_{0}}$, localizing to frequencies
of size about $L$, leads to derivative loss. This makes it impossible
to solve the Benjamin-Ono equation via Picard iteration (cf. \cite{MolinetSautTzvetkov2001,KochTzvetkov2005}).
Via gauge transform (introduced by Tao in \cite{Tao2004}), Ionescu--Kenig
proved global well-posedness in $L^{2}(\R)$ in \cite{IonescuKenig2007}
making use of Fourier restriction spaces; see also \cite{MolinetPilod2012}.
Ifrim--Tataru significantly simplified the proof by normal form transformations
and relying only on Strichartz spaces in \cite{IfrimTataru2019}.
Recently, Talbut proved a priori estimates up to the scaling critical
regularity in \cite{Talbut2018} via complete integrability as well
on the real line as on the circle.

For the \emph{modified Benjamin-Ono equation} ($k=3$), Kenig--Takaoka \cite{KenigTakaoka2006} showed
global well-posedness for real-valued initial data in the energy
space in case of defocusing nonlinearity. Guo extended this to complex-initial data in \cite{Guo2011MBO}. He used smoothing effects on the real
line instead of the gauge transform to overcome the derivative loss.
Moreover, he proved a priori estimates up to $s>1/4$ using short-time
Fourier transform restriction. In this work becomes clear that for
$k\geq3$ the High$\times$High$\times$High$\to$High-interaction
is also problematic below $H^{1/2}$: 
\[
P_{N}\partial_{x}(P_{N_{1}}uP_{N_{2}}uP_{N_{3}}u),
\]
where $N\sim N_{1}\sim N_{2}\sim N_{3}$. In this case the resonance
(see Section \ref{sec:EnergyEstimateSolutions}) can become arbitrarily
small. Furthermore, in \cite{Guo2011MBO} was shown how smoothing
effects on the real line can replace the gauge transform for $k=3$.
For $k=4$, Vento \cite{Vento2010} proved local well-posedness in
$H^{1/3}$, which turned out to be the limit of fixed point arguments,
and reached the scaling critical regularity for $k\geq5$; see also
\cite{MolinetRibaud2004SmallData,MolinetRibaud2004LargeData,BurqPlanchon2006}.

There are fewer results for \eqref{eq:GeneralizedBenjaminOnoEquation}
with periodic boundary conditions. Molinet \cite{Molinet2007} adapted
the gauge transform to the periodic Benjamin-Ono equation to prove
global well-posedness in $L^{2}(\T)$. Herr \cite{Herr2008} showed that
the Benjamin-Ono equation with periodic boundary conditions cannot be solved 
via Picard iteration directly. By complete integrability,
G\'erard--Kappeler--Topalov \cite{GerardKappelerTopalov2020} proved
global well-posedness up to the scaling critical regularity $s_{c}=-1/2$;
see also \cite{GerardKappeler2019}. For $k=3$, Guo--Lin--Molinet
\cite{GuoLinMolinet2014} showed global well-posedness in the energy
space by adapting the gauge transform and using Fourier restriction
spaces. For $k=3$, the second author proved existence and a priori
estimates for $s>1/4$ using short-time Fourier restriction, but not
relying on gauge transforms, in \cite{Schippa2017MBO}. For $k\geq4$,
Molinet--Ribaud \cite{MolinetRibaud2009} proved local well-posedness
in $H^{1}(\T)$ via gauge transforms and Strichartz estimates.

At last, we address ill-posedness issues. Firstly, we remark that Christ's argument \cite{Christ2004}, originally applied to the quadratic
derivative nonlinear Schr\"odinger equation
\begin{equation*}
i \partial_t u + \partial_{xx} u = i u \partial_x u, \quad (t,x) \in \R \times \T,
\end{equation*}
shows norm inflation for complex-valued initial data at any Sobolev regularity.\\
Secondly, the High$\times$Low-interaction \eqref{eq:HighLowHighInteraction} leads to the failure of the multilinear $X^{s,b}$-estimate\footnote{We set $b=1/2$ for simplicity. Note that in this limiting case one actually has to consider a smaller function space.}
\begin{equation*}
\Vert \partial_x(u_1 \ldots u_k) \Vert_{X^{s,-1/2}} \lesssim \prod_{i=1}^k \Vert u_i \Vert_{X^{s,1/2}},
\end{equation*}
also after removing trivial resonances.\\
This contrasts with the generalized KdV-equations on the circle
\begin{equation*}
\partial_t u + \partial_x^3 u = \partial_x (u^k),
\end{equation*}
where Colliander \emph{et al.} \cite{CollianderKeelStaffilaniTakaokaTao2004} showed the crucial multilinear $X^{s,b}$-estimate for $s=1/2$ after renormalizing the nonlinearity.

We now state our main results. Our first result shows the local existence
and a priori estimates for $s>1/2$. 
\begin{thm}[Local existence and a priori estimates]
\label{thm:APrioriGBOI} Let $k\geq4$ and $s>1/2$. Then, for any
$u_{0}\in H^{s}(\T)$ with $\|u_{0}\|_{H^{s}}\leq R$, there is $T=T(R)>0$
such that a solution $u\in C_{T}H^{s}$ to \eqref{eq:GeneralizedBenjaminOnoEquation}
exists in the sense of distributions and the a priori estimate 
\begin{equation}
\sup_{t\in[0,T]}\Vert u(t)\Vert_{H^{s}}\lesssim_{R}\Vert u_{0}\Vert_{H^{s}}\label{eq:APrioriEstimateGBO}
\end{equation}
holds true.
\end{thm}
The result can be globalized in case of quartic nonlinearity.
\begin{thm}[Global existence for quartic nonlinearity]
\label{thm:APrioriGBOII}
In the quartic case $k=4$, the statement of Theorem \ref{thm:APrioriGBOI} holds under the assumption
$\|u_{0}\|_{H^{1/2}}\leq R$ instead of $\|u_{0}\|_{H^{s}}\leq R$.
If in addition $R$ is sufficiently small, then we can find a global
solution $u$.
\end{thm}

Our second result shows local well-posedness for $s\geq3/4$. Since
the difference equation satisfies less symmetries than the original
equation, we can only prove the following weaker result for continuous
dependence.
\begin{thm}[Local well-posedness]
\label{thm:LWPGBO} Let $k\geq4$ and $s\geq3/4$. Then, we find
\eqref{eq:GeneralizedBenjaminOnoEquation} to be locally well-posed. 
\end{thm}

\ 

\emph{Comments on Theorems \ref{thm:APrioriGBOI}, \ref{thm:APrioriGBOII}, and \ref{thm:LWPGBO}.}

\emph{1. The restriction $s>1/2$ in Theorem \ref{thm:APrioriGBOI}.
}In \cite[Theorem 1.1]{Schippa2017MBO}, the second author proved
the analog of Theorem \ref{thm:APrioriGBOI} in case $k=3$ with improved
range $s>1/4$. This improvement relies on the resonance relation.
In $k=3$, the zero set of the resonance function is nontrivial, but
the (symmetrized) multiplier used in the energy estimates simultaneously
vanishes. However, this does not necessarily hold for $k\geq4$ and
our restriction $s>1/2$ comes from these resonant interactions. In
the quartic case, we can show a priori estimates in the short-time
function space $F^{1/2}$. In the quintic and higher cases, our argument
merely gives a priori estimates in the Besov refinements $F_{1}^{1/2}$.
As these miss the energy space, we cannot extend them globally. See
Section \ref{sec:EnergyEstimateSolutions} for more details.

\emph{2. The restriction $s\geq3/4$ of Theorem \ref{thm:LWPGBO}.}
To prove local well-posedness, we need to consider the difference
equation. When deriving energy estimates, the lack of symmetry does
not allow for the same favorable cancellations, as for solutions.
The restriction $s\geq3/4$ again comes from the resonant interactions.
More precisely, it comes from High$\times$High$\times$High$\times$High
interactions. See Section \ref{subsec:diff-est-case3} for details.

\emph{3. Gauge transforms.} Our method does not make use of gauge
transforms in contrast to the works \cite{Vento2010,MolinetRibaud2009}.
Here we mean by gauge transforms the usual ones used in the previous
literatures.\footnote{Any transform of the kind $u\mapsto e^{i\chi}u$ is a gauge transforms.
For the current discussion we focus on the gauge transform introduced
by Tao \cite{Tao2004} and variants thereof.}

The main reason is that the gauge transforms in our case do not behave
as well as on the real line or for $k\in\{2,3\}$. By a gauge transform,
it is possible to delete Low$\times$High interactions of the nonlinearity.
However, we have to deal with error terms generated by the gauge transforms
(e.g. when $\partial_{t}$ falls onto the gauge transforms). On the
real line, these can be estimated\footnote{Or, one can use a fixed-time gauge transform as in Vento \cite{Vento2010}.}
using better linear estimates than on the torus. When $k\in\{2,3\}$,
these errors have better structure than those of $k\geq4$; compare
\cite{GuoLinMolinet2014} of $k=3$ and \cite{MolinetRibaud2009}
of $k\geq4$. When $k\geq4$, $\partial_{t}$ acting on the gauge
transforms leads to problematic High$\times$High$\to$Low interactions,
so we choose to avoid using gauge transforms. Hence, we have to deal
with Low$\times$High interactions in the original nonlinearity, and
we choose to work with short-time Fourier restriction spaces to recover
the derivative loss (cf. \eqref{eq:HighLowHighInteraction}).

Avoiding the use of gauge transforms, our method can be adopted to
other models where the gauge transforms become very involved, if available
at all. Examples include the dispersion-generalized models (cf. \cite{Schippa2020DispersionGeneralizedBenjaminOno}). Moreover, for quadratic nonlinearities, an improvement of the energy method was
proposed by Molinet--Vento \cite{MolinetVento2015}. This makes use
of a precise comprehension of the resonance function and avoids gauge transforms, too.

\emph{4. Extension of the methods for $k\in\{2,3\}$.} We remark that
our method can be easily modified to yield the same results (i.e.,
$s>\frac{1}{4}$ in Theorem \ref{thm:APrioriGBOI}, and $s>1/2$ for
$k=2$ and $s\geq\frac{3}{4}$ for $k=3$ in Theorem \ref{thm:LWPGBO})
for the cases $k\in\{2,3\}$. However, there are stronger results
for $k\in\{2,3\}$ as mentioned above.

\ 

To prove Theorems \ref{thm:APrioriGBOI}-\ref{thm:LWPGBO}, we
use short-time Fourier restriction spaces as in \cite{Schippa2017MBO}.
We extend the approach of \cite{Schippa2017MBO} to $k\geq4$ in the
present paper. In the following we elaborate on short--time Fourier
restriction and the proof of the theorems. Since the body of literature
on short-time Fourier restriction is already huge, we do not aim for
an exhaustive review of references. We also refer to the references
within the discussed literature and the PhD thesis of the second author
\cite{Schippa2019PhDThesis}.

The first key ingredient of our method is the use of short-time Fourier
restriction spaces. This relies on the observation that the frequency
dependent time localization allows to prove low regularity results
for quasilinear dispersive equations. By \emph{quasilinear} we mean
that the equations cannot be solved by fixed point argument in $L^{2}$-based
Sobolev spaces. In Euclidean space, early works on short-time Fourier
restrictions are due to Koch--Tataru \cite{KochTataru2007}, Christ
\emph{et al.} \cite{ChristCollianderTao2008}, and Ionescu \emph{et
al.} \cite{IonescuKenigTataru2008}. Guo \emph{et al.} observed in
\cite{GuoPengWangWang2011} that the frequency dependent time localization
$T=T(N)=N^{-1}$ allows to overcome the derivative loss for High $\times$
Low $\rightarrow$ High--interaction \eqref{eq:HighLowHighInteraction}
on the real line. This showed how to avoid the gauge transform and
proved that inviscid limits recover solutions to the Benjamin-Ono
equation.

The second author observed \cite{Schippa2017MBO,Schippa2019PhDThesis}
that this extends to periodic solutions. Although dispersive effects
on tori are weaker, for time intervals of length $T=T(N)=N^{-1}$
Schr\"odinger wave packets cannot distinguish between Euclidean
space and compact manifolds. Hence, Strichartz estimates on frequency
dependent time intervals remain valid on compact manifolds. This was
observed for linear estimates by Staffilani--Tataru \cite{StaffilaniTataru2002}
and Burq \emph{et al.} \cite{BurqGerardTzvetkov2002} and for bilinear
estimates by Moyua--Vega \cite{MoyuaVega2008} and Hani \cite{Hani2012}.

The second key ingredient is
to use cancellation effects, which allows to control low Sobolev norms,
in a similar spirit with the $I$\emph{-method} (cf. \cite{CollianderKeelStaffilaniTakaokaTao2002}).
For differences of solutions, due to less symmetries, this is known
as \emph{normal form transformations} or \emph{modified energies}
as used by Kwon \cite{Kwon2008} and Kwak \cite{Kwak2016,Kwak2018}.
We are not aware of previous instances of short-time Fourier restriction
analysis combined with modified energies for quartic or higher nonlinearities.
We hope that the arguments of the present work can be applied more
generally. For instance, the model 
\[
\partial_{t}u+\mathcal{H}\partial_{xx}u=\partial_{x}(e^{u})
\]
seems to be in the scope of the methods of the paper.

We end the introduction by explaining key steps of the short-time
analysis. Further details of the proofs are provided in Section \ref{sec:ProofMainResults}.
For solutions, the above program leads to the following set of estimates
for solutions $u$ in the short-time function space $F^{s}(T)$, $s>1/2$,
$\varepsilon=\varepsilon(s)>0$: 
\begin{equation}
\left\{ \begin{array}{cl}
\Vert u\Vert_{F^{s}(T)} & \lesssim\Vert u\Vert_{E^{s}(T)}+\Vert\partial_{x}(u^{k})\Vert_{N^{s}(T)}\\
\Vert\partial_{x}(u^{k})\Vert_{N^{s}(T)} & \lesssim\Vert u\Vert_{F^{s}(T)}^{k}\\
\Vert u\Vert_{E^{s}(T)}^{2} & \lesssim\Vert u_{0}\Vert_{H^{s}}^{2}+\Vert u\Vert_{F^{s-\varepsilon}(T)}^{k+1}+\Vert u\Vert_{F^{s-\varepsilon}(T)}^{2k}.
\end{array} \right. \label{eq:a-priori-intro} 
\end{equation}

This gives a priori estimates and existence of solutions by standard
bootstrap and compactness arguments (cf. \cite{GuoOh2018}) for \emph{small}
initial data $u_{0}$. To deal with large initial data, we rescale
the torus yielding small initial data on tori with large period $\lambda$.
Molinet introduced this argument in the context of short-time Fourier
restriction in \cite{Molinet2012}; see also \cite{Schippa2017MBO}.
We omit the standard arguments and refer to the literature.

For the proof of Theorem \ref{thm:LWPGBO} we firstly show Lipschitz
continuity in the weaker norm $H^{-1/4}$ by the set of estimates
for $v=u_{1}-u_{2}$, $u_{i}$ solutions to \eqref{eq:GeneralizedBenjaminOnoEquation}
in $F^{s}(T)$: 
\begin{equation}
\left\{ \begin{array}{cl}
\Vert v\Vert_{F^{-1/4}(T)} & \lesssim\Vert\partial_{x}(vu^{k-1})\Vert_{N^{-1/4}(T)}+\Vert v\Vert_{E^{-1/4}(T)}\\
\Vert\partial_{x}(vu^{k-1})\Vert_{N^{-1/4}(T)} & \lesssim\Vert v\Vert_{F^{-1/4}(T)}(\Vert u_{1}\Vert_{F^{3/4}(T)}+\Vert u_{2}\Vert_{F^{3/4}(T)})^{k-1}\\
\Vert v\Vert_{E^{-1/4}(T)}^{2} & \lesssim\Vert v(0)\Vert_{H^{-1/4}}^{2}\\
 & \qquad+\Vert v\Vert_{F^{-1/4}(T)}^{2}(\Vert u_{1}\Vert_{F^{3/4}(T)}+\Vert u_{2}\Vert_{F^{3/4}(T)})^{k-1}\\
 & \qquad+\Vert v\Vert_{F^{-1/4}(T)}^{2}(\Vert u_{1}\Vert_{F^{3/4}(T)}+\Vert u_{2}\Vert_{F^{3/4}(T)})^{2k-2}.
\end{array}\right. \label{eq:weak-Lipschitz-intro}
\end{equation}
In the above display $u^{m}$ denotes a linear combination of $u_{1}^{i}u_{2}^{m-i}$,
$i=0,\ldots,m$. This set of estimates yields Lipschitz continuous
dependence in $H^{-1/4}$ for \emph{small} initial data in $H^{3/4}$.
We extend this to large initial data by rescaling the torus as above.
To prove continuous dependence in $H^{3/4}$, we prove in addition
the following set of estimates: 
\begin{equation}
\left\{ \begin{array}{cl}
\Vert v\Vert_{F^{3/4}(T)} & \lesssim\Vert\partial_{x}(vu^{k-1})\Vert_{F^{3/4}(T)}+\Vert v\Vert_{E^{3/4}(T)}\\
\Vert\partial_{x}(vu^{k-1})\Vert_{N^{3/4}(T)} & \lesssim\Vert v\Vert_{F^{3/4}(T)}(\Vert u_{1}\Vert_{F^{3/4}(T)}+\Vert u_{2}\Vert_{F^{3/4}(T)})^{k-1}\\
\Vert v\Vert_{E^{3/4}(T)}^{2} & \lesssim\Vert v(0)\Vert_{H^{3/4}}^{2}\\
 & \quad+\Vert v\Vert_{F^{3/4}(T)}^{2}(\Vert v\Vert_{F^{3/4}(T)}+\Vert u_{2}\Vert_{F^{3/4}(T)})^{k-1}\\
 & \quad+\Vert v\Vert_{F^{-1/4}(T)}\Vert v\Vert_{F^{3/4}(T)}\Vert u_{2}\Vert_{F^{7/4}(T)}\\
 & \qquad\quad(\Vert v\Vert_{F^{3/4}(T)}+\Vert u_{2}\Vert_{F^{3/4}(T)})^{k-2}\\
 & \quad+\Vert v\Vert_{F^{3/4}(T)}^{2}(\Vert v\Vert_{F^{3/4}(T)}+\Vert u_{2}\Vert_{F^{3/4}(T)})^{2k-2}\\
 & \quad+\Vert v\Vert_{F^{-1/4}(T)}\Vert v\Vert_{F^{3/4}(T)}\Vert u_{2}\Vert_{F^{7/4}(T)}\\
 & \qquad\quad(\Vert v\Vert_{F^{3/4}(T)}+\Vert u_{2}\Vert_{F^{3/4}(T)})^{2k-3}.
\end{array}\right.\label{eq:continuity-intro}
\end{equation}
We finish the proof of Theorem \ref{thm:LWPGBO} for \textit{small}
initial data by a variant of the Bona--Smith method (cf. \cite{BonaSmith1975,IonescuKenigTataru2008,Schippa2020DispersionGeneralizedBenjaminOno}).

The case of large initial data additionally requires rescaling to
small initial data on tori with large periods as above. For this,
we need to modify the Sobolev weights for the frequencies less than
$1$ \eqref{eq:DefLowHighSobolevSpace}.

An alternative approach to make the bootstrap argument work for large data is to trade regularity in modulation for powers of $T$ (cf. \cite[Lemma~3.4]{GuoOh2018}). In the present context this seems difficult to realize as our proof of the nonlinear and energy estimates require the full range of modulation regularity for $U^2$-/$V^2$-spaces. This corresponds to the endpoints $b=-1/2$ and $b=1/2$ and does not allow for applying \cite[Lemma~3.4]{GuoOh2018} or a suitable variant for $U^2$-/$V^2$-spaces.


\textit{Outline of the paper.} In Section \ref{sec:Notations} we
introduce notations, function spaces, and recall short--time (bilinear)
Strichartz estimates. In Section \ref{sec:ProofMainResults} we conclude
the proofs of the main results with the crucial short-time nonlinear
and energy estimates at hand. In Section \ref{sec:NonlinearEstimates}
we propagate the nonlinear interaction. In Section \ref{sec:EnergyEstimateSolutions}
we bound the energy norm for solutions and in Section \ref{sec:DifferenceEstimate}
the energy norm for differences of solutions.

In the following we assume that $k$, the power of the nonlinearity,
satisfies $k\geq4$. Moreover, we suppress dependence on $k$ for
the implicit constants. The parameter $\lambda$ (see Section \ref{subsec:Fourier})
is always assumed to be $\lambda\in2^{\N_{0}}$ and $\lambda\geq1$.
The dyadic frequencies range from $2^{\Z}\cap[\lambda^{-1},\infty)$.

\section{\label{sec:Notations}Function spaces, and linear and bilinear short-time
estimates}

\subsection{\label{subsec:Fourier}Fourier analysis on $\lambda\T$}

As mentioned above, the (local-in-time) large-data-theory is reduced
to the small-data-theory via a scaling argument on circles. For this
purpose, we need to develop our arguments working uniformly for functions
with large periods. Set $\lambda\T=\R/(2\pi\lambda\Z)$ with $\lambda\geq1$.
The Fourier transform of a function on $\lambda\T$ will have the
domain $\Z/\lambda$.

Throughout this article, we assume 
\[
\lambda\in2^{\N_{0}}\quad\text{so that}\quad\lambda\geq1.
\]
We will also assume that the dyadic frequencies always range in $[\lambda^{-1},\infty)$,
e.g. 
\[
N,N_{i},M,M_{i},K,K_{i},\dots\in2^{\Z}\cap[\lambda^{-1},\infty).
\]

We define the Lebesgue spaces $L^{p}(\lambda\T)$ on $\lambda\T$
through the norm 
\[
\Vert f\Vert_{L_{\lambda}^{p}}=\left(\int_{\lambda\T}|f(x)|^{p}dx\right)^{1/p},\quad1\leq p<\infty,
\]
with the usual modification for $p=\infty$.

We turn to the Fourier transform on $\lambda\T$. As guideline for
the conventions from below, we require that Plancherel's theorem remains
valid; see \cite{CollianderKeelStaffilaniTakaokaTao2002}. The Fourier
coefficients of $f\in L^{1}(\lambda\T)$ are given by 
\[
\hat{f}(\xi)=\int_{\lambda\T}f(x)e^{-i\xi x}dx,\quad\xi\in\Z/\lambda
\]
such that we have the Fourier inversion formula 
\[
f(x)=\frac{1}{\lambda}\sum_{\xi\in\Z/\lambda}\hat{f}(\xi)e^{ix\xi}
\]
and Plancherel's theorem: 
\[
\|f\|_{L_{\lambda}^{2}}^{2}\sim\frac{1}{\lambda}\sum_{\xi\in\Z/\lambda}|\hat{f}(\xi)|^{2}.
\]

The Littlewood-Paley projectors are defined as follows. Let $\chi:\R\to\R_{\geq0}$
be a smooth, compactly supported, radially decreasing function with
$\chi(x)=1$ for $|x|\leq1$ and $\chi(x)=0$ for $|x|\geq2$. For
dyadic\footnote{Recall that we assume the dyadic numbers to range in $2^{\Z}\cap[\lambda^{-1},\infty)$
throughout this article.} $\mu$, set $\chi_{\mu}(x)=\chi_{\lambda^{-1}}(x)=\chi(\lambda x)$
if $\mu=\lambda^{-1}$ and $\chi_{\mu}(x)=\chi(x/2\mu)-\chi(x/\mu)$
otherwise. For the sequence of functions $\{\chi_{\lambda^{-1}},\chi_{2\lambda^{-1}},\chi_{4\lambda^{-1}},\dots\}$,
we denote the corresponding Fourier multipliers by $P_{\lambda^{-1}}$,
$P_{2\lambda^{-1}}$, $\dots$. We refer to these as Littlewood-Paley
projectors. We note that Bernstein's inequality holds as on $\T$
and $\R$, uniformly in $\lambda$.

The $L^{2}$-based Sobolev spaces $H^{s}(\lambda\T)$ on $\lambda\T$
are defined through the norm 
\[
\Vert f\Vert_{H_{\lambda}^{s}}^{2}=\frac{1}{\lambda}\sum_{\xi\in\Z/\lambda}\langle\xi\rangle^{2s}|\hat{f}(\xi)|^{2},\quad s\in\R,
\]
where we denoted $\langle\xi\rangle=(\xi^{2}+1)^{1/2}$. In view of
Plancherel's theorem, $H^{0}(\lambda\T)=L^{2}(\lambda\T)$. In terms
of Littlewood-Paley projectors, we have 
\[
\Vert f\Vert_{H_{\lambda}^{s}}^{2}\approx\sum_{N<1}\|P_{N}f\|_{L_{\lambda}^{2}}^{2}+\sum_{N\geq1}N^{2s}\Vert P_{N}f\Vert_{L_{\lambda}^{2}}^{2}.
\]

However, as \eqref{eq:GeneralizedBenjaminOnoEquation} for $k\geq4$
is $L^{2}$-supercritical, the usual Sobolev spaces $H^{s}$ does
not work well with the scaling argument. The remedy is to consider
a norm with a different weight on low frequencies. Set for $\overline{s}=(s_{1},s_{2})\in\R^{2}$:
\begin{equation}
\Vert f\Vert_{H_{\lambda}^{\overline{s}}}^{2}=\Vert f\Vert_{H_{\lambda}^{s_{1},s_{2}}}^{2}=\frac{1}{\lambda}\sum_{\substack{\xi\in\Z/\lambda,\\
|\xi|<1
}
}(\lambda^{-1}+|\xi|)^{2s_{1}}|\hat{f}(\xi)|^{2}+\frac{1}{\lambda}\sum_{\substack{\xi\in\Z/\lambda,\\
|\xi|\geq1
}
}|\xi|^{2s_{2}}|\hat{f}(\xi)|^{2}.\label{eq:DefLowHighSobolevSpace}
\end{equation}
Note that $H^{s}(\lambda\T)=H^{0,s}(\lambda\T)$. It is convenient
to introduce the notations 
\[
N^{\overline{s}}=N^{s_{1},s_{2}}=\begin{cases}
N^{s_{1}} & \text{if }N<1,\\
N^{s_{2}} & \text{if }N\geq1,
\end{cases}
\]
and 
\begin{align*}
\overline{s}+s' & =(s_{1},s_{2}+s'),\qquad s'\in\R,\\
cs' & =(cs_{1},cs_{2}),\qquad c\in\R.
\end{align*}
so that 
\[
\|f\|_{H_{\lambda}^{\overline{s}}}^{2}\approx\sum_{N<1}N^{2s_{1}}\|P_{N}f\|_{L_{\lambda}^{2}}^{2}+\sum_{N\geq1}N^{2s_{2}}\Vert P_{N}f\Vert_{L_{\lambda}^{2}}^{2}\approx\sum_{N}N^{2\overline{s}}\|P_{N}f\|_{L^{2}}^{2}.
\]

For the scaling argument, we will use $\overline{s}=(s_{1},s_{2})$
such that $s_{c}<s_{1}\leq s_{2}$, i.e., the \emph{subcritical }regularities.
Let $s_{2}=s$ as in our main theorems. Any large data $u_{0}\in H^{s}(\T)$
is reduced to a small data by 
\[
\lambda^{-\frac{1}{k-1}}\Vert u_{0}(\lambda^{-1}\cdot)\Vert_{H_{\lambda}^{\overline{s}}}\to0\text{ as }\lambda\to\infty
\]
More precisely, we have 
\begin{equation}
\lambda^{s_{c}-s_{2}}\|u_{0}\|_{H^{s}}\lesssim\|\lambda^{-\frac{1}{k-1}}u_{0}(\lambda^{-1}\cdot)\|_{H_{\lambda}^{\overline{s}}}\lesssim\lambda^{s_{c}-s_{1}}\|u_{0}\|_{H^{s}}.\label{eq:scaling}
\end{equation}

\subsection{$U^{p}$-/$V^{p}$-function spaces}

We consider short-time $U^{p}$-/$V^{p}$-function spaces as in \cite{Schippa2020ShriraEquation}.
Adapted $U^{p}$-/$V^{p}$-function spaces to treat nonlinear dispersive
equations were introduced in the work \cite{HadacHerrKoch2009,HadacHerrKoch2009Erratum}.
There are several reasons for this choice. Firstly, (bi-)linear estimates
for linear solutions transfer well to these spaces. Secondly, the
duality estimates are also available. Lastly, these spaces behave
nicely with sharp time localizations. We shall be brief and for details
refer to \cite{Schippa2020ShriraEquation}.

For a time interval $I$, we set 
\begin{align*}
\Vert u\Vert_{U_{BO}^{p}(I)_{\lambda}} & =\Vert e^{-t\mathcal{H}\partial_{xx}}u\Vert_{U^{p}(I;L_{\lambda}^{2})},\\
\Vert v\Vert_{V_{BO}^{p}(I)_{\lambda}} & =\Vert e^{-t\mathcal{H}\partial_{xx}}v\Vert_{V^{p}(I;L_{\lambda}^{2})},\\
\Vert f\Vert_{DU_{BO}^{p}(I)_{\lambda}} & =\Vert e^{-t\mathcal{H}\partial_{xx}}v\Vert_{DU^{p}(I;L_{\lambda}^{2})}.
\end{align*}

Let $T\in(0,1]$ and $\overline{s}=(s_{1},s_{2})$ be given. We define
the $F_{\lambda}^{\overline{s}}(T)$-norm (for solutions) and the
$N_{\lambda}^{\overline{s}}(T)$-norm (for nonlinearities) by 
\begin{align*}
\Vert u\Vert_{F_{\lambda}^{\overline{s}}(T)}^{2} & =\sum_{\substack{\lambda^{-1}\leq N\leq1,\\
N\in2^{\Z}
}
}N^{2s_{1}}\Vert P_{N}u\Vert_{U_{BO}^{2}([0,T])_{\lambda}}^{2}+\sum_{N\in2^{\N_{0}}}N^{2s_{2}}\sup_{\substack{I\subseteq[0,T],\\
|I|=N^{-1}
}
}\Vert P_{N}u\Vert_{U_{BO}^{2}(I)_{\lambda}}^{2},\\
\Vert v\Vert_{N_{\lambda}^{\overline{s}}(T)}^{2} & =\sum_{\substack{\lambda^{-1}\leq N\leq1,\\
N\in2^{\Z}
}
}N^{2s_{1}}\Vert P_{N}u\Vert_{DU_{BO}^{2}([0,T])_{\lambda}}^{2}+\sum_{N\in2^{\N_{0}}}N^{2s_{2}}\sup_{\substack{I\subseteq[0,T],\\
|I|=N^{-1}
}
}\Vert P_{N}u\Vert_{DU_{BO}^{2}(I)_{\lambda}}^{2}.
\end{align*}
As in \cite{Schippa2020ShriraEquation}, if $T\leq N^{-1}$, read
(likewise for $V^{p}$ and $DU^{p}$): 
\[
\sup_{\substack{I\subseteq[0,T],\\
|I|=N^{-1}
}
}\Vert f\Vert_{U_{BO}^{p}(I)}=\Vert f\Vert_{U_{BO}^{p}([0,T])},
\]
We define the energy norm $E_{\lambda}^{\overline{s}}(T)$ (for solutions)
by 
\[
\Vert u\Vert_{E_{\lambda}^{\overline{s}}(T)}^{2}=\Vert P_{\leq1}u(0)\Vert_{H_{\lambda}^{\overline{s}}}^{2}+\sum_{N>1}N^{2s_{2}}\sup_{t\in[0,T]}\Vert P_{N}u(t)\Vert_{L_{\lambda}^{2}}^{2}.
\]

A consequence of the definition of the function spaces is the following
linear estimate; see \cite{ChristHolmerTataru2012} for the proof
in a different context.
\begin{lem}[Linear estimate]
\label{lem:UV-linear-estimate}Let $T\in(0,1]$, $\overline{s}=(s_{1},s_{2})$,
and let $u$ be a smooth solution to 
\[
\partial_{t}u+\mathcal{H}\partial_{xx}u=v\text{ on }(-T,T)\times\lambda\T.
\]
Then, we find the following estimate to hold: 
\[
\Vert u\Vert_{F_{\lambda}^{\overline{s}}(T)}\lesssim\Vert u\Vert_{E_{\lambda}^{\overline{s}}(T)}+\Vert v\Vert_{N_{\lambda}^{\overline{s}}(T)}.
\]
\end{lem}

\subsection{Short-time linear and bilinear Strichartz estimates}

Here we record short-time linear and bilinear Strichartz estimates.
\begin{prop}[Short-time linear Strichartz estimates]
\label{prop:ShorttimeLinearStrichartz}Let $p,q\in[2,\infty]$ be
such that $\frac{2}{q}+\frac{1}{p}=\frac{1}{2}$. Then, we find the
following estimate to hold:\footnote{Note that $\lambda N^{-1}$ is the maximal time scale, on which a
wave packet with frequency $N$ cannot distinguish the domains $\lambda\T$
and $\R$.} 
\[
\Vert P_{N}e^{-t\mathcal{H}\partial_{xx}}f\Vert_{L_{t}^{q}([0,\lambda N^{-1}],L_{\lambda}^{p})}\lesssim\Vert P_{N}f\Vert_{L_{\lambda}^{2}}.
\]
\end{prop}

The analogous linear estimates for the Schr\"odinger propagator $e^{\pm it\partial_{xx}}$
were proved in \cite{BurqGerardTzvetkov2004,StaffilaniTataru2002}.
Proposition \ref{prop:ShorttimeLinearStrichartz} for $\lambda=1$
follows after projecting to positive and negative frequencies due
to $P_{\pm}e^{-t\mathcal{H}\partial_{xx}}=e^{\pm it\partial_{xx}}P_{\pm}$,
where $P_{\pm}$ is the Fourier multiplier operator $\mathbf{1}_{\pm|\xi|\geq0}$.
The general case $\lambda\geq1$ follows from the scaling argument.
We omit the proof.
\begin{prop}[Short-time bilinear Strichartz estimates]
\label{prop:ShorttimeBilinearBenjaminOnoStrichartzEstimate}Let $r\geq\lambda^{-1}$
and $N\geq1$. Let $\eta_{1},\eta_{2}\in\R$ be such that $||\xi_{1}|-|\xi_{2}||\gtrsim N>0$
for $\xi_{i}\in B(\eta_{i},r)$. Then, we find the following estimate
to hold for $f_{i}\in L^{2}(\lambda\T)$ with $\text{supp}(\hat{f}_{i})\subseteq B(\eta_{i},r)$:
\begin{equation}
\Vert e^{-t\mathcal{H}\partial_{xx}}f_{1}e^{-t\mathcal{H}\partial_{xx}}f_{2}\Vert_{L_{t}^{2}([0,\lambda N^{-1}],L_{\lambda}^{2})}\lesssim N^{-1/2}\Vert f_{1}\Vert_{L_{\lambda}^{2}}\Vert f_{2}\Vert_{L_{\lambda}^{2}}.\label{eq:ShorttimeBilinearBenjaminOnoStrichartzEstimate}
\end{equation}
\end{prop}

\begin{rem}
The frequency separation in its \emph{magnitude} is required. 
\end{rem}

For the Schr\"odinger case, it is proved in \cite{Schippa2020ShriraEquation}
that Proposition \ref{prop:ShorttimeBilinearBenjaminOnoStrichartzEstimate}
holds under the (weaker) assumption $|\xi_{1}-\xi_{2}|\gtrsim N>0$
instead of $||\xi_{1}|-|\xi_{2}||\gtrsim N>0$. This is the transversality
assumption $|\partial_{\xi}h(\xi_{1})-\partial_{\xi}h(\xi_{2})|\gtrsim N$,
where $h(\xi)=|\xi|^{2}$ is the dispersion relation for the linear
Schr\"odinger flow. For the Benjamin-Ono case, this becomes $||\xi_{1}|-|\xi_{2}||\gtrsim N$
as the dispersion relation is $h(\xi)=|\xi|\xi$. We omit the proof
of Proposition \ref{prop:ShorttimeBilinearBenjaminOnoStrichartzEstimate}
and refer to \cite{Schippa2020ShriraEquation}\@.

For Schr\"odinger equations on compact manifolds, such short-time
bilinear estimates for dyadically separated frequencies were discussed
in \cite{MoyuaVega2008,Hani2012}.

In most cases, we apply Proposition \ref{prop:ShorttimeBilinearBenjaminOnoStrichartzEstimate}
for dyadic frequency interactions. Moreover, we will restrict the
time interval $[0,\lambda N^{-1}]$ to $[0,N^{-1}]$. One consequence
of this restriction is the gain when a low frequency wave interacts. 
\begin{cor}[Bilinear Strichartz for dyadic frequency interactions]
\label{cor:HighLowShorttimeEstimates}Let $N_{1}\gtrsim1$. Let $u_{i}(t,x)=[e^{-t\mathcal{H}\partial_{xx}}f_{i}](x)$
for $f_{i}\in L_{\lambda}^{2}$, where $i\in\{1,\dots,4\}$. 
\begin{itemize}
\item (High$\times$Low) Assume $N_{1}\gg N_{2}$. Then, we find the following
estimate to hold: 
\[
\Vert P_{N_{1}}u_{1}P_{N_{2}}u_{2}\Vert_{L_{t}^{2}([0,N_{1}^{-1}],L_{\lambda}^{2})}\lesssim N_{1}^{-1/2}N_{2}^{1/2,0}\Vert f_{1}\Vert_{L_{\lambda}^{2}}\Vert f_{2}\Vert_{L_{\lambda}^{2}}.
\]
\item (High$\times$High$\times$High$\to$Low) Assume $N_{1}\sim N_{2}\sim N_{3}\gg N_{4}$.
Then, we find the following estimate to hold: 
\[
\int_{0}^{N_{1}^{-1}}\int_{\lambda\T}\,\prod_{i=1}^{4}P_{N_{i}}u_{i}(s,x)\,dxds\lesssim N_{1}^{-1}N_{4}^{1/2,0}\prod_{i=1}^{4}\Vert f_{i}\Vert_{L_{\lambda}^{2}}.
\]
\end{itemize}
\end{cor}

\begin{proof}
The first estimate for $N_{2}\gtrsim1$ follows from Proposition \ref{prop:ShorttimeBilinearBenjaminOnoStrichartzEstimate}
with almost orthogonality. For $N_{2}\ll1$, we combine H\"older's
and Bernstein's inequality to find: 
\[
\begin{split} & \quad\Vert P_{N_{1}}u_{1}P_{N_{2}}u_{2}\Vert_{L_{t}^{2}([0,N_{1}^{-1}],L_{\lambda}^{2})}\\
 & \lesssim N_{1}^{-1/2}\Vert P_{N_{1}}u_{1}P_{N_{2}}u_{2}\Vert_{L_{t}^{\infty}([0,N_{1}^{-1}],L_{\lambda}^{2})}\\
 & \lesssim N_{1}^{-1/2}\Vert P_{N_{1}}u_{1}\Vert_{L_{t}^{\infty}([0,N_{1}^{-1}],L_{\lambda}^{2})}\Vert P_{N_{2}}u_{2}\Vert_{L_{t}^{\infty}([0,N_{1}^{-1}],L_{\lambda}^{\infty})}\\
 & \lesssim N_{1}^{-1/2}N_{2}^{1/2}\Vert f_{1}\Vert_{L_{\lambda}^{2}}\Vert f_{2}\Vert_{L_{\lambda}^{2}}.
\end{split}
\]

We turn to the second estimate. Let $c$ denote a small constant (depending
on implicit constants involved in $N_{1}\sim N_{2}\sim N_{3}\gg N_{4}$).
Let $I_{j}$ denote consecutive intervals of length $cN_{1}$ and
write $P_{N_{i}}=\sum_{j=1}^{M}P_{I_{j}}P_{N_{i}}$, where $M\lesssim1$
and $P_{I_{j}}$ is a (smooth) Littlewood-Paley projector onto the
frequency interval $I_{j}$. Due to $M\lesssim1$, it is enough to
show 
\begin{equation}
\begin{split} & \quad\int_{0}^{N_{1}^{-1}}\int_{\lambda\T}P_{I_{1}}u_{1}(s,x)P_{I_{2}}u_{2}(s,x)P_{I_{3}}u_{3}(s,x)P_{N_{4}}u_{4}(s,x)\,dxds\\
 & \lesssim N_{1}^{-1}N_{4}^{1/2,0}\prod_{i=1}^{4}\Vert f_{i}\Vert_{L_{\lambda}^{2}}
\end{split}
\label{eq:LocalizedTrilinearInteraction}
\end{equation}
because the claim follows from crudely summing over the $I_{i}$.

We claim that for any nontrivial interaction in \eqref{eq:LocalizedTrilinearInteraction},
there are $i,j\in\{1,2,3\}$, $i\neq j$, such that $||\xi_{i}|-|\xi_{j}||\gtrsim N$
for $\xi_{i}\in I_{i}$ and $\xi_{j}\in I_{j}$. At first, if ($I_{1}$
and $I_{2}$) or ($I_{1}$ and $-I_{2}$) are not neighboring intervals,
we find $||\xi_{1}|-|\xi_{2}||\gtrsim N$ for $\xi_{i}\in I_{i}$.
Now suppose that $I_{1}$ and $I_{2}$ are neighboring intervals.
Due to otherwise impossible frequency interaction, we find $||\xi_{2}|-|\xi_{3}||\gtrsim N$
for $\xi_{2}\in I_{2}$ and $\xi_{3}\in I_{3}$. The case when $I_{1}$
and $-I_{2}$ are neighboring intervals is impossible: it would imply
$|\xi_{1}+\xi_{2}|\leq2cN$, contradicting $|\xi_{3}+\xi_{4}|\gtrsim N$,
if $c$ was chosen sufficiently small.

Having settled the claim, we finish the proof. Say $(i,j)=(1,2)$
is a pair of the indices as in the claim. We apply a bilinear Strichartz
estimates (Proposition \ref{prop:ShorttimeBilinearBenjaminOnoStrichartzEstimate})
and the first estimate to find 
\begin{align*}
 & \quad\bigg|\int_{0}^{N_{1}^{-1}}\int_{\T}P_{I_{1}}u_{1}P_{I_{2}}u_{2}P_{I_{3}}u_{3}P_{N_{4}}u_{4}dxds\bigg|\\
 & \leq\Vert P_{I_{1}}u_{1}P_{I_{2}}u_{2}\Vert_{L_{t}^{2}([0,N_{1}^{-1}],L_{\lambda}^{2})}\Vert P_{I_{3}}u_{3}P_{N_{4}}u_{4}\Vert_{L_{t}^{2}([0,N_{1}^{-1}],L_{\lambda}^{2})}\\
 & \lesssim N_{1}^{-1/2}\Vert P_{I_{1}}f_{1}\Vert_{L_{\lambda}^{2}}\Vert P_{I_{2}}f_{2}\Vert_{L_{\lambda}^{2}}N_{1}^{-1/2}N_{4}^{1/2,0}\Vert P_{I_{3}}f_{3}\Vert_{L_{\lambda}^{2}}\|P_{N_{4}}f_{4}\|_{L_{\lambda}^{2}}\\
 & \lesssim N_{1}^{-1}N_{4}^{1/2,0}\prod_{i=1}^{4}\Vert P_{N_{i}}f_{i}\Vert_{L_{\lambda}^{2}}.
\end{align*}
This completes the proof. 
\end{proof}
We will not treat the case $N_{1}\lesssim1$ with short-time bilinear
Strichartz estimates because H\"older's and Bernstein's inequality
suffice.

We record the corresponding estimates for $U_{BO}^{p}$-/$V_{BO}^{p}$-functions. 
\begin{prop}
Let $N_{1}\gtrsim1$, $N_{1}\gg N_{2}$ and $|I|=N_{1}^{-1}$. Then,
we find the following estimates to hold: 
\begin{itemize}
\item (Short-time $U_{BO}^{2}$-estimate) 
\begin{equation}
\begin{split} & \quad\Vert P_{N_{1}}u_{1}P_{N_{2}}u_{2}\Vert_{L_{t}^{2}(I,L_{\lambda}^{2})}\\
 & \lesssim N_{1}^{-1/2}N_{2}^{1/2,0}\Vert P_{N_{1}}u_{1}\Vert_{U_{BO}^{2}(I)_{\lambda}}\Vert P_{N_{2}}u_{2}\Vert_{U_{BO}^{2}(I)_{\lambda}}.
\end{split}
\label{eq:ShorttimeU2Estimate}
\end{equation}
\item (Short-time $V_{BO}^{2}$-estimate) 
\begin{equation}
\begin{split} & \quad\Vert P_{N_{1}}u_{1}P_{N_{2}}u_{2}\Vert_{L_{t}^{2}(I,L_{\lambda}^{2})}\\
 & \lesssim\log\langle\frac{N_{1}}{N_{2}}\rangle N_{1}^{-1/2}N_{2}^{1/2,0}\Vert P_{N_{1}}u_{1}\Vert_{V_{BO}^{2}(I)_{\lambda}}\Vert P_{N_{2}}u_{2}\Vert_{V_{BO}^{2}(I)_{\lambda}}.
\end{split}
\label{eq:ShorttimeV2Estimate}
\end{equation}
\item (Linear $L_{t,x}^{6}$ and $L_{t}^{8}L_{x}^{4}$ estimates) 
\begin{equation}
\Vert P_{N_{2}}u\Vert_{L_{t}^{6}(I,L_{\lambda}^{6})}+\Vert P_{N_{2}}u\Vert_{L_{t}^{8}(I,L_{\lambda}^{4})}\lesssim\Vert P_{N_{2}}u\Vert_{V_{BO}^{2}(I)_{\lambda}}.\label{eq:L6StrichartzEstimateV2}
\end{equation}
\end{itemize}
\end{prop}

\begin{proof}
\eqref{eq:ShorttimeU2Estimate} is an immediate consequence of the
atomic representation of $U^{p}$-\-func\-tions; see \cite[Section 2]{HadacHerrKoch2009}
for details. In a similar spirit, \eqref{eq:L6StrichartzEstimateV2}
follows from $\Vert P_{N_{2}}u\Vert_{L_{t}^{6}(I,L_{\lambda}^{6})}\lesssim\Vert P_{N_{2}}u\Vert_{U_{BO}^{6}(I)_{\lambda}}\lesssim\Vert P_{N_{2}}u\Vert_{V_{BO}^{2}(I)_{\lambda}}$
and likewise for the $L_{t}^{8}L_{x}^{4}$-norm. For the proof of
\eqref{eq:ShorttimeV2Estimate} we note that 
\[
\begin{split}\Vert P_{N_{1}}u_{1}P_{N_{2}}u_{2}\Vert_{L_{t}^{2}(I,L_{\lambda}^{2})} & \leq\Vert P_{N_{1}}u_{1}\Vert_{L_{t}^{4}(I,L_{\lambda}^{4})}\Vert P_{N_{2}}u_{2}\Vert_{L_{t}^{4}(I,L_{\lambda}^{4})}\\
 & \lesssim N_{1}^{-1/4}\Vert P_{N_{1}}u_{1}\Vert_{L_{t}^{8}(I,L_{\lambda}^{4})}\Vert P_{N_{2}}u_{2}\Vert_{L_{t}^{8}(I,L_{\lambda}^{4})}\\
 & \lesssim N_{1}^{-1/4}\Vert P_{N_{1}}u_{1}\Vert_{U_{BO}^{8}(I)_{\lambda}}\Vert P_{N_{2}}u_{2}\Vert_{U_{BO}^{8}(I)_{\lambda}}\\
 & \lesssim N_{1}^{-1/4}\Vert P_{N_{1}}u_{1}\Vert_{U_{BO}^{8}(I)_{\lambda}}\Vert P_{N_{2}}u_{2}\Vert_{U_{BO}^{8}(I)_{\lambda}},
\end{split}
\]
by H\"older's inequality, Proposition \ref{prop:ShorttimeBilinearBenjaminOnoStrichartzEstimate},
and the transfer principle. Interpolating (cf. \cite[Proposition~2.20,~p.~930]{HadacHerrKoch2009})
with \eqref{eq:ShorttimeU2Estimate} yields \eqref{eq:ShorttimeV2Estimate}. 
\end{proof}

\section{\label{sec:ProofMainResults}Proof of main results}

In this section we show how to conclude Theorems \ref{thm:APrioriGBOI}
- \ref{thm:LWPGBO} with short-time nonlinear and energy estimates
at hand. Many of the below arguments are standard in the literature,
where short-time Fourier restriction is used (cf. \cite{IonescuKenigTataru2008}).

We note that existence of solutions for (rough) $H^{s}(\T)$-data
follows from the smoothing effect in the energy estimate for smooth
solutions and a compactness argument. We refer to the literature \cite{GuoOh2018}
for details, and focus on a priori estimates for smooth solutions
subsequently.

\emph{Local-in-time a priori estimates and existence for small initial
data in} $H^{s}(\T)$, $s>1/2$:

Let $u\in C_{t,x}^{\infty}([0,T_{+}),\T)$ be a smooth solution to
\eqref{eq:GeneralizedBenjaminOnoEquation} with maximal forward-in-time
lifespan $[0,T_{+})$ and $\|u_{0}\|_{H^{s}}\leq\epsilon$ for sufficiently
small $\epsilon>0$. Note that the existence and uniqueness of smooth
solutions are ensured by the classical energy method (cf. \cite{AbdelouhabBonaFellandSaut1989}).
Gathering the linear short-time energy estimate (Lemma \ref{lem:UV-linear-estimate}),
the nonlinear short-time estimate (Proposition \ref{prop:ShorttimeNonlinearEstimate}),
and the energy estimate for solutions (Proposition \ref{prop:EnergyEstimateSolutions})
for $\lambda=1$, we find the set of estimates \eqref{eq:a-priori-intro}
for $T\in(0,1]\cap(0,T_{+})$.

Set $X(T)=\|\partial_{x}(u^{k})\|_{N^{s}(T)}+\|u\|_{E^{s}(T)}$ for
$T\in(0,1]\cap(0,T_{+})$. Thus \eqref{eq:a-priori-intro} reads 
\[
X(T)\lesssim\Vert u_{0}\Vert_{H^{s}}+X(T)^{\frac{k+1}{2}}+X(T)^{k}.
\]
Since the function $T\mapsto X(T)$ is continuous, non-decreasing,
and satisfies $\lim_{T\to0}X(T)\lesssim\Vert u_{0}\Vert_{H^{s}}$
(cf. \cite[Section~1]{KochTataru2007}), a continuity argument gives
\[
X(T)\lesssim\Vert u_{0}\Vert_{H^{s}}\lesssim\epsilon
\]
uniformly for $T\in(0,1]\cap(0,T_{+})$, provided that $\epsilon$
is sufficiently small.

Note the following persistence property: For $1/2<s\leq s^{\prime}$
we find the following estimates in addition: (see for instance Proposition
\ref{prop:EnergyEstimateSolutions})
\begin{equation}
\left\{ \begin{array}{cl}
\Vert u\Vert_{F^{s^{\prime}}(T)} & \lesssim\Vert u\Vert_{E^{s^{\prime}}(T)}+\Vert\partial_{x}(u^{k})\Vert_{N^{s^{\prime}}(T)},\\
\Vert\partial_{x}(u^{k})\Vert_{N^{s^{\prime}}(T)} & \lesssim\Vert u\Vert_{F^{s^{\prime}}(T)}\Vert u\Vert_{F^{s}(T)}^{k-1},\\
\Vert u\Vert_{E^{s^{\prime}}(T)}^{2} & \lesssim\Vert u_{0}\Vert_{H^{s^{\prime}}}^{2}+\Vert u\Vert_{F^{s^{\prime}}(T)}^{2}\Vert u\Vert_{F^{s}(T)}^{k-1}+\Vert u\Vert_{F^{s^{\prime}}(T)}^{2}\Vert u\Vert_{F^{s}(T)}^{2k-2}.
\end{array}\right.\label{eq:a-priori-persistence}
\end{equation}
This gives 
\[
\Vert u\Vert_{F^{s^{\prime}}(T)}\lesssim\Vert u_{0}\Vert_{H^{s^{\prime}}}+\Vert u\Vert_{F^{s^{\prime}}(T)}\Vert u\Vert_{F^{s}(T)}^{\frac{k-1}{2}}+\Vert u\Vert_{F^{s^{\prime}}(T)}\Vert u\Vert_{F^{s}(T)}^{k-1}.
\]
Since we know that $\Vert u\Vert_{F^{s}(T)}\lesssim X(T)\lesssim\varepsilon$
is small, we find 
\[
\Vert u\Vert_{F^{s^{\prime}}(T)}\lesssim\Vert u_{0}\Vert_{H^{s^{\prime}}},
\]
uniformly for $T\in(0,1]\cap(0,T_{+})$. Recall the blow-up alternative
(cf. \cite{AbdelouhabBonaFellandSaut1989}): either $\lim_{t\uparrow T_{+}}\Vert u(t)\Vert_{H^{3}}=\infty$,
or $u$ can be extended beyond $T_{+}$. This points out that $T_{+}>1$
so we can take $T=1$ to get $\|u\|_{F^{s}(1)}\lesssim\|u_{0}\|_{H^{s}}$.

\emph{Local-in-time a priori estimates and existence of solutions
for arbitrary initial data in } $H^{s}(\T)$, $s>1/2$:

To extend the previous claims to large data $u_{0}\in H^{s}(\T)$,
we rescale: $u_{0,\lambda}(x) \rightarrow \lambda^{-\frac{1}{k-1}}u_{0}(\lambda^{-1}x)$.
Let $s_{c}<s_{1}<1/2$ and set $\overline{s}=(s_{1},s)$. For sufficiently
small $\epsilon>0$, we choose $\lambda=\lambda(\|u_{0}\|_{H^{s}},\epsilon)\geq1$
such that $\Vert u_{0,\lambda}\Vert_{H_{\lambda}^{\overline{s}}}\leq\varepsilon$
by \eqref{eq:scaling}. For the corresponding solutions $u_{\lambda}$
on $\lambda\T$, we find 
\begin{equation}
\left\{ \begin{array}{cl}
\Vert u_{\lambda}\Vert_{F_{\lambda}^{\overline{s}}(T)} & \lesssim\Vert u_{\lambda}\Vert_{E_{\lambda}^{\overline{s}}(T)}+\Vert\partial_{x}(u_{\lambda}^{k})\Vert_{N_{\lambda}^{\overline{s}}(T)},\\
\Vert\partial_{x}(u_{\lambda}^{k})\Vert_{N_{\lambda}^{\overline{s}}(T)} & \lesssim\Vert u_{\lambda}\Vert_{F_{\lambda}^{\overline{s}}(T)}^{k},\\
\Vert u\Vert_{E_{\lambda}^{\overline{s}}(T)}^{2} & \lesssim\Vert u_{0}\Vert_{H_{\lambda}^{\overline{s}}}^{2}+\Vert u\Vert_{F_{\lambda}^{\overline{s}}(T)}^{k+1}+\Vert u\Vert_{F_{\lambda}^{\overline{s}}(T)}^{2k}.
\end{array}\right.\label{eq:a-priori-rescaled}
\end{equation}
We recall that the implicit constants can be chosen independently
of $\lambda$. The reason is that the underlying Strichartz estimates
and pointwise estimates are independent of $\lambda$. This yields
like above a priori estimates for $u_{\lambda}$ on $[0,1]$ 
\[
\sup_{t\in[0,1]}\|u_{\lambda}(t)\|_{H_{\lambda}^{\overline{s}}}\lesssim\|u_{0,\lambda}\|_{H_{\lambda}^{\overline{s}}}
\]
with persistence of regularity. Scaling back using \eqref{eq:scaling}
gives
\[
\sup_{t\in[0,\lambda^{-2}]}\Vert u(t)\Vert_{H^{s}}\lesssim\lambda^{s-s_{1}}\Vert u_{0}\Vert_{H^{s}}.
\]
This proves Theorem \ref{thm:APrioriGBOI}.

\emph{Improved local-in-time a priori estimates and existence of solutions
in the quartic case}:

Here, it suffices to show that we can choose $T$ depending only on
$\|u_{0}\|_{H^{1/2}}$. Let $s_{c}<s_{1}<1/2$ and set $\overline{s}=(s_{1},1/2)$.
Choose $\lambda=\lambda(\|u_{0}\|_{H^{1/2}},\epsilon)\geq1$ such
that $\|u_{0,\lambda}\|_{H_{\lambda}^{\overline{s}}}\leq\epsilon$.
For the quartic case, we have \eqref{eq:a-priori-rescaled} uniformly
in $\lambda$, even for $\overline{s}=(s_{1},1/2)$. Together with
persistence of regularity and the blow-up alternative, we have a priori
estimates for $u_{\lambda}$ on $[0,1]$,
\[
\sup_{t\in[0,1]}\|u_{\lambda}(t)\|_{H_{\lambda}^{s_{1},s}}\lesssim\|u_{0,\lambda}\|_{H_{\lambda}^{s_{1},s}}
\]
for any $s>1/2$. Scaling back, we have 
\[
\sup_{t\in[0,\lambda^{-2}]}\Vert u(t)\Vert_{H^{s}}\lesssim\lambda^{s-s_{1}}\Vert u_{0}\Vert_{H^{s}}.
\]
Since $\lambda$ only depends on $\|u_{0}\|_{H^{1/2}}$, the desired
local-in-time a priori estimate is proved.

\emph{Global existence for $H^{s}(\T)$-solutions, $s>1/2$, with
small $H^{1/2}(\T)$ initial data, in the quartic case}:

Again, we focus on a priori estimates for smooth solutions $u\in C_{t,x}^{\infty}([0,T_{+}),\T)$
with $\|u_{0}\|_{H^{1/2}}\leq\epsilon$ for sufficiently small $\epsilon>0$.
Recall that the energy and mass 
\[
E[u]=\int_{\T}\frac{u\mathcal{H}\partial_{x}u}{2}dx+\int_{\T}\frac{u^{5}}{5}dx,\quad M[u]=\int_{\T}u^{2}dx
\]
are conserved quantities. The potential energy is small relative to
the sum of kinetic energy and mass: 
\[
\Big|\int_{\T}\frac{u^{5}}{5}dx\Big|\lesssim\|u\|_{H^{1/2}}^{5}\ll\|u\|_{H^{1/2}}^{2}\sim\int_{\T}\frac{u\mathcal{H}\partial_{x}u}{2}dx+\int u^{2}dx,
\]
provided that $\|u\|_{H^{1/2}}$ is small. Therefore, the conservation
of mass and energy yields $\sup_{t\in[0,T_{+})}\|u(t)\|_{H^{1/2}}\lesssim\epsilon$,
provided that $\epsilon$ is sufficiently small. Together with persistence
of regularity and the blow-up alternative, we find $T_{+}>1$ and
$\sup_{t\in[0,1]}\Vert u(t)\Vert_{H^{s'}}\lesssim\Vert u_{0}\Vert_{H^{s'}}$
for any $s'>1/2$. However, the $H^{1/2}$-norm of $u(t)$ is uniformly
bounded, the argument can be iterated. Thus $T_{+}=\infty$ and iterating
$\Vert u(n+1)\Vert_{H^{s'}}\lesssim\Vert u(n)\Vert_{H^{s'}}$ gives
\[
\Vert u(t)\Vert_{H^{s'}}\lesssim e^{Ct}\Vert u_{0}\Vert_{H^{s'}}
\]
for any $s'>1/2$. This yields global existence in the quartic case.
This concludes the proof of Theorem \ref{thm:APrioriGBOII}.

\vspace{0.5cm}
We turn to the proof of Theorem \ref{thm:LWPGBO}. We restrict to
the lowest regularity $s=3/4$.


\emph{Lipschitz continuous dependence in} $H^{-1/4}$:

For small initial data in $H^{3/4}$, say $\Vert u_{i}(0)\Vert_{H^{3/4}}\leq\varepsilon\ll1$
for $i=1,2$, we can proceed as follows. Let $v=u_{1}-u_{2}$ denote
the difference of solutions $u_{i}$ to \eqref{eq:GeneralizedBenjaminOnoEquation}
in $F^{3/4}(1)$. $v$ is governed by (cf. \eqref{eq:first-expression})
\[
\partial_{t}v+\mathcal{H}\partial_{xx}v=\partial_{x}(v(u_{1}^{k-1}+u_{1}^{k-2}u_{2}+\ldots+u_{2}^{k-1})).
\]
Lemma \ref{lem:UV-linear-estimate}, Proposition \ref{prop:ShorttimeNonlinearEstimate},
and Proposition \ref{prop:EnergyEstimatesDifferences} yield the set
of estimates \eqref{eq:weak-Lipschitz-intro}. Due to smallness of
$\|u_{i}\|_{F^{3/4}(1)}\lesssim\epsilon$ (by the a priori estimates
for solutions), the set of estimates \eqref{eq:weak-Lipschitz-intro}
yields 
\[
\Vert v\Vert_{F^{-1/4}(1)}\lesssim\Vert v(0)\Vert_{H^{-1/4}}.
\]

The previous argument can be extended to the large data case paralleling
the arguments for the a priori estimates. Let $s_{c}<s_{1}<1/2$ and
set $\overline{s}=(s_{1},3/4)$. Choose $\lambda$ sufficiently large
such that $\|u_{i,\lambda}(0)\|_{H_{\lambda}^{\overline{s}}}\leq\epsilon\ll1$.
We find the following set of estimates independent of $\lambda$:\footnote{$T=1$ is omitted in the estimate for brevity.}
\[
\left\{ \begin{array}{cl}
\Vert v_{\lambda}\Vert_{F_{\lambda}^{\overline{s}-1}} & \lesssim\Vert v_{\lambda}\Vert_{E_{\lambda}^{\overline{s}-1}}+\Vert\partial_{x}(v_{\lambda}u_{\lambda}^{k-1})\Vert_{N_{\lambda}^{\overline{s}-1}}\\
\Vert\partial_{x}(v_{\lambda}u_{\lambda}^{k-1})\Vert_{N_{\lambda}^{\overline{s}-1}} & \lesssim\Vert v_{\lambda}\Vert_{F_{\lambda}^{\overline{s}-1}}(\Vert u_{1,\lambda}\Vert_{F_{\lambda}^{\overline{s}}}+\Vert u_{2,\lambda}\Vert_{F_{\lambda}^{\overline{s}}})^{k-1},\\
\Vert v_{\lambda}\Vert_{E_{\lambda}^{\overline{s}-1}}^{2} & \lesssim\Vert v_{\lambda}(0)\Vert_{H_{\lambda}^{\overline{s}-1}}^{2}+\Vert v_{\lambda}\Vert_{F_{\lambda}^{\overline{s}-1}}^{2}(\Vert u_{1,\lambda}\Vert_{F_{\lambda}^{\overline{s}}}+\Vert u_{2,\lambda}\Vert_{F_{\lambda}^{\overline{s}}})^{2k-2}.
\end{array}\right.
\]
By the a priori estimates for solutions, $\|u_{i,\lambda}\|_{F_{\lambda}^{\overline{s}}(1)}\lesssim\epsilon$.
This smallness yields 
\[
\sup_{t\in[0,1]}\Vert v_{\lambda}(t)\Vert_{H_{\lambda}^{\overline{s}-1}}\lesssim\|v_{\lambda}\|_{F_{\lambda}^{\overline{s}-1}(1)}\lesssim\Vert v_{\lambda}(0)\Vert_{H_{\lambda}^{\overline{s}-1}}.
\]
Scaling back yields 
\[
\sup_{t\in[0,\lambda^{-2}]}\Vert v(t)\Vert_{H^{-1/4}}\lesssim\lambda^{s_{1}+\frac{1}{4}}\Vert v(0)\Vert_{H^{-1/4}}.
\]
Since $\lambda$ only depends on the $H^{3/4}$-norms of initial data,
this yields Lipschitz continuous dependence in $H^{-1/4}$.

\emph{Continuous dependence in }$H^{3/4}$:

Lastly, we prove continuous dependence with the Bona--Smith approximation.
Let $s_{c}<s_{1}<1/2$ and set $\overline{s}=(s_{1},3/4)$. Lemma
\ref{lem:UV-linear-estimate}, Proposition \ref{prop:ShorttimeNonlinearEstimate},
and Proposition \ref{prop:EnergyEstimatesDifferences} read 
\[
\left\{ \begin{array}{cl}
\Vert v\Vert_{F_{\lambda}^{\overline{s}}} & \lesssim\Vert\partial_{x}(vu^{k-1})\Vert_{N_{\lambda}^{\overline{s}}}+\Vert v\Vert_{E_{\lambda}^{\overline{s}}},\\
\Vert\partial_{x}(vu^{k-1})\Vert_{N_{\lambda}^{\overline{s}}} & \lesssim\Vert v\Vert_{F_{\lambda}^{\overline{s}}}(\Vert u_{1}\Vert_{F_{\lambda}^{\overline{s}}}+\Vert u_{2}\Vert_{F_{\lambda}^{\overline{s}}})^{k-1},\\
\Vert v\Vert_{E_{\lambda}^{\overline{s}}}^{2} & \lesssim\Vert v(0)\Vert_{H_{\lambda}^{\overline{s}}}^{2}+\Vert v\Vert_{F_{\lambda}^{\overline{s}}}^{2}(\Vert v\Vert_{F_{\lambda}^{\overline{s}}}+\Vert u_{2}\Vert_{F_{\lambda}^{\overline{s}}})^{k-1}\\
 & \quad+\Vert v\Vert_{F_{\lambda}^{\overline{s}}}\Vert v\Vert_{F_{\lambda}^{\overline{s}-1}}\Vert u_{2}\Vert_{F_{\lambda}^{\overline{s}+1}}(\Vert v\Vert_{F_{\lambda}^{\overline{s}}}+\Vert u_{2}\Vert_{F_{\lambda}^{\overline{s}}})^{k-2}\\
 & \quad+\Vert v\Vert_{F_{\lambda}^{\overline{s}}}^{2}(\Vert v\Vert_{F_{\lambda}^{\overline{s}}}+\Vert u_{2}\Vert_{F_{\lambda}^{\overline{s}}})^{2k-2}\\
 & \quad+\Vert v\Vert_{F_{\lambda}^{\overline{s}}}\Vert v\Vert_{F_{\lambda}^{\overline{s}-1}}\Vert u_{2}\Vert_{F_{\lambda}^{\overline{s}+1}}(\Vert v\Vert_{F_{\lambda}^{\overline{s}}}+\Vert u_{2}\Vert_{F_{\lambda}^{\overline{s}}})^{2k-3}.
\end{array}\right.
\]

Consider a $H_{\lambda}^{\overline{s}}$-Cauchy sequence of smooth
initial data $u_{0}^{n}\in C^{\infty}(\lambda\T)$ with $\|u_{0}^{n}\|_{H_{\lambda}^{\overline{s}}}<\epsilon$.
Note that this smallness can be assumed thanks to the scaling argument.
Due to the smallness of $H_{\lambda}^{\overline{s}}$-norm, we can
define $S_{1}^{\infty}(u_{0}^{n})$ by the classical solution to \eqref{eq:GeneralizedBenjaminOnoEquation}
with initial data $u_{0}^{n}$, on the time interval $[0,1]$. By
the density argument, it suffices to show that $S_{1}^{\infty}(u_{0}^{n})$
is a Cauchy sequence in $C([0,1],H_{\lambda}^{\overline{s}})$.

To show this, we start from writing 
\begin{equation}
\begin{split}S_{1}^{\infty}(u_{0}^{n})-S_{1}^{\infty}(u_{0}^{m}) & =(S_{1}^{\infty}(u_{0}^{n})-S_{1}^{\infty}(u_{0,\leq N}^{n}))+(S_{1}^{\infty}(u_{0,\leq N}^{n})-S_{1}^{\infty}(u_{0,\leq N}^{m}))\\
 & +(S_{1}^{\infty}(u_{0,\leq N}^{m})-S_{1}^{\infty}(u_{0}^{m})),
\end{split}
\label{eq:DecompositionDifference}
\end{equation}
where $f_{\leq N}$ denotes $P_{\leq N}f$ and $N>1$ will be chosen
large. For the first (and third) term of \eqref{eq:DecompositionDifference},
temporarily denoting $v=S_{1}^{\infty}(u_{0}^{n})-S_{1}^{\infty}(u_{0,\leq N}^{n})$,
we use the above set of estimates to get 
\[
\Vert v\Vert_{F_{\lambda}^{\overline{s}}(1)}^{2}\lesssim\Vert u_{0,>N}^{n}\Vert_{H_{\lambda}^{\overline{s}}}^{2}+\Vert v\Vert_{F_{\lambda}^{\overline{s}}(1)}^{2}\epsilon^{k-1}+\|v\|_{F_{\lambda}^{\overline{s}}(1)}\|v\|_{F_{\lambda}^{\overline{s}-1}(1)}N\epsilon^{k-1},
\]
where the factor $N$ comes from the persistence estimate 
\[
\|S_{1}^{\infty}(u_{0,\leq N}^{n})\|_{F_{\lambda}^{\overline{s}+1}(1)}\lesssim\|u_{0,\leq N}^{n}\|_{H_{\lambda}^{\overline{s}+1}}\lesssim N\|u_{0,\leq N}^{n}\|_{H_{\lambda}^{\overline{s}}}\lesssim N\epsilon.
\]
The above loss of $N$ is balanced by the Lipschitz dependence in
$H_{\lambda}^{\overline{s}-1}$: 
\[
\|v\|_{F_{\lambda}^{\overline{s}-1}(1)}\lesssim\|u_{0,>N}^{n}\|_{H_{\lambda}^{\overline{s}-1}}\lesssim N^{-1}\|u_{0,>N}^{n}\|_{H_{\lambda}^{s}}.
\]
Therefore, smallness of $\epsilon$ implies 
\[
\|S_{1}^{\infty}(u_{0}^{n})-S_{1}^{\infty}(u_{0,\leq N}^{n})\|_{F_{\lambda}^{\overline{s}}(1)}=\Vert v\Vert_{F_{\lambda}^{\overline{s}}(1)}\lesssim\Vert u_{0,>N}^{n}\Vert_{H_{\lambda}^{\overline{s}}},
\]
uniformly in $N$ and $n$. As $u_{0}^{n}$ is a Cauchy sequence in
$H_{\lambda}^{\overline{s}}$, we have 
\[
\sup_{n}\Vert S_{1}^{\infty}(u_{0,\leq N}^{n})-S_{1}^{\infty}(u_{0}^{n})\Vert_{C([0,1],H_{\lambda}^{\overline{s}})}\to0\text{ as }N\to\infty.
\]
The second term in \eqref{eq:DecompositionDifference} is bounded
by the standard energy estimate\footnote{In the following estimate, $C=C(k)$ varies line by line.}
\[
\begin{split}\Vert S_{1}^{\infty}(u_{0,\leq N}^{n})-S_{1}^{\infty}(u_{0,\leq N}^{m})\Vert_{C([0,1],H_{\lambda}^{\bar{s}})} & \lesssim\Vert S_{1}^{\infty}(u_{0,\leq N}^{n})-S_{1}^{\infty}(u_{0,\leq N}^{m})\Vert_{C([0,1],H_{\lambda}^{1})}\\
 & \lesssim N^{C}\Vert P_{\leq N}(u_{0}^{n}-u_{0}^{m})\Vert_{H_{\lambda}^{1}}\\
 & \lesssim N^{C}\lambda^{s_{1}}\Vert u_{0}^{n}-u_{0}^{m}\Vert_{H_{\lambda}^{\overline{s}}}.
\end{split}
\]
Thus, for any $N$, this term converges to zero. This concludes that
$S_{1}^{\infty}(u_{0}^{n})$ is a Cauchy sequence in $C([0,1],H_{\lambda}^{\overline{s}})$.
The proof of Theorem \ref{thm:LWPGBO} is completed.

\section{\label{sec:NonlinearEstimates}Nonlinear estimates}

The main purpose of this section is to propagate the nonlinearity
in short-time function spaces:
\begin{prop}[Nonlinear estimates]
\label{prop:ShorttimeNonlinearEstimate}Let $T\in(0,1]$. Then, there
is $s_{k}<1/2$ such that we find the following estimate to hold:
\begin{equation}
\Vert\partial_{x}(\prod_{i=1}^{k}u_{i})\Vert_{N_{\lambda}^{\overline{s}}(T)}\lesssim\prod_{i=1}^{k}\Vert u_{i}\Vert_{F_{\lambda}^{\overline{s}}(T)},\label{eq:ShorttimeNonlinearEstimate}
\end{equation}
provided that $0<s_{1}<1/2$ and $s_{2}>s_{k}$, and we denoted $\overline{s}=(s_{1},s_{2})$.
\end{prop}

In Section \ref{subsec:short-time-DU2-estimate}, we prove Proposition
\ref{prop:ShorttimeNonlinearEstimate}. In Section \ref{subsec:L1L2-estimate},
we estimate the nonlinearity in the smaller space $L_{t}^{1}L_{x}^{2}$;
see Lemma \ref{lem:shorttime-L1L2}. This aids in simplifying the
energy estimates in Sections \ref{sec:EnergyEstimateSolutions} and
\ref{sec:DifferenceEstimate}.

\subsection{\label{subsec:short-time-DU2-estimate}Short-time $DU_{BO}^{2}$-estimate}

In this subsection, we prove Proposition \ref{prop:ShorttimeNonlinearEstimate}.
After localizing the frequencies of the $u_{i}$ and the output frequency,
we reduce the proof of Proposition \ref{prop:ShorttimeNonlinearEstimate}
to estimates of the following kind: 
\begin{equation}
\begin{split}\Vert P_{N}\partial_{x}(P_{N_{1}}u_{1}\ldots P_{N_{k}}u_{k})\Vert_{DU_{BO}^{2}(I)_{\lambda}} & \lesssim C(N,N_{1},\ldots,N_{k})\\
 & \qquad\prod_{i=1}^{k}\sup_{\substack{I_{i}\subseteq[0,T],\\
|I_{i}|=N_{i}^{-1}
}
}\Vert P_{N_{i}}u_{i}\Vert_{U_{BO}^{2}(I_{i})_{\lambda}},
\end{split}
\label{eq:FrequencyLocalizedShorttimeNonlinearEstimate}
\end{equation}
where $I\subseteq[0,T]$, $|I|=N^{-1}$. $C(N,N_{1},\dots,N_{k})$
is a constant allowing for dyadic summation after adding frequency
weights giving \eqref{eq:ShorttimeNonlinearEstimate}.

By symmetry and otherwise impossible frequency interaction, we can
suppose that $N_{1}\geq N_{2}\geq\ldots\geq N_{k}$ and $N_{1}\sim N_{2}$.
We give an overview of the arising frequency interactions. 
\begin{itemize}
\item Firstly, suppose that $N\sim N_{1}$ with $N_{1}\gtrsim1$. 
\begin{itemize}
\item If $N_{2}\ll N_{1}$, we apply H\"older in time and a bilinear Strichartz
estimate to ameliorate the derivative loss. The remaining factors
are estimated with pointwise bounds. In the following cases, we do
not mention the application of pointwise bounds to the remaining factors. 
\item If $N_{1}\sim N_{2}\gg N_{3}$, we can apply two bilinear Strichartz
estimates after dualization. 
\item If $N_{1}\sim N_{2}\sim N_{3}\gtrsim N_{4}$, we use H\"older in
time and linear Strichartz estimates on the high frequencies to estimate
the interaction. 
\end{itemize}
\item Secondly, suppose that $N\ll N_{1}$ with $N_{1}\gtrsim1$. Due to
otherwise impossible frequency interaction, we can suppose that $N_{1}\sim N_{2}$. 
\begin{itemize}
\item If $N_{1}\sim N_{2}\gg N_{3}$, we apply two bilinear Strichartz estimates
after making use of duality.
\item If $N_{1}\sim N_{2}\sim N_{3}\gg N_{4}$, we still apply two bilinear
Strichartz estimates as observed in Corollary \ref{cor:HighLowShorttimeEstimates}. 
\item If $N_{1}\sim N_{2}\sim N_{3}\sim N_{4}$, we use linear Strichartz
estimates and H\"older's inequality to estimate the interaction. 
\end{itemize}
\item Finally, suppose that $N_{1}\ll1$. We merely use H\"older and Bernstein's
inequalities. 
\end{itemize}
\begin{lem}[Short-time $DU_{BO}^{2}$-estimate]
\label{lem:short-time-DU2-estimate} Let $T\in(0,1]$ and $N_{1}\gtrsim\ldots\gtrsim N_{k}$.
The constant $C(N,N_{1},\dots,N_{k})$ can be chosen as follows: 
\begin{itemize}
\item \emph{Case A:} $N\sim N_{1}\gtrsim1$. 
\[
C(N,N_{1},\dots,N_{k})\lesssim\begin{cases}
N_{2}^{1/2,0}\prod_{i=3}^{k}N_{i}^{1/2} & \text{if }N_{1}\gg N_{2},\\
\log\langle N\rangle N_{3}^{1/2,0}\prod_{i=4}^{k}N_{i}^{1/2} & \text{if }N_{1}\sim N_{2}\gg N_{3},\\
\prod_{i=3}^{k}N_{i}^{1/2} & \text{if }N_{1}\sim N_{2}\sim N_{3}.
\end{cases}
\]
\item \emph{Case B:} $N\ll N_{1}\sim N_{2}$ with $N_{1}\gtrsim1$. 
\[
C(N,N_{1},\dots,N_{k})\lesssim N^{3/2,0}\begin{cases}
\log\langle N_{1}\rangle N_{3}^{1/2,0}\prod_{i=4}^{k}N_{i}^{1/2} & \text{if }N_{1}\gg N_{3},\\
\log\langle N_{1}\rangle\prod_{i=4}^{k}N_{i}^{1/2} & \text{if }N_{1}\sim N_{3}\gg N_{4},\\
N^{0,1/2}\prod_{i=4}^{k}N_{i}^{1/2} & \text{if }N_{1}\sim N_{4}.
\end{cases}
\]
\item \emph{Case C:} $N_{1}\lesssim1$. 
\[
C(N,N_{1},\ldots,N_{k})\lesssim N^{1/2}\prod_{i=1}^{k}N_{i}^{1/2}.
\]
\end{itemize}
\end{lem}

\begin{proof}
\textbf{Case A:} $N\sim N_{1}\gtrsim1$.

\underline{Subcase I:} $N_{1}\gg N_{2}$. We claim that \eqref{eq:FrequencyLocalizedShorttimeNonlinearEstimate}
holds with $C(N,N_{1},\ldots,N_{k})=N_{2}^{1/2,0}\prod_{i=3}^{k}N_{i}^{1/2}$.
We use the embedding $L^{1}(I,L_{\lambda}^{2})\hookrightarrow DU_{BO}^{2}(I)_{\lambda}$,
H\"older in time, bilinear Strichartz \eqref{eq:ShorttimeU2Estimate},
and pointwise bounds to find 
\[
\begin{split} & \quad\Vert P_{N}\partial_{x}(P_{N_{1}}u_{1}\ldots P_{N_{k}}u_{k})\Vert_{DU_{BO}^{2}(I)_{\lambda}}\\
 & \lesssim N\Vert P_{N}(P_{N_{1}}u_{1}\ldots P_{N_{k}}u_{k})\Vert_{L^{1}(I,L_{\lambda}^{2})}\\
 & \lesssim N^{1/2}\Vert P_{N_{1}}u_{1}\ldots P_{N_{k}}u_{k}\Vert_{L_{t}^{2}(I,L_{\lambda}^{2})}\\
 & \lesssim N^{1/2}\Vert P_{N_{1}}u_{1}P_{N_{2}}u_{2}\Vert_{L_{t}^{2}(I,L_{\lambda}^{2})}\prod_{i=3}^{k}\Vert P_{N_{i}}u_{i}\Vert_{L_{t}^{\infty}(I,L_{\lambda}^{\infty})}\\
 & \lesssim N_{2}^{1/2,0}\Big(\prod_{i=3}^{k}N_{i}^{1/2}\Big)\prod_{i=1}^{k}\sup_{\substack{I_{i}\subseteq[0,T],\\
|I_{i}|=N_{i}^{-1}
}
}\Vert P_{N}u_{i}\Vert_{U_{BO}^{2}(I_{i})_{\lambda}}.
\end{split}
\]

\underline{Subcase II:} $N_{1}\sim N_{2}\gg N_{3}$. We claim that
\eqref{eq:FrequencyLocalizedShorttimeNonlinearEstimate} holds with
$C(N,N_{1},\ldots,N_{k})=\log\langle N\rangle N_{3}^{1/2,0}\prod_{i=4}^{k}N_{i}^{1/2}$.
We use duality (cf. \cite[Proposition~2.10]{HadacHerrKoch2009}) to
write 
\begin{align}
 & \quad\Vert\partial_{x}(\prod_{i=1}^{k}P_{N_{i}}u_{i})\Vert_{DU_{BO}^{2}(I)_{\lambda}}\label{eq:ShorttimeDuality-1}\\
 & =\sup_{\Vert v\Vert_{V_{BO}^{2}(I)_{\lambda}}=1}\left|\int_{I}\int_{\lambda\T}(\partial_{x}P_{N}v(s,x))P_{N_{1}}u_{1}(s,x)\ldots P_{N_{k}}u_{k}(s,x)dxds\right|.\nonumber 
\end{align}
As in Corollary \ref{cor:HighLowShorttimeEstimates}, either $P_{N}vP_{N_{1}}u_{1}$
or $P_{N}vP_{N_{2}}u_{2}$ are amenable to a bilinear Strichartz estimate
(after breaking the frequency support into intervals of length $cN$
for some $c\ll1$). Say we can use the bilinear Strichartz estimate
for $P_{N}vP_{N_{1}}u_{1}$. Then we find using \eqref{eq:ShorttimeU2Estimate}
and pointwise bounds: 
\[
\begin{split}\eqref{eq:ShorttimeDuality-1} & \lesssim N\sup_{\Vert v\Vert_{V_{BO}^{2}}=1}\Vert P_{N}vP_{N_{1}}u_{1}\Vert_{L_{t}^{2}(I,L_{\lambda}^{2})}\Vert P_{N_{2}}u_{2}P_{N_{3}}u_{3}\Vert_{L_{t}^{2}(I,L_{\lambda}^{2})}\\
 & \quad\prod_{i=4}^{k}\Vert P_{N_{i}}u_{i}\Vert_{L_{t}^{\infty}(I,L_{\lambda}^{\infty})}\\
 & \lesssim\log\langle N\rangle N_{3}^{1/2,0}\left(\prod_{i=4}^{k}N_{i}^{1/2}\right)\prod_{i=1}^{k}\sup_{\substack{I\subseteq[0,T],\\
|I|=N_{i}^{-1}
}
}\Vert P_{N}u_{i}\Vert_{U_{BO}^{2}(I)_{\lambda}}.
\end{split}
\]

\underline{Subcase III:} $N_{1}\sim N_{2}\sim N_{3}$. We claim that
\eqref{eq:FrequencyLocalizedShorttimeNonlinearEstimate} holds with
$C(N,N_{1},\ldots,N_{k})=\prod_{i=3}^{k}N_{i}^{1/2}$. Indeed, we
use the embedding $L^{1}(I,L_{\lambda}^{2})\hookrightarrow DU_{BO}^{2}(I)_{\lambda}$,
H\"older in time, and $L_{t,x}^{6}$-Strichartz \eqref{eq:L6StrichartzEstimateV2}
to find 
\[
\begin{split} & \quad\Vert P_{N}\partial_{x}(P_{N_{1}}u_{1}\ldots P_{N_{k}}u_{k})\Vert_{DU_{BO}^{2}(I)_{\lambda}}\\
 & \lesssim N\Vert P_{N_{1}}u_{1}\ldots P_{N_{k}}u_{k}\Vert_{L_{t}^{1}(I,L_{\lambda}^{2})}\\
 & \lesssim N^{1/2}\Vert P_{N_{1}}u_{1}\ldots P_{N_{k}}u_{k}\Vert_{L_{t}^{2}(I,L_{\lambda}^{2})}\\
 & \lesssim N^{1/2}\prod_{i=1}^{k}\Vert P_{N_{i}}u_{i}\Vert_{L_{t}^{6}(I,L_{\lambda}^{6})}\prod_{i=4}^{k}\Vert P_{N_{i}}u_{i}\Vert_{L_{t}^{\infty}(I,L_{\lambda}^{\infty})}\\
 & \lesssim\left(\prod_{i=3}^{k}N_{i}^{1/2}\right)\prod_{i=1}^{k}\sup_{\substack{I_{i}\subseteq[0,T],\\
|I_{i}|=N_{i}^{-1}
}
}\Vert P_{N_{i}}u_{i}\Vert_{U_{BO}^{2}(I_{i})_{\lambda}}.
\end{split}
\]

\textbf{Case B:} $N_{1}\sim N_{2}\gg N$, $N_{1}\gtrsim1$.

In this case, we have to increase the time localization matching the
highest frequency.

\underline{Subcase I:} $N_{1}\gg N_{3}$. We claim that \eqref{eq:FrequencyLocalizedShorttimeNonlinearEstimate}
holds with $C(N,N_{1},\ldots,N_{k})=N^{3/2,0}\log\langle N_{1}\rangle N_{3}^{1/2,0}\prod_{i=4}^{k}N_{i}^{1/2}$.
Indeed, we use duality to write 
\[
\begin{split} & \quad\Vert P_{N}\partial_{x}(P_{N_{1}}u_{1}\ldots P_{N_{k}}u_{k})\Vert_{DU_{BO}^{2}(I)_{\lambda}}\\
 & =\sup_{\Vert v\Vert_{V_{BO}^{2}(I)_{\lambda}=1}}\bigg|\int_{I}\int_{\lambda\T}(\partial_{x}P_{N}v)(s,x)P_{N_{1}}u_{1}(s,x)\ldots P_{N_{k}}u_{k}(s,x)dxds\bigg|\\
 & \leq\sup_{\Vert v\Vert_{V_{BO}^{2}(I)_{\lambda}=1}}\sum_{\substack{J\subseteq I,\\
|J|=N_{1}^{-1}
}
}\bigg|\int_{J}\int_{\lambda\T}(\partial_{x}P_{N}v)(s,x)P_{N_{1}}u_{1}(s,x)\ldots P_{N_{k}}u_{k}(s,x)dxds\bigg|.
\end{split}
\]
For an interval $J\subseteq[0,T]$ with $|J|=N_{1}^{-1}$, we estimate
the above integral by 
\[
\begin{split} & \quad\Vert(\partial_{x}P_{N}v)P_{N_{1}}u_{1}\dots P_{N_{k}}u_{k}\Vert_{L_{t}^{1}(J,L_{\lambda}^{2})}\\
 & \lesssim\Vert(\partial_{x}P_{N}v)P_{N_{1}}u_{1}\Vert_{L_{t}^{2}(J,L_{\lambda}^{2})}\Vert P_{N_{2}}u_{2}P_{N_{3}}u_{3}\Vert_{L_{t}^{2}(J,L_{\lambda}^{2})}\prod_{i=4}^{k}\Vert P_{N_{i}}u_{i}\Vert_{L_{t}^{\infty}(J,L_{\lambda}^{\infty})}\\
 & \lesssim N_{1}^{-1}\log\langle N_{1}\rangle N^{1/2,0}N_{3}^{1/2,0}\Vert\partial_{x}P_{N}v\Vert_{V_{BO}^{2}(J)_{\lambda}}\prod_{i=1}^{3}\Vert P_{N_{i}}u_{i}\Vert_{U_{BO}^{2}(J)_{\lambda}}\\
 & \quad\prod_{i=4}^{k}N_{i}^{1/2}\Vert P_{N_{i}}u_{i}\Vert_{U_{BO}^{2}(J)_{\lambda}}.
\end{split}
\]
The claim follows from replacing $\partial_{x}$ by $N$ and summing
over $J\subseteq I$, $|J|=N_{1}^{-1}$, which loses $N^{0,-1}N_{1}$.

\underline{Subcase II:} $N_{1}\sim N_{3}\gg N_{4}$. We claim that
\eqref{eq:FrequencyLocalizedShorttimeNonlinearEstimate} holds with
$C(N,N_{1},\ldots,N_{k})=N^{3/2,0}\log\langle N_{1}\rangle\prod_{i=4}^{k}N_{i}^{1/2}$.
Indeed, we notice that as in Corollary \ref{cor:HighLowShorttimeEstimates},
(after breaking the frequency support into intervals of length $cN$
for some $c\ll1$) there are two frequencies among $N_{1},N_{2},N_{3}$
for which the bilinear Strichartz estimate applies. Thus the same
proof as in Subcase B.I works.

\underline{Subcase III:} $N_{1}\sim N_{4}$. We claim that \eqref{eq:FrequencyLocalizedShorttimeNonlinearEstimate}
holds with $C(N,N_{1},\ldots,N_{k})=N^{3/2,1/2}\prod_{i=4}^{k}N_{i}^{1/2}$.
We use the embedding $L^{1}(I,L_{\lambda}^{2})\hookrightarrow DU_{BO}^{2}(I)_{\lambda}$,
H\"older in time, and Bernstein to find: 
\[
\begin{split} & \quad\|P_{N}\partial_{x}(P_{N_{1}}u_{1}\dots P_{N_{k}}u_{k})\|_{DU_{BO}^{2}(I)_{\lambda}}\\
 & \lesssim N\|P_{N}(P_{N_{1}}u_{1}\dots P_{N_{k}}u_{k})\|_{L_{t}^{1}(I,L_{\lambda}^{2})}\\
 & \lesssim N_{1}N^{1,0}\sup_{\substack{J\subseteq I,\\
|J|=N_{1}^{-1}
}
}\|P_{N}(P_{N_{1}}u_{1}\dots P_{N_{k}}u_{k})\|_{L_{t}^{1}(J,L_{\lambda}^{2})}\\
 & \lesssim N_{1}^{1/2}N^{3/2,1/2}\sup_{\substack{J\subseteq I,\\
|J|=N_{1}^{-1}
}
}\|P_{N_{1}}u_{1}\dots P_{N_{k}}u_{k}\|_{L_{t}^{2}(J,L_{\lambda}^{1})}.
\end{split}
\]
For an interval $J\subseteq[0,T]$ with $|J|=N_{1}^{-1}$, we estimate
the above $L_{J}^{2}L_{\lambda}^{1}$-norm using $L_{t}^{8}L_{x}^{4}$-Strichartz
\eqref{eq:L6StrichartzEstimateV2} and pointwise bounds to find 
\[
\begin{split} & \quad\|P_{N_{1}}u_{1}\dots P_{N_{k}}u_{k}\|_{L_{t}^{2}(J,L_{\lambda}^{1})}\\
 & \lesssim\prod_{i=1}^{4}\sup_{\substack{I_{i}\subseteq[0,T],\\
|I_{i}|=N_{i}^{-1}
}
}\|P_{N_{i}}u_{i}\|_{L_{t}^{8}(I_{i},L_{\lambda}^{4})}\prod_{i=5}^{k}\sup_{\substack{I_{i}\subseteq[0,T],\\
|I_{i}|=N_{i}^{-1}
}
}\|P_{N_{i}}u_{i}\|_{L_{t}^{\infty}(I_{i},L_{\lambda}^{\infty})}\\
 & \lesssim\Big(\prod_{i=5}^{k}N_{i}^{1/2}\Big)\bigg(\prod_{i=1}^{k}\sup_{\substack{I_{i}\subseteq[0,T],\\
|I_{i}|=N_{i}^{-1}
}
}\|P_{N_{i}}u_{i}\|_{U_{BO}^{2}(I_{i})_{\lambda}}\bigg).
\end{split}
\]
The claim follows from substituting this bound into the above $DU_{BO}^{2}(I)_{\lambda}$
bound with $N_{1}\sim N_{4}$.

\textbf{Case C:} $N_{1}\lesssim1$.

Note that $N\lesssim N_{1}\lesssim1$. Thus we may assume that $I=[0,T]$.
We use the embedding $DU_{BO}^{2}(I)_{\lambda}\hookrightarrow L_{t}^{1}(I,L_{\lambda}^{2})$,
H\"older and Bernstein's inequalities to find 
\[
\begin{split} & \Vert P_{N}\partial_{x}(P_{N_{1}}u_{1}\ldots P_{N_{k}}u_{k})\Vert_{DU_{BO}^{2}(I)_{\lambda}}\\
 & \lesssim N\Vert P_{N_{1}}u_{1}\ldots P_{N_{k}}u_{k}\Vert_{L_{t}^{1}(I,L_{\lambda}^{2})}\\
 & \lesssim N\Vert P_{N_{1}}u_{1}\Vert_{L_{t}^{\infty}(I,L_{\lambda}^{2})}\prod_{i=2}^{k}\Vert P_{N_{i}}u_{i}\Vert_{L_{t}^{\infty}(I,L_{\lambda}^{\infty})}\\
 & \lesssim N^{1/2}\prod_{i=1}^{k}N_{i}^{1/2}\Vert P_{N_{i}}u_{i}\Vert_{U_{BO}^{2}(I)_{\lambda}}.
\end{split}
\]

This completes the proof. 
\end{proof}

Proposition \ref{prop:ShorttimeNonlinearEstimate} follows from Lemma
\ref{lem:short-time-DU2-estimate} by summing over the Littlewood-Paley
pieces. We remark that one should use the gain (e.g. $N_{2}^{1,0}$
in Case A.I.) for low frequencies $\lesssim1$ under the condition
$s_{1}<1/2$. We omit the proof.

\subsection{\label{subsec:L1L2-estimate}$L_{t}^{1}L_{x}^{2}$-estimates}

In this section, we estimate the nonlinearity in $L_{t}^{1}L_{x}^{2}$.
Due to the embedding $L_{t}^{1}L_{x}^{2}\hookrightarrow DU_{BO}^{2}$,
the $L_{t}^{1}L_{x}^{2}$-norm is never smaller than the $DU_{BO}^{2}$-norm.
Indeed, if one uses Lemma \ref{lem:shorttime-L1L2} (see below) instead
of Lemma \ref{lem:short-time-DU2-estimate}, then the short-time nonlinear
estimate follows only for $s_{2}\geq\frac{1}{2}$. However, $L_{t}^{1}L_{x}^{2}$
is well-suited for the product estimate (e.g. $\|fg\|_{L_{t}^{1}L_{x}^{2}}\lesssim\|f\|_{L_{t}^{1}L_{x}^{2}}\|g\|_{L_{t,x}^{\infty}}$).
This helps us to deal with the spacetime error terms in Section \ref{sec:DifferenceEstimate}.

We start with the short-time $L_{t}^{1}L_{x}^{2}$-estimate of each
Littlewood-Paley piece. 
\begin{lem}[Short-time $L_{t}^{1}L_{x}^{2}$-estimate]
\label{lem:shorttime-L1L2}Let $T\in(0,1]$ and $N_{1}\gtrsim N_{2}\gtrsim\dots\gtrsim N_{k}$.
We find 
\begin{align*}
 & \sup_{I\subseteq[0,T],\,|I|=N^{-1}}\Vert P_{N}\partial_{x}(P_{N_{1}}u_{1}\ldots P_{N_{k}}u_{k})\Vert_{L_{t}^{1}(I,L_{\lambda}^{2})}\\
 & \;\lesssim\begin{cases}
N^{1,0}N_{2}^{1/2,0}\prod_{i=3}^{k}N_{i}^{1/2}\cdot\prod_{i=1}^{k}\|P_{N_{i}}u_{i}\|_{F_{\lambda}^{0}}, & \text{if }N\sim N_{1}\gg N_{2},\\
N^{1/2}N_{1}^{1,0}N_{3}^{1/2,0}\prod_{i=4}^{k}N_{i}^{1/2}\cdot\prod_{i=1}^{k}\|P_{N_{i}}u_{i}\|_{F_{\lambda}^{0}} & \text{if }N_{1}\sim N_{2}\gtrsim N.
\end{cases}
\end{align*}
\end{lem}

\begin{proof}
By the proof of Lemma \ref{lem:short-time-DU2-estimate}, it suffices
to deal with Cases A.II, B.I, and B.II.

\underline{Case A.II:} $N\sim N_{1}\sim N_{2}\gg N_{3}$ with $N\gtrsim1$.
We use H\"older in time, bilinear Strichartz, and pointwise bounds
to find 
\[
\begin{split} & \quad\Vert P_{N}\partial_{x}(P_{N_{1}}u_{1}\ldots P_{N_{k}}u_{k})\Vert_{L_{t}^{1}(I,L_{\lambda}^{2})}\\
 & \lesssim N^{1/2}\|P_{N_{1}}u_{1}\|_{L_{t}^{\infty}(I,L_{\lambda}^{\infty})}\|P_{N_{2}}u_{2}P_{N_{3}}u_{3}\|_{L_{t}^{2}(I,L_{\lambda}^{2})}\prod_{i=4}^{k}\|P_{N_{i}}u_{i}\|_{L_{t}^{\infty}(I,L_{\lambda}^{\infty})}\\
 & \lesssim N^{1/2}N_{3}^{1/2,0}\prod_{i=4}^{k}N_{i}^{1/2}\cdot\prod_{i=1}^{k}\Vert P_{N_{i}}u_{i}\Vert_{F_{\lambda}^{0}}.
\end{split}
\]

\underline{Case B.I:} We consider the case $N_{1}\sim N_{2}\gg\max(N,N_{3})$
with $N_{1}\gtrsim1$. Here we need to increase time localization.
Decompose $I$ into $O(N^{0,-1}N_{1})$-many intervals $J$ of length
$N_{1}^{-1}$. We apply Bernstein and bilinear Strichartz to find
\[
\begin{split} & \quad\Vert P_{N}\partial_{x}(P_{N_{1}}u_{1}\ldots P_{N_{k}}u_{k})\Vert_{L_{t}^{1}(J,L_{\lambda}^{2})}\\
 & \lesssim N^{3/2}\|P_{N_{1}}u_{1}\dots P_{N_{k}}u_{k}\|_{L_{t}^{1}(J,L_{\lambda}^{1})}\\
 & \lesssim N^{3/2}N_{1}^{-1}\|P_{N_{1}}u_{1}P_{N_{3}}u_{3}\|_{L_{t}^{2}(J,L_{\lambda}^{2})}\|P_{N_{2}}u_{2}P_{N_{4}}u_{3}\|_{L_{t}^{2}(J,L_{\lambda}^{2})}\prod_{i=5}^{k}\|P_{N_{i}}u_{i}\|_{L_{t}^{\infty}(I,L_{\lambda}^{\infty})}\\
 & \lesssim\Big(N^{3/2}N_{1}^{-1}N_{3}^{1/2,0}N_{4}^{1/2,0}\prod_{i=5}^{k}N_{i}^{1/2}\Big)\prod_{i=1}^{k}\Vert P_{N_{i}}u_{i}\Vert_{F_{\lambda}^{0}}.
\end{split}
\]
The claim follows from summing over $J\subseteq I$, $|J|=N_{1}^{-1}$,
which loses $N^{0,-1}N_{1}$.

\underline{Case B.II:} We consider the case $N_{1}\sim N_{2}\sim N_{3}\gg\max(N,N_{4})$
with $N_{1}\gtrsim1$. As in the proof of Lemma \ref{lem:short-time-DU2-estimate},
there is a pair of frequencies among $N_{1},N_{2},N_{3}$ amenable
to the bilinear Strichartz estimate. Say $N_{1},N_{3}$ is such a
pair. One then proceeds as in Case B.I. This completes the proof. 
\end{proof}
We turn to the $L_{T}^{1}L_{x}^{2}$-estimate of the nonlinearity.
For later purposes, we estimate each $P_{N}$-portion of the nonlinearity,
i.e., $P_{N}\partial_{x}(u^{k})$. Due to High$\times$High$\to$Low
interaction, we cannot say that $P_{N}\partial_{x}(u^{k})=(P_{N}\partial_{x}u)u^{k-1}$.
However we want to ensure $P_{N}\partial_{x}(u^{k})\lesssim c_{N}u^{k-1}$,
by introducing some $\ell_{N}^{2}$-sequence $c_{N}$ that mimics
$\|P_{\sim N}u\|_{F^{0}}$ and also incorporates frequencies other
than $N$. This is related with the frequency envelopes in \cite{Tao2001WaveMaps}.

Let $\overline{s}=(s_{1},s_{2})$, $0<s_{1}<1/2<s_{2}$, $\delta:=\min(\frac{1}{2}-s_{1},s_{2}-\frac{1}{2},\frac{1}{100})$,
and $s'\in\{-1\}\cup[0,\infty)$. Define for dyadic $N$ 
\begin{equation}
c_{N}^{(v,\overline{s},s')}:=\|P_{\sim N}v\|_{F_{\lambda}^{0}}+N^{\delta,-\delta}\cdot N^{-(\overline{s}+s')}\|v\|_{F_{\lambda}^{\overline{s}+s'}}\label{eq:c_N-definition}
\end{equation}
We use short-hand notation 
\[
c_{N}^{(u,\overline{s})}=c_{N}^{(u,\overline{s},0)}.
\]

Set 
\begin{equation}
d_{N}^{(v,\overline{s},s^{\prime})}=N^{\overline{s}+s'}c_{N}^{(v,\overline{s},s')}.\label{eq:d_N-definition}
\end{equation}
Due to $\delta>0$, we have a $\ell_{N}^{2}$-summation property 
\begin{equation}
\|d_{N}^{(v,\overline{s},s')}\|_{\ell_{N}^{2}}\lesssim\|v\|_{F_{\lambda}^{\overline{s}+s'}}.\label{eq:l2-summation-c_N}
\end{equation}

\begin{lem}[$L_{T}^{1}L_{x}^{2}$-estimate]
\label{lem:decomp-prod}Let $T\in(0,1]$; let $\overline{s}=(s_{1},s_{2})$,
$0<s_{1}<1/2<s_{2}$, $\delta:=\min(\frac{1}{2}-s_{1},s_{2}-\frac{1}{2},\frac{1}{100})$,
and $s'\in\{-1\}\cup[0,\infty)$. Then we have the $L_{T}^{1}L_{x}^{2}$-nonlinear
estimates 
\[
\begin{aligned} & \|P_{N}\partial_{x}(vu^{k-1})\|_{L_{t}^{1}([0,T],L_{\lambda}^{2})}\\
 & \lesssim N^{1/2,0}\|u\|_{F_{\lambda}^{\overline{s}}}^{k-2}\cdot\begin{cases}
\|u\|_{F_{\lambda}^{\overline{s}}}c_{N}^{(v,\overline{s},-1)} & \text{if }s'=-1,\\
c_{N}^{(v,\overline{s},s')}\|u\|_{F_{\lambda}^{\overline{s}}}+c_{N}^{(u,\overline{s},s')}\|v\|_{F_{\lambda}^{\overline{s}}} & \text{if }s'\geq0.
\end{cases}
\end{aligned}
\]
\end{lem}

\begin{rem}
Lemma \ref{lem:decomp-prod} can be viewed as 
\[
P_{N}\partial_{x}(vu^{k-1})\approx(P_{N}\partial_{x}v)u^{k-1}+(P_{N}\partial_{x}u)vu^{k-2},
\]
provided that $v\in F_{\lambda}^{\overline{s}-1}$ and $u\in F_{\lambda}^{\overline{s}}$
for $s'=-1$, or $v,u\in F_{\lambda}^{\overline{s}+s'}$ for $s'\geq0$.
We need to choose a function space to make sense of $\approx$. The
crude choice $L_{t}^{\infty}L_{x}^{2}$ rather than $L_{t}^{1}L_{x}^{2}$
would work for $s_{2}>\frac{3}{4}$, but it seems difficult to reach
$s_{2}>\frac{1}{2}$ due to problematic High$\times$High$\to$Low
interaction of $vu$. In order to reach $s_{2}>\frac{1}{2}$, we choose
a slightly larger space $L_{t}^{1}L_{x}^{2}$.
\end{rem}

\begin{rem}
These estimates will be used in the energy estimates for the differences
of solutions. The case $s'=-1$ will be used for the Lipschitz estimates
in a weaker topology; see \eqref{eq:s-1-int-by-pts-time}. The case
$s'\in\{0,1\}$ will be used in establishing continuity estimates;
see \eqref{eq:s-int-by-pts-time-1} and \eqref{eq:s-int-by-pts-time-2}.
\end{rem}


\begin{proof}[Proof of Lemma \ref{lem:decomp-prod}]
For technical convenience, let us give a proof only for $1/2<s_{2}<1$.
The case $s_{2}\geq1$ is in fact easier, because we have higher regularity.

We decompose $[0,T]$ into $O(N^{0,1})$-many intervals of length
$\sim\min(N^{-1},T)$. It suffices to show the estimates 
\begin{align}
\|P_{N}\partial_{x}(vu^{k-1})\|_{L_{t}^{1}(I,L_{\lambda}^{2})} & \lesssim N^{1/2,0}\|u\|_{F_{\lambda}^{\overline{s}}}^{k-1}c_{N}^{(v,\overline{s},-1)},\label{eq:L1L2-claim1}\\
\|P_{N}\partial_{x}(vu^{k-1})\|_{L_{t}^{1}(I,L_{\lambda}^{2})} & \lesssim N^{1/2,0}\|u\|_{F_{\lambda}^{\overline{s}}}^{k-2}(c_{N}^{(v,\overline{s},s')}\|u\|_{F_{\lambda}^{s}}+c_{N}^{(u,\overline{s},s')}\|v\|_{F_{\lambda}^{\overline{s}}})\label{eq:L1L2-claim2}
\end{align}
uniformly for $I\subseteq[0,T]$ with $|I|=\min(N^{-1},T)$. Fix such
$I$. We will estimate each piece 
\begin{equation}
\|P_{N}\partial_{x}(P_{K}vP_{K_{1}}u\dots P_{K_{k-1}}u)\|_{L_{t}^{1}(I,L_{\lambda}^{2})}\label{eq:L1L2-est-temp}
\end{equation}
and sum over $K,K_{1},\dots,K_{k-1}$. Denote by $M_{1},\dots,M_{k}$
the decreasing rearrangement of $K,K_{1},\dots,K_{k-1}$. Due to otherwise
impossible frequency interactions, we consider two cases: (Case A)
$N\sim M_{1}\gg M_{2}$ and (Case B) $M_{1}\sim M_{2}\gtrsim N$.

\textbf{Case A:} $N\sim M_{1}\gg M_{2}$. Recall from Lemma \ref{lem:shorttime-L1L2}
that 
\begin{align*}
\eqref{eq:L1L2-est-temp} & \lesssim N^{1,0}M_{2}^{1/2,0}\Big(\prod_{i=3}^{k}M_{i}^{1/2}\Big)\|P_{K}v\|_{F_{\lambda}^{0}}\prod_{i=1}^{k-1}\|P_{K_{i}}u\|_{F_{\lambda}^{0}}.
\end{align*}

\underline{Subcase I:} $K=M_{1}$. We find 
\[
\eqref{eq:L1L2-est-temp}\lesssim N^{1,0}\|P_{K}v\|_{F_{\lambda}^{0}}\prod_{i=1}^{k-1}\|P_{K_{i}}u\|_{F_{\lambda}^{1/2}}.
\]
Summing over $K_{1},\dots,K_{k-1}$ and $K\sim N$ yields both \eqref{eq:L1L2-claim1}
and \eqref{eq:L1L2-claim2}. (Recall $s_{1}<\frac{1}{2}<s_{2}$.)

\underline{Subcase II:} $K\in\{M_{2},\dots,M_{k}\}$. Note that $K_{1}=M_{1}\sim N$.
We find 
\[
\eqref{eq:L1L2-est-temp}\lesssim N^{1,0}\|P_{K_{1}}u\|_{F_{\lambda}^{0}}K^{\frac{1}{2}-s_{1},-(s_{2}+s')}\|P_{K}v\|_{F_{\lambda}^{\overline{s}+s^{\prime}}}\prod_{i=2}^{k-1}\|P_{K_{i}}u\|_{F_{\lambda}^{1/2}}.
\]
We sum over $K,K_{2},\dots,K_{k-1}$ and then use $K_{1}\sim N$.
If $s'=-1$, then this sums to $N^{\frac{3}{2}-s_{1},-2s_{2}+1}\|v\|_{F_{\lambda}^{\overline{s}-1}}\|u\|_{F_{\lambda}^{\overline{s}}}^{k-1}$,
yielding \eqref{eq:L1L2-claim1}. If $s'\geq0$, then this sums to
$N^{\frac{3}{2}-s_{1},0}\|P_{\sim N}u\|_{F_{\lambda}^{0}}\|v\|_{F_{\lambda}^{\overline{s}}}\|u\|_{F_{\lambda}^{\overline{s}}}^{k-2}$,
yielding \eqref{eq:L1L2-claim2}.

\textbf{Case B:} $M_{1}\sim M_{2}\gtrsim N$. Recall from Lemma \ref{lem:shorttime-L1L2}
that 
\[
\eqref{eq:L1L2-est-temp}\lesssim\Big(N^{1/2}M_{1}^{1,0}M_{3}^{1/2,0}\prod_{i=4}^{k}M_{i}^{1/2}\Big)\|P_{K}v\|_{F_{\lambda}^{0}}\prod_{i=1}^{k-1}\|P_{K_{i}}u\|_{F_{\lambda}^{0}}.
\]

\underline{Subcase I:} $K\in\{M_{1},M_{2}\}$. Note that $\{K,K_{1}\}=\{M_{1},M_{2}\}$.
We find 
\[
\eqref{eq:L1L2-est-temp}\lesssim N^{\frac{1}{2}}M_{1}^{-2s_{1}+1,-2s_{2}-s'}\|P_{K}v\|_{F_{\lambda}^{\overline{s}+s'}}\|P_{K_{1}}u\|_{F_{\lambda}^{\overline{s}}}\prod_{i=2}^{k-1}\|P_{K_{i}}u\|_{F_{\lambda}^{1/2}}.
\]
Since $s_{2}>\frac{1}{2}$ and $s'\geq-1$, we have $-2s_{2}-s'<0$.
Separately considering the cases $N\leq1$ and $N\geq1$, this sums
to $N^{\frac{1}{2},-2s_{2}-s'+\frac{1}{2}}\|v\|_{F_{\lambda}^{\overline{s}+s'}}\|u\|_{F_{\lambda}^{\overline{s}}}^{k-1}$,
yielding both \eqref{eq:L1L2-claim1} and \eqref{eq:L1L2-claim2}.

\underline{Subcase II:} $K\in\{M_{3},\dots,M_{k}\}$. Note that $\{K_{1},K_{2}\}=\{M_{1},M_{2}\}$.
We find 
\begin{align*}
\eqref{eq:L1L2-est-temp} & \lesssim N^{1/2}M_{1}^{1-2s_{1},-2s_{2}-\ell}\|P_{M_{1}}u\|_{F_{\lambda}^{\overline{s}+\ell}}\|P_{M_{2}}u\|_{F_{\lambda}^{\overline{s}}}\\
 & \quad\quad\times K^{\frac{1}{2}-s_{1},-(s_{2}+j)}\|P_{K}v\|_{F_{\lambda}^{\overline{s}+j}}\prod_{i=3}^{k-1}\|P_{K_{i}}u\|_{F_{\lambda}^{1/2}},
\end{align*}
where $j,\ell\in\R$. Choosing $j=-1$ and $\ell=0$, the above sums
to 
\[
N^{\frac{1}{2},-3s_{2}+\frac{3}{2}}\|v\|_{F_{\lambda}^{\overline{s}-1}}\|u\|_{F_{\lambda}^{\overline{s}}}^{k-1},
\]
yielding \eqref{eq:L1L2-claim1}. Choosing $j=0$ and $\ell=s'\geq0$,
the above sums to 
\[
N^{\frac{1}{2},-2s_{2}+\frac{1}{2}-s'}\|v\|_{F_{\lambda}^{\overline{s}}}\|u\|_{F_{\lambda}^{\overline{s}+s'}}\|u\|_{F_{\lambda}^{\overline{s}}}^{k-2},
\]
yielding \eqref{eq:L1L2-claim2}. This finishes the proof.
\end{proof}

\section{\label{sec:EnergyEstimateSolutions}Energy estimates}

The purpose of this section is to propagate the energy norm. 
\begin{prop}[Energy estimates for solutions]
\label{prop:EnergyEstimateSolutions} Let $0<s_{1}<1/2\leq s_{2}$
and set $\overline{s}=(s_{1},s_{2})$. Let $s'\geq0$. Then there
exists $\epsilon=\epsilon(\overline{s})>0$ with the following properties:
(denote $\overline{s}(\epsilon)=(s_{1}+\epsilon,s_{2}-\epsilon)$)
\begin{itemize}
\item If $s_{2}>\frac{1}{2}$ and $u$ is a smooth solution to \eqref{eq:GeneralizedBenjaminOnoEquation}
on $[0,T]\subseteq[0,1]$, then we have 
\begin{equation}
\Vert u\Vert_{E_{\lambda}^{\overline{s}+s'}(T)}^{2}\lesssim\Vert u_{0}\Vert_{H_{\lambda}^{\overline{s}+s'}}^{2}+\|u\|_{F_{\lambda}^{\overline{s}(\epsilon)+s'}(T)}^{2}(\Vert u\Vert_{F_{\lambda}^{\overline{s}(\epsilon)}(T)}^{k-1}+\Vert u\Vert_{F_{\lambda}^{s_{1}+\varepsilon,\frac{1}{2}-\varepsilon}(T)}^{2k-2}).\label{eq:EnergyEstimateSolutions}
\end{equation}
\item If $s_{2}=\frac{1}{2}$ and $u$ is a smooth solution to \eqref{eq:GeneralizedBenjaminOnoEquation}
with $k=4$ on $[0,T]\subseteq[0,1]$, then we have 
\begin{equation}
\Vert u\Vert_{E_{\lambda}^{\overline{s}+s'}(T)}^{2}\lesssim\Vert u_{0}\Vert_{H_{\lambda}^{\overline{s}+s'}}^{2}+\|u\|_{F_{\lambda}^{\overline{s}+s'}(T)}^{2}(\Vert u\Vert_{F_{\lambda}^{\overline{s}}(T)}^{k-1}+\Vert u\Vert_{F_{\lambda}^{\overline{s}(\epsilon)}(T)}^{2k-2}).\label{eq:EnergyEstimateII}
\end{equation}
\end{itemize}
\end{prop}

\begin{rem}
The case $s'>0$ corresponds to the persistence of regularity estimates.
\end{rem}

\begin{rem}[Besov version for $k\geq5$ at $s_{2}=1/2$]
If $u$ is a smooth solution to \eqref{eq:GeneralizedBenjaminOnoEquation}
with $k\geq5$, the argument of the proof of \eqref{eq:EnergyEstimateII}
gives the estimate 
\[
\Vert u\Vert_{E_{1,\lambda}^{s_{1},1/2}(T)}^{2}\lesssim\Vert u_{0}\Vert_{B_{1,\lambda}^{s_{1},1/2}(T)}^{2}+\Vert u\Vert_{F_{1,\lambda}^{s_{1}+\varepsilon,1/2}(T)}^{k+1}+\Vert u\Vert_{F_{1,\lambda}^{s_{1}+\varepsilon,1/2}(T)}^{2k}.
\]
In the above display $B_{1,\lambda}^{s_{1},1/2}$, $E_{1,\lambda}^{s_{1},1/2}(T)$,
and $F_{1,\lambda}^{s_{1}+\varepsilon,1/2}(T)$ denote $1$-Besov
variants, i.e., 
\[
\Vert u\Vert_{E_{1,\lambda}^{s_{1},1/2}(T)}=\Vert P_{\leq1}u(0)\Vert_{H_{\lambda}^{s_{1},1/2}}+\sum_{N\geq1}N^{1/2}\sup_{t\in[0,T]}\Vert P_{N}u(t)\Vert_{L_{\lambda}^{2}},
\]
and likewise for $B_{1,\lambda}^{s_{1},1/2}$ and $F_{1,\lambda}^{s_{1},1/2}(T)$.
\end{rem}

To prove Proposition \ref{prop:EnergyEstimateSolutions}, we have
to estimate the $E_{\lambda}^{\overline{s}+s'}$-norm 
\[
\|u\|_{E_{\lambda}^{\overline{s}+s'}(T)}^{2}=\Vert P_{\leq1}u(0)\Vert_{H_{\lambda}^{\overline{s}+s'}}^{2}+\sum_{N>1}N^{2(s_{2}+s')}\|P_{N}u\|_{L_{t}^{\infty}([0,T],L_{\lambda}^{2})}^{2}.
\]
Since the low frequency part ($\leq1$) is only written in terms of
the initial data, it suffices to consider the high frequency part
($>1$). Instead of estimating each $N^{2(s_{2}+s')}\|P_{N}u\|_{L_{T}^{\infty}L_{\lambda}^{2}}^{2}$
(having the symbol $N^{2(s_{2}+s')}\chi_{N}^{2}$) directly, it is
more convenient to estimate a variant of $N^{2(s_{2}+s')}\|P_{N}u\|_{L_{T}^{\infty}L_{\lambda}^{2}}^{2}$
with slowly varying symbols, which are smoother in a sense.

More concretely, we consider the following set of symbols: 
\begin{align*}
S^{s_{2}+s'} & =\{a_{N}^{s_{2}+s'}:N\gtrsim1\},\\
a_{N}^{s_{2}+s'}(\xi) & =N^{2(s_{2}+s')}(1-\chi(4\xi/N)).
\end{align*}
In what follows, we use short-hand notations $a_{N}=a_{N}^{s_{2}+s'}$
or $a=a_{N}$. These symbols satisfy the following properties (uniformly
in $N\gtrsim1$):
\begin{itemize}
\item (Vanishing) $a_{N}(\xi)=0$ if $|\xi|\leq N/4$. 
\item (Comparison with Sobolev weights) 
\begin{equation}
\begin{gathered}a_{N}(\xi)=N^{2(s_{2}+s')}\text{ for }|\xi|\geq N/2,\\
\sum_{N>1}a_{N}(\xi)\lesssim|\xi|^{2(s_{2}+s')}.
\end{gathered}
\label{eq:ComparisonSobolevWeight}
\end{equation}
\item (Regularity) $|\partial^{\alpha}a_{N}(\xi)|\lesssim_{\alpha}\max(N^{2(s_{2}+s')},a_{N}(\xi))\langle\xi\rangle^{-\alpha}$
for $|\xi|\geq N/4$.
\end{itemize}

For $a\in S^{s_{2}+s'}$, we define the generalized Sobolev (semi-)norm
\[
\Vert u(t)\Vert_{H^{a}}^{2}=\frac{1}{\lambda}\sum_{\xi\in\Z/\lambda}a(\xi)|\hat{u}(t,\xi)|^{2}=\int_{\Gamma_{\lambda}^{2}}a(\xi_{1})\hat{u}(t,\xi_{1})\hat{u}(t,\xi_{2})d\Gamma_{\lambda}^{2},
\]
where the real-valuedness $\overline{u}=u$ is used in the last equality.
The measure $\Gamma_{\lambda}^{k}$ on $\{(\xi_{1},\ldots,\xi_{k})\in\times_{i=1}^{k}\Z/\lambda\,|\,\sum_{i=1}^{k}\xi_{i}=0\}$
is defined via pullback: 
\[
\int_{\Gamma_{\lambda}^{k}}f(\xi_{1},\ldots,\xi_{k})d\Gamma_{\lambda}^{k}=\frac{1}{\lambda^{k-1}}\sum_{\xi_{1},\ldots,\xi_{k-1}\in\Z/\lambda}f(\xi_{1},\ldots,\xi_{k-1},-\xi_{1}-\ldots-\xi_{k-1}).
\]
Following \cite[Section 6]{KochTataru2012} we compute by the fundamental
theorem of calculus and \emph{symmetrization} 
\begin{align*}
 & \Vert u(t)\Vert_{H^{a}}^{2}\\
 & =\Vert u_{0}\Vert_{H^{a}}^{2}+C\int_{0}^{t}\int_{\Gamma_{\lambda}^{k+1}}A_{k}(\xi_{1},\dots,\xi_{k+1})\hat{u}(s,\xi_{1})\ldots\hat{u}(s,\xi_{k+1})d\Gamma_{\lambda}^{k+1}ds\\
 & =:\Vert u_{0}\Vert_{H^{a}}^{2}+CR_{k}(u,t),
\end{align*}
where 
\[
A_{k}(\xi_{1},\ldots,\xi_{k+1})=\sum_{i=1}^{k+1}a(\xi_{i})\xi_{i}.
\]

This symmetrization can be viewed as a multilinear version of the
integration by parts in space. As the symbol $\xi\mapsto a(\xi)\xi$
is odd, $A_{k}(\xi_{1},\dots,\xi_{k+1})$ enjoys better bounds than
$a(\xi_{i})\xi_{i}$ on the diagonal set $\Gamma_{\lambda}^{k+1}$.
For instance, when $(\xi_{1},\dots,\xi_{k+1})\in\Gamma_{\lambda}^{k+1}$
satisfies $|\xi_{1}|\sim|\xi_{2}|\gtrsim|\xi_{3}|\gtrsim\cdots\gtrsim|\xi_{k+1}|$,
we have 
\begin{align}
A_{k}(\xi_{1},\dots,\xi_{k+1}) & =a(\xi_{1})\xi_{1}+a(\xi_{2})\xi_{2}+O(a(\xi_{3})|\xi_{3}|)\label{eq:int-by-pts-space-a}\\
 & =a(\xi_{1})(\xi_{1}+\xi_{2})+(a(\xi_{2})-a(-\xi_{1}))\xi_{2}+O(a(\xi_{3})|\xi_{3}|)\nonumber \\
 & =O(a(\xi_{1})|\xi_{3}|),\nonumber 
\end{align}
which is better than the crude bound $a(\xi_{1})|\xi_{1}|$.

The symmetrization by itself does not make use of dispersion. To observe
further smoothing effects from the dispersion, we change to the interaction
picture via $\hat{v}(t,\xi)=e^{it\xi|\xi|}\hat{u}(t,\xi)$ to find
\begin{multline}
R_{k}(u,t)\\
=\int_{0}^{t}\int_{\Gamma_{\lambda}^{k+1}}e^{-is\Omega_{k}(\xi_{1},\ldots,\xi_{k+1})}A_{k}(\xi_{1},\ldots,\xi_{k+1})\hat{v}(s,\xi_{1})\ldots\hat{v}(s,\xi_{k+1})d\Gamma_{\lambda}^{k+1}ds,\label{eq:R_k}
\end{multline}
where $\Omega_{k}$ is the $k$\emph{-resonance} \emph{function }
\[
\Omega_{k}(\xi_{1},\ldots,\xi_{k+1})=\sum_{i=1}^{k+1}\xi_{i}|\xi_{i}|.
\]
For frequency interactions such that $\Omega_{k}(\xi_{1},\dots,\xi_{k+1})$
is large, i.e., the \emph{nonresonant interactions}, we can use that
$e^{is\Omega_{k}}$ rapidly oscillates by writing $e^{-is\Omega_{k}}=(-i\Omega_{k})^{-1}\partial_{s}e^{-is\Omega_{k}}$
and integrate by parts in time. See for instance Lemma \ref{lem:Contribution-bOmega}.

In previous works on the Benjamin-Ono and the modified Benjamin-Ono
equation (see \cite{Schippa2017MBO,Schippa2020DispersionGeneralizedBenjaminOno}
for details), we have 
\begin{align*}
|\Omega_{2}(\xi_{1},\xi_{2},\xi_{3})| & \sim|\xi_{2}\xi_{3}|,\\
\left|\frac{A_{k}(\xi_{1},\ldots,\xi_{k+1})}{\Omega_{k}(\xi_{1},\ldots,\xi_{k+1})}\right| & \lesssim\frac{a(\xi_{1})}{|\xi_{1}|}\qquad\forall k\in\{2,3\}.
\end{align*}
In particular, in case of $k=2$, frequency interactions among mean
zero functions are always nonresonant. In case of $k=3$, there are
nontrivial resonant interactions, but the multiplier $A_{k}$ also
vanishes simultaneously.

In case of $k\geq4$, the multiplier $A_{k}$ does not necessarily
vanish for resonant interactions. However, we show that $A_{k}$ specialized
to resonant interaction satisfies better bounds than \eqref{eq:int-by-pts-space-a}.
It turns out (Lemma \ref{lem:decomposition-Ak}) that we have a decomposition
\[
A_{k}=b\Omega_{k}+c,
\]
such that 
\[
|b|\lesssim a(\xi_{1})|\xi_{1}|^{-1}\quad\text{and}\quad|c|\lesssim\big(a(\xi_{1})|\xi_{1}|^{-1}|\xi_{3}|+a(\xi_{3})\big)|\xi_{5}|.
\]

Thanks to \eqref{eq:ComparisonSobolevWeight}, Proposition \ref{prop:EnergyEstimateSolutions}
follows from summing over $N$ the following.
\begin{lem}[Estimate of each $\|u\|_{L_{T}^{\infty}H_{\lambda}^{a_{N}}}^{2}$]
\label{lem:ReductionEnergyEstimates}Assume the hypotheses of Proposition
\ref{prop:EnergyEstimateSolutions}. Let $a_{N}\in S^{s_{2}+s'}$
for $N\gtrsim1$.
\begin{itemize}
\item If $s_{2}>1/2$, there exists $\epsilon=\epsilon(\overline{s})>0$
such that we have the estimates 
\[
\sup_{t\in[0,T]}|R_{k}(u,t)|\lesssim N^{-\epsilon}(\Vert u\Vert_{F_{\lambda}^{\overline{s}(\epsilon)}(T)}^{k+1}+\Vert u\Vert_{F_{\lambda}^{s_{1}+\epsilon,1/2-\epsilon}(T)}^{2k}).
\]
\item If $k=4$ and $s_{2}=1/2$, there exists $\epsilon>0$ such that we
have the estimates 
\[
\sup_{t\in[0,T]}|R_{4}(u,t)|\lesssim\Vert u\Vert_{F_{\lambda}^{a_{N}}(T)}^{2}\Vert u\Vert_{F_{\lambda}^{\overline{s}}(T)}^{3}+N^{-\varepsilon}\|u\|_{F_{\lambda}^{\overline{s}(\epsilon)+s'}(T)}^{2}\Vert u\Vert_{F_{\lambda}^{\overline{s}(\epsilon)}(T)}^{6}.
\]
\end{itemize}
Here, we recall that $\overline{s}=(s_{1},s_{2})$ and $\overline{s}(\epsilon)=(s_{1}+\epsilon,s_{2}-\epsilon)$,
and used the notation 
\begin{equation}
\Vert u\Vert_{F_{\lambda}^{a_{N}}(T)}^{2}=N^{2(s_{2}+s')}\sum_{N^{\prime}\geq N}\sup_{\substack{I\subseteq[0,T],\\
|I|=(N^{\prime})^{-1}
}
}\Vert P_{N^{\prime}}u\Vert_{U_{BO}^{2}(I;L_{\lambda}^{2})}^{2}.\label{eq:F-aN-def}
\end{equation}
\end{lem}

Henceforth, we focus on showing Lemma \ref{lem:ReductionEnergyEstimates}.
In Section \ref{subsec:decomposition-Ak}, we give details of the
above decomposition of $A_{k}$. In Section \ref{subsec:Contribution-bOmega},
the contribution of $b\Omega_{k}$ in $R_{k}(u,t)$ is handled in
a favorable way through integration by parts. In Section \ref{subsec:Contribution-c},
the contribution of $c$ in $R_{k}(u,t)$ is estimated. We do not
integrate by parts in time, but we can use the bound of $c$ in \eqref{eq:SymbolEstimate-c},
which is better than the integration by parts in space bound \eqref{eq:int-by-pts-space-a}.

\subsection{\label{subsec:decomposition-Ak}Decomposition of $A_{k}$}

The aim of this subsection is to detail the decomposition of $A_{k}$:
Lemma \ref{lem:decomposition-Ak}.

Let us introduce some notations. Let $N_{1},\dots,N_{k+1}\in2^{\Z}\cap[\lambda^{-1},\infty)$.
For $\pm_{1},\dots,\pm_{k+1}\in\{+,-\}$, define 
\begin{align*}
D_{\pm_{1}N_{1},\dots,\pm_{k+1}N_{k+1}} & :=\{(\xi_{1},\dots,\xi_{k+1})\in\Z^{k+1}/\lambda:\\
 & \quad\pm_{i}\xi_{i}\geq0,\ |\xi_{i}|\sim N_{i},i\in\{1,\dots,k+1\}\}
\end{align*}
with the exception to take $|\xi_{i}|\lesssim N_{i}$ for $N_{i}=\lambda^{-1}$.
We then define the set $D_{N_{1},\dots,N_{k+1}}$ relevant to the
frequency interactions: 
\begin{align*}
D_{N_{1},\dots,N_{k+1}}:=\bigcup\Big\{ D_{\pm_{1}N_{1},\dots,\pm_{k+1}N_{k}} & :D_{\pm_{1}N_{1},\dots,\pm_{k+1}N_{k}}\cap\Gamma_{\lambda}^{k+1}\neq\emptyset,\\
 & \qquad\pm_{1},\dots,\pm_{k+1}\in\{+,-\}\Big\}.
\end{align*}
In other words, $D_{N_{1},\dots,N_{k+1}}$ is the union of rectangles
on which nontrivial frequency interactions can occur.
\begin{lem}[Decomposition of $A_{k}$]
\label{lem:decomposition-Ak}Fix $a=a_{N}\in S^{s_{2}+s'}$ for $N\gtrsim1$.
Let $N_{1}\sim N_{2}\gtrsim\dots\gtrsim N_{k+1}$. Then, there exist
$b$ and $c$ defined on $D_{N_{1},\dots,N_{k+1}}$ satisfying the
following:
\begin{enumerate}
\item (Decomposition of $A_{k}$ on $\Gamma_{\lambda}^{k+1}$) We have 
\[
A_{k}=b\Omega_{k}+c\quad\text{on }\Gamma_{\lambda}^{k+1}\cap D_{N_{1},\dots,N_{k+1}}.
\]
\item (Symbol regularity) For $\alpha_{1},\dots,\alpha_{k+1}\in\N_{0}$,
the symbols $b$ and $c$ satisfy\footnote{In fact, when $a=a_{N}\in S^{s_{2}}$ and $N_{1}\sim N$, then one
should replace $a(N_{1})$ by $N^{2s_{2}}$ because $a(N_{1})$ itself
may vanish. The same remark applies for $a(N_{3})$, and other estimates
having upper bounds in terms of $a(N_{1})$ with $N_{1}\sim N$.} 
\begin{align}
|\partial_{\xi_{1}}^{\alpha_{1}}\dots\partial_{\xi_{k+1}}^{\alpha_{k+1}}b(\xi_{1},\dots,\xi_{k+1})| & \lesssim_{\alpha_{1},\dots,\alpha_{k+1}}a(N_{1})N_{1}^{-1}N^{-\alpha},\label{eq:SymbolEstimate-b}\\
|\partial_{\xi_{1}}^{\alpha_{1}}\dots\partial_{\xi_{k+1}}^{\alpha_{k+1}}c(\xi_{1},\dots,\xi_{k+1})| & \lesssim_{\alpha_{1},\dots,\alpha_{k+1}}(a(N_{1})N_{1}^{-1}N_{3}+a(N_{3}))N_{5}N^{-\alpha},\label{eq:SymbolEstimate-c}
\end{align}
where we denoted $N^{-\alpha}=N_{1}^{-\alpha_{1}}\dots N_{k+1}^{-\alpha_{k+1}}$.
\item (Support property of $c$) $c$ is supported on the region where $N_{3}\sim N_{4}$.
\end{enumerate}
\end{lem}

\begin{proof}
For the symbol regularity, we only show the pointwise bounds of $b$
and $c$ on $\Gamma_{\lambda}^{k+1}\cap D_{N_{1},\dots,N_{k+1}}$
for simplicity. The symbol regularity estimates on the larger set
$D_{N_{1},\dots,N_{k+1}}$ follow from deriving the explicit representation
formulas of $b$ and $c$; see for instance \cite[Proposition 6.3]{KochTataru2012}.

We can suppose that $N_{1}\gtrsim1$ as for $N_{1}\ll1$ due to vanishing
property of $a$ on low frequencies $A_{k}\equiv0$, and we can set
$b=c\equiv0$.

We proceed by Case-by-Case analysis.

\textbf{Case A:} $N_{3}\gg N_{4}$.

In this case, there are only nonresonant interactions; we claim that
\begin{equation}
|\Omega_{k}(\xi_{1},\ldots,\xi_{k+1})|\sim N_{1}N_{3}\sim|\Omega_{3}(\xi_{1},\xi_{2},\xi_{3})|.\label{eq:BO-like-nonresonance}
\end{equation}
By otherwise impossible interaction and symmetry, we can suppose that
$\xi_{1}>0$, $\xi_{2}<0$, $\xi_{3}<0$. We compute using $\xi_{1}=-(\xi_{2}+\xi_{3})+O(N_{4})$
\[
\begin{split}\Omega_{k}(\xi_{1},\ldots,\xi_{k+1}) & =\xi_{1}^{2}-\xi_{2}^{2}-\xi_{3}^{2}+O(N_{4}^{2})\\
 & =2\xi_{2}\xi_{3}+2(\xi_{2}+\xi_{3})\cdot O(N_{4})+O(N_{4}^{2})\\
 & =2\xi_{2}\xi_{3}+O(N_{2}N_{4}).
\end{split}
\]
Since $N_{3}\gg N_{4}$ and $N_{1}\sim N_{2}$, we have 
\[
|\Omega_{k}(\xi_{1},\dots,\xi_{k+1})|\sim N_{1}N_{3}\sim|\Omega_{3}(\xi_{1},\xi_{2},\xi_{3})|.
\]
Having settled the claim, we set $b$ and $c$ by 
\[
b:=\frac{A_{k}}{\Omega_{k}}\quad\text{and}\quad c:=0.
\]
The estimate of $b$ follows from \eqref{eq:int-by-pts-space-a} and
\eqref{eq:BO-like-nonresonance}.

\textbf{Case B:} $N_{1}\sim N_{2}\gg N_{3}\sim N_{4}$.

Due to otherwise impossible interaction and symmetry, we can suppose
that $\xi_{1}>0$, $\xi_{2}<0$. Choose $\ell\in\{4,\dots,k+1\}$
such that $N_{1}\sim N_{2}\gg N_{3}\sim N_{\ell}\gg N_{\ell+1}$ (set
$N_{k+2}=0$ when $\ell=k+1$).

\underline{Subcase I:} $\xi_{3}\dots,\xi_{\ell}$ have the same sign.
By $\xi_{1}>0$ and $\xi_{2}<0$, this means $\xi_{3},\ldots,\xi_{l}<0$.
Then, the following computation shows that this case is nonresonant:
\begin{align*}
\Omega_{k} & =\xi_{1}^{2}-\xi_{2}^{2}+O(N_{3}^{2})\\
 & =(\xi_{1}-\xi_{2})(-\xi_{3}-\cdots-\xi_{\ell}+O(N_{\ell+1}))+O(N_{3}^{2})\sim N_{1}N_{3},
\end{align*}
where in the last inequality we used the fact that $\xi_{3},\dots,\xi_{\ell}$
have same sign. Therefore, we set 
\[
b:=\frac{A_{k}}{\Omega_{k}}\quad\text{and}\quad c:=0.
\]
The estimate of $b$ follows from \eqref{eq:int-by-pts-space-a} and
$\Omega_{k}\sim N_{1}N_{3}$.

\underline{Subcase II:} $\xi_{3},\dots,\xi_{\ell}$ have both positive
and negative signs. This case is similar to \cite[Proposition 6.3]{KochTataru2012}.
By symmetry, we may assume $\xi_{3}>0$ and $\xi_{4}<0$. We first
observe that 
\begin{align*}
\Omega_{k} & =\xi_{1}^{2}-\xi_{2}^{2}+\xi_{3}^{2}-\xi_{4}^{2}+O(N_{5}^{2})\\
 & =(\xi_{1}-\xi_{2})(\xi_{1}+\xi_{2})+(\xi_{3}-\xi_{4})(-\xi_{1}-\xi_{2}+O(N_{5}))+O(N_{5}^{2})\\
 & =(\xi_{1}-\xi_{2}-\xi_{3}+\xi_{4})(\xi_{1}+\xi_{2})+O(N_{3}N_{5})\\
 & =-2(\xi_{1}+\xi_{4})(\xi_{1}+\xi_{2})+O(N_{3}N_{5}).
\end{align*}
We define a smooth function (because $a$ is even and smooth) 
\[
q(\xi,\eta):=\frac{a(\xi)\xi+a(\eta)\eta}{\xi+\eta}\text{ such that }|q(\xi,\eta)|\lesssim a(\max\{|\xi|,|\eta|\})
\]
and observe 
\begin{align*}
a(\xi_{1})\xi_{1}+a(\xi_{2})\xi_{2} & =q(\xi_{1},\xi_{2})(\xi_{1}+\xi_{2})=-\frac{q(\xi_{1},\xi_{2})}{2(\xi_{1}+\xi_{4})}(\Omega_{k}+O(N_{3}N_{5})),\\
a(\xi_{3})\xi_{3}+a(\xi_{4})\xi_{4} & =-q(\xi_{3},\xi_{4})(\xi_{1}+\xi_{2}+O(N_{5}))=\frac{q(\xi_{3},\xi_{4})}{2(\xi_{1}+\xi_{4})}(\Omega_{k}+O(N_{1}N_{5})).
\end{align*}
Thus 
\begin{align*}
A_{k} & =a(\xi_{1})\xi_{1}+\cdots+a(\xi_{4})\xi_{4}+O(a(N_{5})N_{5})\\
 & =\frac{q(\xi_{3},\xi_{4})-q(\xi_{1},\xi_{2})}{2(\xi_{1}+\xi_{4})}\Omega_{k}+O(a(N_{1})N_{1}^{-1}N_{3}N_{5}+a(N_{3})N_{5}).
\end{align*}
Therefore, we set 
\[
b:=\frac{q(\xi_{3},\xi_{4})-q(\xi_{1},\xi_{2})}{2(\xi_{1}+\xi_{4})}\quad\text{and}\quad c:= A_{k}-b\Omega_{k}.
\]
The estimate $|b|\lesssim a(N_{1})N_{1}^{-1}$ follows from $|q(\xi_{3},\xi_{4})|,|q(\xi_{1},\xi_{2})|\lesssim a(N_{1})$.
The estimate for $c$ is immediate.

\textbf{Case C:} $N_{1}\sim N_{2}\sim N_{3}\sim N_{4}$.

\underline{Subcase I:} $N_{4}\gg N_{5}$. Due to otherwise impossible
interactions and symmetry, it suffices to consider the following two
cases: the four highest frequencies contain (i) three positive and
one negative frequencies (say $\xi_{1},\xi_{2},\xi_{3}>0$ and $\xi_{4}<0$);
(ii) two positive and two negative frequencies (say $\xi_{1},\xi_{3}>0$
and $\xi_{2},\xi_{4}<0$).

The case (i) is nonresonant: 
\begin{align*}
\Omega_{k} & =\xi_{1}^{2}+\xi_{2}^{2}+\xi_{3}^{2}-\xi_{4}^{2}+O(N_{5}^{2})\\
 & =\xi_{1}^{2}+\xi_{2}^{2}+\xi_{3}^{2}-(\xi_{1}+\xi_{2}+\xi_{3}+O(N_{5}))^{2}+O(N_{5}^{2})\\
 & =-2(\xi_{1}\xi_{2}+\xi_{2}\xi_{3}+\xi_{3}\xi_{1})+O(N_{1}N_{5})\sim N_{1}^{2},
\end{align*}
where in the last inequality we used the fact that $\xi_{1},\xi_{2},\xi_{3}$
have same sign. Therefore, we set 
\[
b:=\frac{A_{k}}{\Omega_{k}}\quad\text{and}\quad c:=0.
\]
The estimate $|b|\lesssim a(N_{1})N_{1}^{-1}$ follows from $|A_{k}|\lesssim a(N_{1})N_{1}$
and $\Omega_{k}\sim N_{1}^{2}$.

The case (ii) is similar to \cite[Proposition 6.3]{KochTataru2012}. We note that $a(\xi_5) \xi_5 + \ldots + a(\xi_{k+1}) \xi_{k+1}$ can be absorbed into $c$. In the following we focus on $a(\xi_1) \xi_1 + \ldots + a(\xi_4) \xi_4$. The decomposition of \cite[Proposition 6.3]{KochTataru2012} provides an extension off the diagonal $\{ \xi_1+ \xi_2 + \xi_3+ \xi_4 = 0 \}$:
\begin{equation}
\label{eq:DecompositionKochTataru}
a(\xi_1) \xi_1+ \ldots + a(\xi_4) \xi_4 = b_4(\xi_1,\xi_2,\xi_4)(\xi_1^2 - \xi_2^2 + \xi_3^2 - \xi_4^2) 
\end{equation}
on $\{ \xi_1 + \xi_2 + \xi_3 + \xi_4 = 0 \}$. The function satisfies the bounds
\begin{equation*}
|b_4(\xi_1,\ldots,\xi_4)| \lesssim a(N_1) N_1^{-1}.
\end{equation*}
In the present context, we have
\begin{equation*}
\begin{split}
b_4(\xi_1,\xi_2,\xi_4) &= \frac{a(\xi_1) \xi_1 + a(\xi_2) \xi_2}{2(\xi_1+\xi_2)(\xi_1+\xi_4)} + \frac{a(-\xi_1-\xi_2-\xi_4)(-\xi_1-\xi_2-\xi_4) + a(\xi_4)\xi_4}{2(\xi_1+\xi_2)(\xi_1+\xi_4)} \\
&= \frac{q(\xi_1,\xi_2)}{2(\xi_1+\xi_4)} - \frac{q(\xi_4,-\xi_1-\xi_2-\xi_4)}{2(\xi_1+\xi_4)} \\
&= \frac{q(-\xi_1,-\xi_2) - q(-\xi_1+(\xi_1+\xi_4),-\xi_2-(\xi_1+\xi_4))}{2(\xi_1+\xi_4)}.
\end{split}
\end{equation*}

We want to use this decomposition for higher nonlinearity to find an off-diagonal extension: Let $\xi_3^* = \xi_3 + \xi_5 + \ldots + \xi_{k+1} = \xi_3 + \eta$ such that
\begin{equation*}
a(\xi_1) \xi_1 + a(\xi_2) \xi_2 + a(\xi_3) \xi_3 + a(\xi_4) \xi_4 = a(\xi_1) \xi_1 + a(\xi_2) \xi_2 + a(\xi_3^*) \xi_3^* + a(\xi_4) \xi_4 + f(\xi_1,\ldots,\xi_{k+1})
\end{equation*}
with $|f(\xi_1,\ldots,\xi_{k+1})| \lesssim a(N_1) N_5$ by Taylor's formula. We apply the decomposition \eqref{eq:DecompositionKochTataru} to find
\begin{equation*}
a(\xi_1) \xi_1 + a(\xi_2) \xi_2 + a(\xi_3^*) \xi_3^* + a(\xi_4) \xi_4  = b_4(\xi_1,\xi_2,\xi_4) (\xi_1^2 - \xi_2^2 + (\xi_3^*)^2 - \xi_4^2)
\end{equation*}
on $\{ \xi_1 + \xi_2 + \xi_3^* + \xi_4 = 0 \}$. We invert the substitution to find
\begin{equation*}
\begin{split}
\; &= b_4(\xi_1,\xi_2,\xi_4)(\xi_1^2 - \xi_2^2 + \xi_3^2 - \xi_4^2 + 2 \xi_3 \eta + \eta^2) \\
\; &= b_4(\xi_1,\xi_2,\xi_4)(\Omega_{k}) + g(\xi_1,\ldots,\xi_{k+1}).
\end{split}
\end{equation*}
Conclusively,
\begin{equation*}
\begin{split}
\sum_{j=1}^{k+1} a(\xi_j) \xi_j &= b_4(\xi_1,\xi_2,\xi_4) \Omega_{k+1} + g(\xi_1,\ldots,\xi_{k+1}) + f(\xi_1,\ldots,\xi_{k+1}) + \sum_{j=5}^{k+1} a(\xi_j) \xi_j \\
&= b(\xi_1,\xi_2,\xi_4) \Omega_{k} + c(\xi_1,\ldots,\xi_{k+1})
\end{split}
\end{equation*}
provides an off-diagonal extension with the claimed size and regularity.


\underline{Subcase II:} $N_{4}\sim N_{5}$. In this case, we set
\[
b:=0\quad\text{and}\qquad c:= A_{k}.
\]
The estimate of $c$ follows from the crude estimate $|A_{k}|\lesssim a(N_{1})N_{1}$.

This completes the proof.
\end{proof}

\subsection{\label{subsec:Contribution-bOmega}Contribution of $b\Omega_{k}$}

The goal of this subsection is to estimate the contribution of $b\Omega_{k}$
in $R_{k}(u,t)$: Lemma \ref{lem:Contribution-bOmega}. We integrate
by parts in time. The case $k=3$ is detailed in \cite{Schippa2017MBO}.
Here we deal with general $k$.
\begin{lem}[Contribution of $b\Omega_{k}$]
\label{lem:Contribution-bOmega}Let $T\in(0,1]$ and $0<s_{1}<1/2\leq s_{2}$.
For sufficiently small $\epsilon>0$, we find 
\begin{multline}
\sup_{t\in[0,T]}\sum_{N_{1}\sim N_{2}\gtrsim\cdots\gtrsim N_{k+1}}\bigg|\int_{0}^{t}\int_{\Gamma_{\lambda}^{k+1}}b(\xi_{1},\dots,\xi_{k+1})\Omega_{k}(\xi_{1},\dots,\xi_{k+1})\\
\times\prod_{i=1}^{k+1}\chi_{N_{i}}(\xi_{i})\hat{u}(s,\xi_{i})\,d\Gamma_{\lambda}^{k+1}ds\bigg|\lesssim N^{-\epsilon}\|u\|_{F_{\lambda}^{\overline{s}(\epsilon)+s'}}^{2}(\|u\|_{F_{\lambda}^{s_{1}+\epsilon,1/2-\epsilon}}^{k-1}+\|u\|_{F_{\lambda}^{s_{1}+\epsilon,1/2-\epsilon}}^{2k-2}),\label{eq:Contribution-bOmega-LHS}
\end{multline}
where we recall that $\overline{s}=(s_{1},s_{2})$ and $\overline{s}(\epsilon)=(s_{1}+\epsilon,s_{2}-\epsilon)$.
\end{lem}

\begin{rem}
We remark that $s_{2}=1/2$ is included in Lemma \ref{lem:Contribution-bOmega}.
Moreover, it is possible to lower the threshold for $s_{2}$ slightly.
Thus the contribution of $b\Omega_{k}$ is \emph{not} the source for
the restriction $s_{2}>1/2$ in our energy estimate (Proposition \ref{prop:EnergyEstimateSolutions}).
\end{rem}

We first estimate each Littlewood-Paley piece.
\begin{lem}[Integration by parts in time]
\label{lem:int-by-pts-time-apriori}Let $T\in(0,1]$. We have the
following integration by parts in time estimate: 
\begin{align*}
 & \sup_{t\in[0,T]}\bigg|\int_{0}^{t}\int_{\Gamma_{\lambda}^{k+1}}b(\xi_{1},\dots,\xi_{k+1})\Omega_{k}(\xi_{1},\dots,\xi_{k+1})\prod_{i=1}^{k+1}\chi_{N_{i}}(\xi_{i})\hat{u}(s,\xi_{i})\,d\Gamma_{k+1}ds\bigg|\\
 & \lesssim a(N_{1})N_{1}^{-1}\Big(\prod_{i=3}^{k+1}N_{i}^{1/2}\Big)\Big(\prod_{i=1}^{k+1}\|P_{N_{i}}u\|_{L_{t}^{\infty}([0,T],L_{\lambda}^{2})}\Big)\\
 & \quad+a(N_{1})N_{1}^{-1}N_{3}\sup_{t\in[0,T]}\big|\int_{0}^{t}\int_{\lambda\T}P_{N_{j}}(u^{k})\prod_{i\neq j}P_{N_{i}}udxds\big|.
\end{align*}
\end{lem}

\begin{rem}[Separation of variables]
\label{rem:SeparationVariables}In what follows, we need to estimate
the expression of the form $\int_{\Gamma_{\lambda}^{k+1}}b(\xi_{1},\dots,\xi_{k+1})\prod_{i=1}^{k+1}\hat{f_{i}}(\xi_{i})d\Gamma_{\lambda}^{k+1}$.
We will simply estimate this by $\|b\|_{L^{\infty}}\int_{\lambda\T}\prod_{i=1}^{k+1}f_{i}(x)dx$.
By the symbol regularity \eqref{eq:SymbolEstimate-b}, $b(\xi_{1},\dots,\xi_{k+1})$
has a rapidly converging Fourier series, say $\sum b_{\lambda_{1},\dots,\lambda_{k+1}}e^{i(\lambda_{1}\xi_{1}+\cdots+\lambda_{k+1}\xi_{k+1})}$
with rapidly decaying coefficients $b_{\lambda_{1},\dots,\lambda_{k+1}}$,
where the summation ranges over $\lambda_{i}$'s in $N_{i}^{-1}\Z$.
As the modulation on the Fourier space corresponds to the translation
on the physical space, estimating $\int_{\Gamma_{k+1}}b(\xi_{1},\dots,\xi_{k+1})\prod_{i=1}^{k+1}\hat{f_{i}}(\xi_{i})d\Gamma_{\lambda}^{k+1}$
essentially reduces to estimating $\|b\|_{L^{\infty}}\int_{\lambda\T}\prod_{i=1}^{k+1}f_{i}(x)dx$,
as long as our estimates are translation invariant. This can be justified
in this paper because we rely on frequency interactions, and our arguments
do not depend on how functions are distributed in physical space.
For a detailed discussion, see \cite{Taylor1974} and \cite{ChristHolmerTataru2012}
in the context of dispersive equations. 
\end{rem}

\begin{proof}[Proof of Lemma \ref{lem:int-by-pts-time-apriori}]
Changing to the interaction picture by introducing 
\[
\hat{v}(s,\xi):= e^{is|\xi|\xi}\hat{u}(s,\xi),
\]
we need to estimate 
\begin{equation}
\int_{0}^{t}\int_{\Gamma_{\lambda}^{k+1}}e^{-is\Omega_{k}(\xi_{1},\ldots,\xi_{k+1})}b(\xi_{1},\dots,\xi_{k+1})\Omega_{k}(\xi_{1},\dots,\xi_{k+1})\prod_{i=1}^{k+1}\chi_{N_{i}}(\xi_{i})\hat{v}(s,\xi_{i})\,d\Gamma_{\lambda}^{k+1}ds.\label{eq:int-by-pts-time-apriori-temp1}
\end{equation}
Integrating by parts in time using 
\begin{align*}
\partial_{s}e^{-is\Omega_{k}} & =-i\Omega_{k}e^{is\Omega_{k}},\\
\partial_{s}\hat{v}(s,\xi) & =(i\xi)e^{is|\xi|\xi}(u^{k}(s,\cdot))^{\wedge}(\xi),
\end{align*}
we find 
\begin{align*}
\eqref{eq:int-by-pts-time-apriori-temp1} & =\int_{\Gamma_{\lambda}^{k+1}}b(\xi_{1},\dots,\xi_{k+1})\prod_{i=1}^{k+1}\chi_{N_{i}}(\xi_{i})\hat{u}(s,\xi_{i})d\Gamma_{\lambda}^{k+1}\bigg|_{0}^{t}\\
 & \quad+\sum_{j=1}^{k+1}\int_{0}^{t}\int_{\Gamma_{\lambda}^{k+1}}b(\xi_{1},\dots,\xi_{k+1})\prod_{i=1}^{k+1}\chi_{N_{i}}(\xi_{i})\cdot\xi_{j}(u^{k})^{\wedge}(s,\xi_{j})\prod_{i\neq j}\hat{u}(s,\xi_{i})d\Gamma_{\lambda}^{k+1}ds.
\end{align*}
We call the first term the \emph{boundary term} and the second term
the \emph{spacetime term}. We turn to the estimates.

We firstly estimate the boundary term. By Remark \ref{rem:SeparationVariables},
we may replace $b$ by $a(N_{1})N_{1}^{-1}$ and return to the physical
space representation: 
\[
a(N_{1})N_{1}^{-1}\int_{\lambda\T}\prod_{i=1}^{k+1}P_{N_{i}}u\,dx\bigg|_{0}^{t}.
\]
Applying $L_{x}^{2}$ to the highest two frequencies and pointwise
bounds to the remaining ones, the boundary term is estimated by 
\[
a(N_{1})N_{1}^{-1}\Big(\prod_{i=3}^{k+1}N_{i}^{1/2}\Big)\prod_{i=1}^{k+1}\|P_{N_{i}}u\|_{L_{t}^{\infty}([0,T],L_{\lambda}^{2})}.
\]

Next, we estimate the spacetime term. We focus on the $\sum_{j=1}^{2}$-part
of the spacetime term. We rewrite the $\sum_{j=1}^{2}$-part as 
\begin{align*}
 & \int_{0}^{t}\int_{\Gamma_{\lambda}^{2k}}\Big(b(\xi_{sum},\xi_{1},\xi_{2},\dots,\xi_{k})\xi_{sum}\chi_{N_{1}}(\xi_{sum})\chi_{N_{2}}(\xi_{1})\\
 & \quad+b(\xi_{1},\xi_{sum},\xi_{2},\dots,\xi_{k})\xi_{sum}\chi_{N_{1}}(\xi_{1})\chi_{N_{2}}(\xi_{sum})\Big)\Big(\prod_{i=3}^{k+1}\chi_{N_{i}}(\xi_{i-1})\Big)\prod_{i=1}^{2k}\hat{u}(s,\xi_{i})ds,
\end{align*}
where we denote $\xi_{sum}=\xi_{k+1}+\cdots+\xi_{2k}$. We then \emph{integrate
by parts in space} as in \eqref{eq:int-by-pts-space-a}; the following
holds after taking $\int_{\Gamma_{2k}}(\cdot)\prod_{i=3}^{k+1}\chi_{N_{i}}(\xi_{i-1})\prod_{i=1}^{2k}\hat{u}(s,\xi_{i})d\Gamma_{2k}$:
\begin{align*}
 & \phantom{=}b(\xi_{sum},\xi_{1},\xi_{2},\dots,\xi_{k})\xi_{sum}\chi_{N_{1}}(\xi_{sum})\chi_{N_{2}}(\xi_{1})\\
 & \qquad+b(\xi_{1},\xi_{sum},\xi_{2},\dots,\xi_{k})\xi_{sum}\chi_{N_{1}}(\xi_{1})\chi_{N_{2}}(\xi_{sum})\\
 & =b(\xi_{sum},\xi_{1},\xi_{2},\dots,\xi_{k})(\xi_{1}+\xi_{sum})\chi_{N_{1}}(\xi_{sum})\chi_{N_{2}}(\xi_{1})\\
 & =-b(\xi_{sum},\xi_{1},\xi_{2},\dots,\xi_{k})(\xi_{2}+\cdots+\xi_{k})\chi_{N_{1}}(\xi_{sum})\chi_{N_{2}}(\xi_{1})\\
 & =O(N_{3})\cdot b(\xi_{sum},\xi_{1},\xi_{2},\dots,\xi_{k}))\chi_{N_{1}}(\xi_{sum})\chi_{N_{2}}(\xi_{1}),
\end{align*}
where we changed the variables $\xi_{1}\leftrightarrow\xi_{sum}$
and used frequency localization $|\xi_{i-1}|\sim N_{i}$ for $3\leq i\leq k+1$.
For the $\sum_{j=3}^{k+1}$-part, we simply bound $\xi_{j}$ by $N_{3}$.

As a result, the spacetime term is estimated by 
\[
a(N_{1})N_{1}^{-1}N_{3}\bigg|\int_{0}^{t}\int_{\lambda\T}P_{N_{j}}(u^{k})\prod_{i\neq j}P_{N_{i}}u\,dxds\bigg|
\]
by Remark \ref{rem:SeparationVariables}. The proof is completed.
\end{proof}
For $s_{2}>1/2$, we can estimate the spacetime term with $L_{x}^{2}$
in the highest frequencies and pointwise bounds for the remaining
ones. We then give $L_{t}^{1}$ to $P_{N_{j}}(u^{k})$ and apply Lemma
\ref{lem:decomp-prod} with $s'=0$. Then, the spacetime term is estimated
by 
\[
\|u\|_{F_{\lambda}^{s_{1}+\varepsilon,s_{2}-\varepsilon}}^{k-1}a(N_{1})N_{1}^{-1}N_{3}\Big(\prod_{i=3}^{k+1}N_{i}^{1/2}\Big)\prod_{i=1}^{k+1}c_{N_{i}}^{(u,(s_{1}+\varepsilon,s_{2}-\varepsilon))},
\]
with $\epsilon(\overline{s})>0$ sufficiently small. For $s_{2}=1/2$,
this argument fails as it does not allow for the use of multilinear
estimates involving factors from $P_{N_{j}}(u^{k})$ and $P_{N_{i}}u$,
$i\neq j$. This shortcoming is remedied in the following:
\begin{lem}[Estimates of spacetime term]
\label{lem:MultilinearSpaceTimeEstimate} Let $j\in\{1,\dots k+1\}$
and 
\[
R_{2k,j}(N_{1},\ldots,N_{k+1},M_{1},\ldots,M_{k})=\bigg|\int_{0}^{t}\int_{\lambda\T}P_{N_{j}}\big(\prod_{i=1}^{k}P_{M_{i}}u\big)\prod_{i\neq j}P_{N_{i}}udxds\bigg|.
\]
Let $K_{1},\ldots,K_{2k}$ denote a decreasing rearrangement of $M_{1},\ldots,M_{k},N_{i}\,(i\neq j)$
with multiplicity. Then, for $K_{1}\sim K_{2}\gtrsim1$, we find the
following estimate to hold: 
\[
R_{2k,j}(N_{1},\ldots,N_{k+1},M_{1},\ldots,M_{k})\lesssim C(K_{1},\ldots,K_{2k})\prod_{i=1}^{2k}\Vert P_{K_{i}}u\Vert_{F_{\lambda}^{0}},
\]
where 
\begin{equation}
\begin{aligned} & C(K_{1},\dots,K_{2k})\\
 & =\begin{cases}
K_{3}^{1/2,0}K_{4}^{1/2,0}\prod_{i=5}^{2k}K_{i}^{1/2} & \text{if }K_{2}\gg K_{3}\text{ or }K_{1}\sim K_{3}\gg K_{4},\\
K_{1}^{1/2}K_{5}^{1/2,0}\prod_{i=6}^{2k}K_{i}^{1/2} & \text{if }K_{1}\sim K_{4}\gg K_{6}\text{ or }K_{1}\sim K_{5}\gg K_{6},\\
K_{1}\prod_{i=7}^{2k}K_{i}^{1/2} & \text{if }K_{1}\sim K_{6}.
\end{cases}
\end{aligned}
\label{eq:MultilinearSpaceTimeEstimateTemp1}
\end{equation}
\end{lem}

\begin{proof}
As explained in Remark \ref{rem:SeparationVariables}, we can dispose
of $P_{N_{j}}$ by expanding the Fourier multiplier into a rapidly
converging Fourier series. The resulting expression 
\[
\int_{0}^{t}\int_{\lambda\T}P_{K_{1}}u\ldots P_{K_{2k}}udxds
\]
is estimated by linear and bilinear Strichartz estimates after partitioning
$[0,t]$ into $O(K_{1}^{-1})$ time intervals of length $K_{1}^{-1}$:
\begin{enumerate}
\item If $K_{1}\sim K_{2}\gg K_{3}$ or $K_{1}\sim K_{3}\gg K_{4}$, following
the proof of Corollary \ref{cor:HighLowShorttimeEstimates} we can
apply two bilinear Strichartz estimates involving $K_{1},\ldots,K_{4}$.
This gives the first bound.
\item If $K_{1}\sim K_{4}\gg K_{5}$ or $K_{1}\sim K_{5}\gg K_{6}$, we
can apply one bilinear Strichartz estimate and three $L_{t,x}^{6}$-Strichartz
estimates involving $K_{1},\dots,K_{5}$. Pointwise bounds for the
lower frequencies give the second estimate.
\item If $K_{1}\sim K_{6}$, the claim follows from six linear $L_{t,x}^{6}$-Strichartz
estimates.
\end{enumerate}
\end{proof}
\begin{proof}[Proof of Lemma \ref{lem:Contribution-bOmega}]
By Lemma \ref{lem:int-by-pts-time-apriori}, it suffices to estimate
the summation of boundary terms and spacetime terms. For the boundary
terms, we use 
\[
a(N_{1})N_{1}^{-1}\Big(\prod_{i=3}^{k+1}N_{i}^{1/2}\Big)\lesssim N^{-\epsilon}N_{1}^{\frac{1}{2},s_{2}+s'-2\epsilon}N_{2}^{\frac{1}{2},s_{2}+s'-2\epsilon}\prod_{i=3}^{k+1}N_{i}^{\frac{1}{2},\frac{1}{2}-2\varepsilon}
\]
to find 
\begin{align*}
 & \ \sum_{N_{1}\sim N_{2}\gtrsim\cdots\gtrsim N_{k+1}}a(N_{1})N_{1}^{-1}\Big(\prod_{i=3}^{k+1}N_{i}^{1/2}\Big)\Big(\prod_{i=1}^{k+1}\|P_{N_{i}}u\|_{L_{t}^{\infty}([0,T],L_{\lambda}^{2})}\Big)\\
 & \lesssim N^{-\epsilon}\|u\|_{F_{\lambda}^{\overline{s}(\epsilon)+s'}}^{2}\|u\|_{F_{\lambda}^{s_{1}+\epsilon,1/2-\epsilon}}^{k-1}.
\end{align*}

For the spacetime terms, by Lemma \ref{lem:MultilinearSpaceTimeEstimate},
it suffices to estimate 
\begin{equation}
\sum_{\substack{K_{1}\sim K_{2}\gtrsim\cdots\gtrsim K_{2k}\\
K_{1}\sim K_{2}\gtrsim1
}
}N^{2(s_{2}+s')-1}K_{3}C(K_{1},\dots,K_{2k})\prod_{i=1}^{2k}\Vert P_{K_{i}}u\Vert_{F_{\lambda}^{0}},\label{eq:Lemma5.5temp1}
\end{equation}
where $C(K_{1},\dots,K_{2k})$ is given in \eqref{eq:MultilinearSpaceTimeEstimateTemp1},
and we used $a(N_{1})N_{1}^{-1}\lesssim N^{2s_{2}-1}$ and $K_{1}\sim K_{2}\gtrsim N\gtrsim1$.
Due to \eqref{eq:MultilinearSpaceTimeEstimateTemp1}, we have 
\[
N^{2(s_{2}+s')-1}K_{3}C(K_{1},\dots,K_{2k})\lesssim N^{-\epsilon}K_{1}^{s_{2}+s'-2\epsilon}K_{2}^{s_{2}+s'-2\epsilon}\prod_{i=3}^{2k}K_{i}^{\frac{1}{2},\frac{1}{2}-2\epsilon},
\]
provided that $\epsilon$ is sufficiently small. Therefore, 
\[
\eqref{eq:Lemma5.5temp1}\lesssim N^{-\epsilon}\|u\|_{F_{\lambda}^{\overline{s}(\epsilon)+s'}}^{2}\|u\|_{F_{\lambda}^{s_{1}+\epsilon,1/2-\epsilon}}^{2k-2}.
\]
This completes the proof.
\end{proof}

\subsection{\label{subsec:Contribution-c}Contribution of $c$}

\begin{lem}[Contribution of $c$]
\label{lem:Contribution-c}Let $T\in(0,1]$ and $0<s_{1}<1/2$.
\begin{itemize}
\item If $s_{2}>1/2$, then there exists $\epsilon=\epsilon(\overline{s})>0$
such that we find 
\[
\begin{split} & \;\sum_{\substack{N_{1}\sim N_{2}\gtrsim\cdots\gtrsim N_{k+1}\\
N_{1}\sim N_{2}\gtrsim N
}
}\bigg|\int_{0}^{T}\int_{\Gamma_{\lambda}^{k+1}}c(\xi_{1},\dots,\xi_{k+1})\prod_{i=1}^{k+1}\chi_{N_{i}}(\xi_{i})\hat{u}(s,\xi_{i})\,d\Gamma_{\lambda}^{k+1}ds\bigg|\\
 & \lesssim N^{-\epsilon}\|u\|_{F_{\lambda}^{\overline{s}(\epsilon)+s'}}^{2}\|u\|_{F_{\lambda}^{\overline{s}(\epsilon)}}^{k-1}.
\end{split}
\]
\item If $k=4$ and $s_{2}=1/2$, we find 
\[
\begin{split} & \;\sum_{\substack{N_{1}\sim N_{2}\gtrsim\cdots\gtrsim N_{k+1}\\
N_{1}\sim N_{2}\gtrsim N
}
}\bigg|\int_{0}^{T}\int_{\Gamma_{\lambda}^{5}}c(\xi_{1},\dots,\xi_{k+1})\prod_{i=1}^{5}\chi_{N_{i}}(\xi_{i})\hat{u}(s,\xi_{i})\,d\Gamma_{\lambda}^{5}ds\bigg|\\
 & \lesssim\|u\|_{F_{\lambda}^{a_{N}}}^{2}\|u\|_{F_{\lambda}^{\overline{s}}}^{3}.
\end{split}
\]
\end{itemize}
Here, we recall that the $F_{\lambda}^{a_{N}}$-norm is defined in
\eqref{eq:F-aN-def}.
\end{lem}

\begin{proof}
By Lemma \ref{lem:decomposition-Ak}, $c$ is nonzero only when $N_{3}\sim N_{4}$.
By Remark \ref{rem:SeparationVariables}, we may replace $c$ by $(a(N_{1})N_{1}^{-1}N_{3}+a(N_{3}))N_{5}$
and change back to the physical space representation. $a=a_{N}$ enables
us to restrict to $N_{1}\sim N_{2}\gtrsim N\gtrsim1$. Thus it suffices
to estimate 
\[
\sum_{\substack{N_{1}\sim N_{2}\gtrsim N_{3}\sim N_{4},\\
N_{5}\gtrsim\cdots\gtrsim N_{k+1},\\
N_{1}\sim N_{2}\gtrsim1
}
}(a(N_{1})N_{1}^{-1}N_{3}+a(N_{3}))N_{5}\bigg|\int_{0}^{T}\int_{\lambda\T}\prod_{i=1}^{k+1}P_{N_{i}}u\,dxds\bigg|.
\]
We also note that the proof of Lemma \ref{lem:MultilinearSpaceTimeEstimate}
implies 
\[
\bigg|\int_{0}^{T}\int_{\lambda\T}\prod_{i=1}^{k+1}P_{N_{i}}u\,dxds\bigg|\lesssim\prod_{i=1}^{4}\|P_{N_{i}}u\|_{F_{\lambda}^{0}}\prod_{i=5}^{k+1}N_{i}^{1/2}\|P_{N_{i}}u\|_{F_{\lambda}^{0}}.
\]
Therefore, it suffices to estimate 
\begin{equation}
\sum_{\substack{N_{1}\sim N_{2}\gtrsim N_{3}\sim N_{4},\\
N_{5}\gtrsim\cdots\gtrsim N_{k+1},\\
N_{1}\sim N_{2}\gtrsim1
}
}(a(N_{1})N_{1}^{-1}N_{3}+a(N_{3}))N_{5}\prod_{i=5}^{k+1}N_{i}^{1/2}\prod_{i=1}^{k+1}\|P_{N_{i}}u\|_{F_{\lambda}^{0}}.\label{eq:Contribution-c-temp1}
\end{equation}

If $s_{2}>1/2$, then we find 
\[
(a(N_{1})N_{1}^{-1}N_{3}+a(N_{3}))N_{5}\prod_{i=5}^{k+1}N_{i}^{1/2}\lesssim N^{-\epsilon}N_{1}^{1/2,s_{2}+s'-2\varepsilon}N_{2}^{1/2,s_{2}+s'-2\varepsilon}\prod_{i=3}^{k+1}N_{i}^{1/2,s_{2}-2\varepsilon}.
\]
Therefore, 
\[
\eqref{eq:Contribution-c-temp1}\lesssim N^{-\epsilon}\|u\|_{F_{\lambda}^{\overline{s}(\epsilon)+s'}}^{2}\|u\|_{F_{\lambda}^{\overline{s}(\epsilon)}}^{k-1}.
\]
If $k=4$ and $s_{2}=1/2$, then we find 
\[
(a(N_{1})N_{1}^{-1}N_{3}+a(N_{3}))N_{5}^{3/2}\lesssim a(N_{1})N_{5}^{3/2}\lesssim N^{2s'+1}N_{5}^{3/2}.
\]
Therefore,
\begin{align*}
\eqref{eq:Contribution-c-temp1} & \lesssim\sum_{N_{1}\sim N_{2}\gtrsim N_{3}\sim N_{4}\gtrsim N_{5}}N^{2s'+1}N_{5}^{3/2}\prod_{i=1}^{5}\Vert P_{N_{i}}u\Vert_{F_{\lambda}^{0}}\\
 & \lesssim\|u\|_{F_{\lambda}^{\overline{s}}}\sum_{N_{1}\sim N_{2}\gtrsim N_{3}\sim N_{4}}N^{2s'+1}N_{4}\prod_{i=1}^{4}\Vert P_{N_{i}}u\Vert_{F_{\lambda}^{0}}\lesssim\|u\|_{F_{\lambda}^{\overline{s}+s'}}^{2}\|u\|_{F_{\lambda}^{\overline{s}}}^{3}
\end{align*}

%
%
\end{proof}
The proof of Lemma \ref{lem:ReductionEnergyEstimates} (and hence
that of Proposition \ref{prop:EnergyEstimateSolutions}) is now completed
by Lemmas \ref{lem:Contribution-bOmega} and \ref{lem:Contribution-c}.

\section{\label{sec:DifferenceEstimate}Estimates for differences of solutions}

The goal of this section is to estimate the differences of solutions:
Proposition \ref{prop:EnergyEstimatesDifferences}. We firstly prove
the Lipschitz continuity in a weaker topology $H_{\lambda}^{\overline{s}-1}$
for $H_{\lambda}^{\overline{s}}$-solutions \eqref{eq:LipschitzContinuityNegativeSobolevSpaces}.
Furthermore, we bound the $H_{\lambda}^{\overline{s}}$-difference
of solutions by the $H_{\lambda}^{\overline{s}-1}$-difference estimate
and the $H_{\lambda}^{\overline{s}+1}$-a priori bound \eqref{eq:HsContinuityBonaSmithRegularities},
as usual for the Bona-Smith approximation (cf. \cite[Section~4]{IonescuKenigTataru2008}). 
\begin{prop}[Energy estimates for differences of solutions]
\label{prop:EnergyEstimatesDifferences}Let $\overline{s}=(s_{1},s_{2})$
with $0<s_{1}<1/2$ and $s_{2}\geq3/4$. Let $u_{1}$ and $u_{2}$
be smooth solutions to \eqref{eq:GeneralizedBenjaminOnoEquation}
defined on $[0,T]\subseteq[0,1]$. Set $v=u_{1}-u_{2}$. Then, we
find the following estimates to hold:
\begin{enumerate}
\item Lipschitz continuity in $H_{\lambda}^{\overline{s}-1}$ for $H_{\lambda}^{\overline{s}}$
solutions: 
\begin{align}
\Vert v\Vert_{E_{\lambda}^{\overline{s}-1}(T)}^{2} & \lesssim\Vert v(0)\Vert_{H_{\lambda}^{\overline{s}-1}}^{2}+\Vert v\Vert_{F_{\lambda}^{\overline{s}-1}}^{2}(\Vert u_{1}\Vert_{F_{\lambda}^{\overline{s}}}+\Vert u_{2}\Vert_{F_{\lambda}^{\overline{s}}})^{k-1}\label{eq:LipschitzContinuityNegativeSobolevSpaces}\\
 & \quad+\Vert v\Vert_{F_{\lambda}^{\overline{s}-1}}^{2}(\Vert u_{1}\Vert_{F_{\lambda}^{\overline{s}}}+\Vert u_{2}\Vert_{F_{\lambda}^{\overline{s}}})^{2k-2}.\nonumber 
\end{align}
\item Continuity in $H_{\lambda}^{\overline{s}}$: 
\begin{align}
\Vert v\Vert_{E_{\lambda}^{\overline{s}}(T)}^{2} & \lesssim\Vert v(0)\Vert_{H_{\lambda}^{\overline{s}}}^{2}+\Vert v\Vert_{F_{\lambda}^{\overline{s}}}^{2}(\Vert v\Vert_{F_{\lambda}^{\overline{s}}}+\Vert u_{2}\Vert_{F_{\lambda}^{\overline{s}}})^{k-1}\label{eq:HsContinuityBonaSmithRegularities}\\
 & \quad+\Vert v\Vert_{F_{\lambda}^{\overline{s}}}\Vert v\Vert_{F_{\lambda}^{\overline{s}-1}}\Vert u_{2}\Vert_{F_{\lambda}^{\overline{s}+1}}(\Vert v\Vert_{F_{\lambda}^{\overline{s}}}+\Vert u_{2}\Vert_{F_{\lambda}^{\overline{s}}})^{k-2}\nonumber \\
 & \quad+\Vert v\Vert_{F_{\lambda}^{\overline{s}}}^{2}(\Vert v\Vert_{F_{\lambda}^{\overline{s}}}+\Vert u_{2}\Vert_{F_{\lambda}^{\overline{s}}})^{2k-2}\nonumber \\
 & \quad+\Vert v\Vert_{F_{\lambda}^{\overline{s}}}\Vert v\Vert_{F_{\lambda}^{\overline{s}-1}}\Vert u_{2}\Vert_{F_{\lambda}^{\overline{s}+1}}(\Vert v\Vert_{F_{\lambda}^{\overline{s}}}+\Vert u_{2}\Vert_{F_{\lambda}^{\overline{s}}})^{2k-3}.\nonumber 
\end{align}
\end{enumerate}
\end{prop}

We start with the equation for $v$: 
\[
\partial_{t}v+\mathcal{H}\partial_{xx}v=\partial_{x}(u_{1}^{k}-u_{2}^{k}).
\]
We write $\partial_{x}(u_{1}^{k}-u_{2}^{k})$ in two ways. A standard
way of writing $\partial_{x}(u_{1}^{k}-u_{2}^{k})$ is 
\[
\partial_{x}(u_{1}^{k}-u_{2}^{k})=\partial_{x}(v(u_{1}^{k-1}+u_{1}^{k-2}u_{2}+\cdots+u_{2}^{k-1})).
\]
For simplicity of notations, let us express this as 
\begin{equation}
\partial_{x}(u_{1}^{k}-u_{2}^{k})=\partial_{x}(vu^{k-1}),\label{eq:first-expression}
\end{equation}
where $u^{k-1}$ means a linear combination of $u_{1}^{k-1},u_{1}^{k-2}u_{2},\dots,u_{2}^{k-1}$.
We use \eqref{eq:first-expression} to show \eqref{eq:LipschitzContinuityNegativeSobolevSpaces}.
However, when we show \eqref{eq:HsContinuityBonaSmithRegularities},
we express $\partial_{x}(u_{1}^{k}-u_{2}^{k})$ in another way. It
is straight-forward that there exist integers $c_{0},\dots,c_{k-2}$
and $d_{0},\dots,d_{k-3}$ such that 
\[
\partial_{x}(u_{1}^{k}-u_{2}^{k})=(\partial_{x}v)(\sum_{i=0}^{k-2}c_{i}u_{1}^{k-1-i}u_{2}^{i})+v(\partial_{x}u_{2})(\sum_{i=0}^{k-3}d_{i}u_{1}^{k-2-i}u_{2}^{i}).
\]
We compactly write this as 
\begin{equation}
\partial_{x}(u_{1}^{k}-u_{2}^{k})=(\partial_{x}v)u^{k-1}+vw^{k-1},\label{eq:second-expression}
\end{equation}
where $u^{k-1}$ means a linear combination of $u_{1}^{k-1},u_{1}^{k-2}u_{2},\dots,u_{2}^{k-1}$
as above, but $w^{k-1}$ means a linear combination of $u_{1}^{k-2}\partial_{x}u_{2},u_{1}^{k-3}u_{2}\partial_{x}u_{2},\dots,u_{2}^{k-2}\partial_{x}u_{2}$.
An advantage of using \eqref{eq:second-expression} is that $\partial_{x}$
is not applied to $u_{1}$ so that we can avoid $F^{7/4}$-norm for
$u_{1}$, as stated in \eqref{eq:HsContinuityBonaSmithRegularities}.

For \eqref{eq:LipschitzContinuityNegativeSobolevSpaces}, we take
$P_{N}$ to \eqref{eq:first-expression} for $N\geq1$ (the case $N<1$
does not require an estimate by the definition of the $E_{\lambda}^{s}$-norm),
multiply with $N^{s_{2}-1}P_{N}v$, and then integrate on $[0,t]\times\lambda\T$
to get 
\[
N^{2s_{2}-2}\|P_{N}v(t)\|_{L_{\lambda}^{2}}^{2}=N^{2s_{2}-2}\|P_{N}v(0)\|_{L_{\lambda}^{2}}^{2}+N^{2s_{2}-2}\int_{0}^{t}\int_{\lambda\T}P_{N}v\partial_{x}P_{N}(vu^{k-1})dxds
\]
for any $t\in[0,T]$ and $N$. Thus, 
\[
\|v\|_{E_{\lambda}^{\overline{s}-1}}^{2}\leq\|v(0)\|_{H_{\lambda}^{\overline{s}-1}}^{2}+\sum_{N\geq1}N^{2s_{2}-2}\sup_{t\in[0,T]}\Big|\int_{0}^{t}\int_{\lambda\T}P_{N}v\partial_{x}P_{N}(vu^{k-1})dxds\Big|.
\]
We split the integrand into the sum of Littlewood-Paley pieces: 
\begin{multline}
\sum_{N\geq1}N^{2s_{2}-2}\sup_{t\in[0,T]}\Big|\int_{0}^{t}\int_{\lambda\T}P_{N}v\partial_{x}P_{N}(vu^{k-1})dxds\Big|\\
\leq\sum_{\substack{K,K_{1},\dots,K_{k-1},\\
N\geq1
}
}N^{2s_{2}-2}\sup_{t\in[0,T]}\Big|\int_{0}^{t}\int_{\lambda\T}P_{N}v\partial_{x}P_{N}(P_{K}vP_{K_{1}}u\dots P_{K_{k-1}}u)dxds\Big|.\label{eq:s-1-expansion}
\end{multline}
Here, $N\geq1$ and $P_{K_{i}}u$ can be either $P_{K_{i}}u_{1}$
or $P_{K_{i}}u_{2}$.

For \eqref{eq:HsContinuityBonaSmithRegularities}, we similarly find
\[
\|v\|_{E_{\lambda}^{\overline{s}}}^{2}\leq\|v(0)\|_{H_{\lambda}^{\overline{s}}}^{2}+\sum_{N\geq1}N^{2s_{2}}\sup_{t\in[0,T]}\Big|\int_{0}^{t}\int_{\lambda\T}P_{N}v\partial_{x}P_{N}(vu^{k-1})dxds\Big|,
\]
and further using the expression \eqref{eq:second-expression}, 
\begin{multline}
\sum_{N\geq1}N^{2s_{2}}\sup_{t\in[0,T]}\Big|\int_{0}^{t}\int_{\lambda\T}P_{N}v\,P_{N}((\partial_{x}v)u^{k-1}+vw^{k-1})dxds\Big|\\
\leq\sum_{\substack{K,K_{1},\dots,K_{k-1},\\
N\geq1
}
}N^{2s_{2}}\bigg(\sup_{t\in[0,T]}\Big|\int_{0}^{t}\int_{\lambda\T}P_{N}vP_{N}((\partial_{x}P_{K}v)P_{K_{1}}u\dots P_{K_{k-1}}u)dxds\Big|\\
+\sup_{t\in[0,T]}\Big|\int_{0}^{t}\int_{\lambda\T}P_{N}vP_{N}(P_{K}vP_{K_{1}}w\dots P_{K_{k-1}}w)dxds\Big|\bigg).\label{eq:s-expansion}
\end{multline}
Here, $P_{K_{i}}u$ can be either $P_{K_{i}}u_{1}$ or $P_{K_{i}}u_{2}$;
$P_{K_{i}}w$ can be either $P_{K_{i}}u_{1}$, $P_{K_{i}}u_{2}$,
or $P_{K_{i}}\partial_{x}u_{2}$, but $P_{K_{i}}\partial_{x}u_{2}$
should appear \emph{exactly once} among $P_{K_{1}}w,\dots,P_{K_{k-1}}w$.

We may assume $K_{1}\geq\cdots\geq K_{k-1}$. Let $M_{1},\dots,M_{k+1}$
be the decreasing rearrangement of $N,K,K_{1},\dots,K_{k-1}$. In
particular, $M_{1}\sim M_{2}\gtrsim1$. We distinguish three cases: 
\begin{itemize}
\item $M_{1}\sim M_{2}\gtrsim M_{3}\gg M_{4}$: We treat this case in Section
\ref{subsec:diff-est-case1}. Here, we have seen in \eqref{eq:BO-like-nonresonance}
that $|\Omega_{k}|\sim M_{1}M_{3}$ and integrate by parts in time.
The estimates allow for $s_{2}>\frac{1}{2}$. 
\item $M_{1}\sim M_{2}\gg M_{3}\sim M_{4}$: We treat this case in Section
\ref{subsec:diff-est-case2}. We do not integrate by parts in time,
but apply two bilinear Strichartz estimates to the four highest frequencies.
The estimates allow for $s_{2}>\frac{1}{2}$. 
\item $M_{1}\sim M_{2}\sim M_{3}\sim M_{4}$: We treat this case in Section
\ref{subsec:diff-est-case3}. We do not integrate by parts in time.
We merely apply the linear Strichartz estimates $(L_{t}^{8}L_{x}^{4})$
to the highest four frequencies. The estimates allow for $s_{2}\geq\frac{3}{4}$.
This is why we set $s_{2}=\frac{3}{4}$ in Proposition \ref{prop:EnergyEstimatesDifferences}. 
\end{itemize}

\subsection{\label{subsec:diff-est-case1}Case $M_{1}\sim M_{2}\gtrsim M_{3}\gg M_{4}$}

Recall from \eqref{eq:BO-like-nonresonance} that 
\[
|\Omega_{k}|\sim M_{1}M_{3}.
\]
We handle these interactions via integration by parts in time (only
for $M_{3}\gtrsim1$). The following lemma systematically treats the
error terms arising from integration by parts in time.

\begin{lem}[Integration by parts in time]
\label{lem:int-by-pts-time-difference}Let $T\in(0,1]$; let $\overline{s}=(s_{1},s_{2})$
be such that $0<s_{1}<1/2<s_{2}$. Assume $M_{1}\sim M_{2}\gtrsim M_{3}\gg M_{4}$
and $M_{1}\gtrsim1$.
\begin{itemize}
\item For the $H_{\lambda}^{\overline{s}-1}$-estimate, we find 
\begin{align}
 & \quad\sup_{t\in[0,T]}\Big|\int_{0}^{t}\int_{\lambda\T}P_{N}vP_{K}vP_{K_{1}}u\dots P_{K_{k-1}}u\,dxds\Big|\nonumber \\
 & \lesssim(M_{1}^{-1}+\|u\|_{F_{\lambda}^{\overline{s}}}^{k-1})M_{3}^{0,-1}\Big(\prod_{i=3}^{k+1}M_{i}^{1/2}\Big)c_{N}^{(v,\overline{s},-1)}c_{K}^{(v,\overline{s},-1)}\prod_{i=1}^{k-1}c_{K_{i}}^{(u,\overline{s})}.\label{eq:s-1-int-by-pts-time}
\end{align}
\item For the $H_{\lambda}^{\overline{s}}$-estimate, we find 
\begin{align}
 & \quad\sup_{t\in[0,T]}\Big|\int_{0}^{t}\int_{\lambda\T}P_{N}v\partial_{x}P_{K}vP_{K_{1}}u\dots P_{K_{k-1}}u\,dxds\Big|\nonumber \\
 & \lesssim KM_{3}^{0,-1}\Big(\prod_{i=3}^{k+1}M_{i}^{1/2}\Big)\Big(\prod_{i=1}^{k-1}c_{K_{i}}^{(u,\overline{s})}\Big)\label{eq:s-int-by-pts-time-1}\\
 & \quad\times\bigg\{(M_{1}^{-1}+\|u\|_{F_{\lambda}^{\overline{s}}}^{k-1})c_{N}^{(v,\overline{s})}c_{K}^{(v,\overline{s})}+\|u\|_{F_{\lambda}^{\overline{s}}}^{k-2}\|v\|_{F_{\lambda}^{\overline{s}}}(c_{N}^{(v,\overline{s})}c_{K}^{(u,\overline{s})}+c_{N}^{(u,\overline{s})}c_{K}^{(v,\overline{s})})\bigg\}\nonumber 
\end{align}
and 
\begin{align}
 & \quad\sup_{t\in[0,T]}\Big|\int_{0}^{t}\int_{\T}P_{N}vP_{K}vP_{K_{1}}w\dots P_{K_{k-1}}w\,dxds\Big|\nonumber \\
 & \lesssim M_{3}^{0,-1}\Big(\prod_{i=3}^{k+1}M_{i}^{1/2}\Big)\Big(\prod_{i=1}^{k-1}c_{K_{i}}^{(w,\overline{s})}\Big)\label{eq:s-int-by-pts-time-2}\\
 & \quad\times\bigg\{(M_{1}^{-1}+\|u\|_{F_{\lambda}^{\overline{s}}}^{k-1})c_{N}^{(v,\overline{s})}+\|u\|_{F_{\lambda}^{\overline{s}}}^{k-2}\|v\|_{F_{\lambda}^{\overline{s}}}c_{N}^{(u,\overline{s})}\bigg\} c_{K}^{(v,\overline{s},-1)},\nonumber 
\end{align}
where $c_{K_{i}}^{(w,\overline{s})}=K_{i}c_{K_{i}}^{(u_{2},\overline{s},1)}$
when $w=\partial_{x}u_{2}$. 
\end{itemize}
Here, we recall that the quantities $c_{N}^{(v,\overline{s},s')}$
are defined in \eqref{eq:c_N-definition}.
\end{lem}

\begin{proof}
When $M_{3}\ll1$, the estimates follow from applying two short-time
bilinear Strichartz estimates to $M_{1},\dots,M_{4}$ after localizing
to intervals of length $M_{1}^{-1}$.

From now on, we assume $M_{3}\gtrsim1$. Here we focus on \eqref{eq:s-1-int-by-pts-time}.
The proofs of \eqref{eq:s-int-by-pts-time-1} and \eqref{eq:s-int-by-pts-time-2}
will be briefly sketched at the end of the proof. Fix $t\in[0,T]$.
The following estimates will be uniform for $t\in[0,T]$. Since $|\Omega_{k}|\sim M_{1}M_{3}$,
we integrate 
\[
\int_{0}^{t}\int_{\lambda\T}P_{N}vP_{K}vP_{K_{1}}u\dots P_{K_{k-1}}u\,dxds
\]
by parts in time to get the boundary term 
\[
M_{1}^{-1}M_{3}^{-1}\int_{\lambda\T}P_{N}vP_{K}vP_{K_{1}}u\dots P_{K_{k-1}}u\,dx\bigg|_{0}^{t}
\]
and the spacetime term 
\begin{align*}
M_{1}^{-1}M_{3}^{-1} & \bigg(\int_{0}^{t}\int_{\lambda\T}P_{N}\partial_{x}(vu^{k-1})P_{K}vP_{K_{1}}u\dots P_{K_{k-1}}u\,dxds\\
 & +\int_{0}^{t}\int_{\lambda\T}P_{N}vP_{K}\partial_{x}(vu^{k-1})P_{K_{1}}u\dots P_{K_{k-1}}u\,dxds\\
 & +\sum_{j=1}^{k-1}\int_{0}^{t}\int_{\lambda\T}P_{N}vP_{K}vP_{K_{j}}\partial_{x}(u^{k-1})\prod_{i\neq j}P_{K_{i}}u\,dxds\bigg).
\end{align*}
The boundary term is estimated by applying $L_{x}^{2}$ to the two
highest frequencies, applying pointwise bounds for the remaining factors:
\[
\Big(M_{1}^{-1}M_{3}^{-1}\prod_{i=3}^{k+1}M_{i}^{1/2}\Big)\|P_{N}v\|_{L_{t}^{\infty}([0,T],L_{\lambda}^{2})}\|P_{K}v\|_{L_{t}^{\infty}([0,T],L_{\lambda}^{2})}\prod_{i=1}^{k-1}\|P_{K_{i}}u\|_{L_{t}^{\infty}([0,T],L_{\lambda}^{2})}.
\]
The spacetime term is estimated by applying $L_{x}^{2}$ to the two
highest frequencies, pointwise bounds to the remaining factors, and
$L_{t}^{1}$ to $\partial_{x}(vu^{k-1})$ or $\partial_{x}(u^{k})$.
We find by Lemma \ref{lem:decomp-prod} ($s'=-1$ for $\partial_{x}(vu^{k-1})$
and $s'=0$ for $\partial_{x}(u^{k})$): 
\[
\|u\|_{F_{\lambda}^{\overline{s}}}^{k-1}\Big(M_{3}^{-1}\prod_{i=3}^{k+1}M_{i}^{1/2}\Big)c_{N}^{(v,\overline{s},-1)}c_{K}^{(v,\overline{s},-1)}\prod_{i=1}^{k-1}c_{K_{i}}^{(u,\overline{s})}.
\]
This completes the proof of \eqref{eq:s-1-int-by-pts-time}.

For \eqref{eq:s-int-by-pts-time-1}, we replace $\partial_{x}$ by
$K$ and integrate by parts in time. We apply Lemma \ref{lem:decomp-prod}
with $s'=0$ for both $\partial_{x}(vu^{k-1})$ and $\partial_{x}(u^{k})$.
For \eqref{eq:s-int-by-pts-time-2}, we apply Lemma \ref{lem:decomp-prod}
with $s'=-1$ for $P_{K}\partial_{x}(vu^{k-1})$, $s'=0$ for $P_{N}\partial_{x}(vu^{k-1})$
and $P_{K_{i}}\partial_{x}(u^{k})$, and $s'=1$ for $P_{K_{i}}\partial_{xx}(u_{2}^{k})$.
For the last one $P_{K_{i}}\partial_{xx}(u_{2}^{k})$, we use 
\[
\|P_{K_{i}}\partial_{xx}(u_{2}^{k})\|_{L_{t}^{1}([0,T],L_{\lambda}^{2})}\lesssim K_{i}\|P_{K_{i}}\partial_{x}(u_{2}^{k})\|_{L_{t}^{1}([0,T],L_{\lambda}^{2})}\lesssim K_{i}^{1/2,1}\|u_{2}\|_{F_{\lambda}^{\overline{s}}}^{k-1}K_{i}c_{K_{i}}^{(u_{2},\overline{s},1)}.
\]
This finishes the proof.
\end{proof}
From now on, we estimate \eqref{eq:s-1-expansion} and \eqref{eq:s-expansion}
using the above integration by parts in time lemma. Recall again that
$M_{1}\sim M_{2}\gtrsim N\gtrsim1$.

\textbf{Case A:} $N\in\{M_{1},M_{2}\}$.

\underline{Subcase I:} $K\in\{M_{1},M_{2}\}$. We start with the
$H_{\lambda}^{\overline{s}-1}$-estimate. In this case, we can integrate
by parts in \eqref{eq:s-1-expansion} to move the derivative $\partial_{x}$
to $K_{1},\dots,K_{k-1}$. We may replace $\partial_{x}$ by $K_{1}$
and let $K=N$ (due to $M_{1}\sim M_{2}$). Then \eqref{eq:s-1-expansion}
is of the form 
\[
\sum_{N\gtrsim1}\sum_{K_{1},\dots,K_{k-1}\lesssim N}N^{2s_{2}-2}K_{1}\sup_{t\in[0,T]}\bigg|\int_{0}^{t}\int_{\lambda\T}(P_{N}v)^{2}P_{K_{1}}u\dots P_{K_{k-1}}u\,dxds\bigg|.
\]
By \eqref{eq:s-1-int-by-pts-time}, each summand is estimated by 
\[
(N^{-1}+\|u\|_{F_{\lambda}^{\overline{s}}}^{k-1})(d_{N}^{(v,\overline{s},-1)})^{2}\prod_{i=1}^{k-1}K_{i}^{1/2}c_{K_{i}}^{(u,\overline{s})},
\]
where we recall the definition $d_{N}^{(v,\overline{s},-1)}$ of \eqref{eq:d_N-definition}.
Using \eqref{eq:l2-summation-c_N}, this sums up to \eqref{eq:LipschitzContinuityNegativeSobolevSpaces}.

For the $H_{\lambda}^{\overline{s}}$-estimate, we consider \eqref{eq:s-expansion}.
As before, we can move the derivative $\partial_{x}$ on $v$ to $K_{1},\dots,K_{k-1}$.
We may replace $\partial_{x}$ by $K_{1}$ and let $K=N$. Then \eqref{eq:s-expansion}
is of the form 
\begin{align*}
\sum_{N\gtrsim1}\sum_{K_{1},\dots,K_{k-1}\lesssim N}N^{2s_{2}} & K_{1}\sup_{t\in[0,T]}\bigg|\int_{0}^{t}\int_{\lambda\T}(P_{N}v)^{2}P_{K_{1}}u\dots P_{K_{k-1}}udxds\bigg|.
\end{align*}
By a variant of \eqref{eq:s-int-by-pts-time-1} (deleting $K$), each
summand is estimated by 
\[
\Big(\prod_{i=1}^{k-1}K_{i}^{1/2}c_{K_{i}}^{(u,\overline{s})}\Big)\bigg\{(N^{-1}+\|u\|_{F_{\lambda}^{\overline{s}}}^{k-1})(d_{N}^{(v,\overline{s})})^{2}+\|u\|_{F_{\lambda}^{\overline{s}}}^{k-2}\|v\|_{F_{\lambda}^{\overline{s}}}d_{N}^{(v,\overline{s})}d_{N}^{(u,\overline{s})}\bigg\}.
\]
Using \eqref{eq:l2-summation-c_N}, this sums up to \eqref{eq:HsContinuityBonaSmithRegularities}.

\underline{Subcase II:} $K\in\{M_{3},\dots,M_{k+1}\}$. We may assume
$K_{1}=N$. We start with the $H_{\lambda}^{\overline{s}-1}$-estimate.
Replacing $\partial_{x}P_{N}$ by $N$, we read \eqref{eq:s-1-expansion}
as 
\[
\sum_{N\gtrsim1}\sum_{K,K_{2},\dots,K_{k-1}\lesssim N}N^{2s_{2}-1}\sup_{t\in[0,T]}\bigg|\int_{0}^{t}\int_{\lambda\T}P_{N}vP_{K}vP_{N}u\dots P_{K_{k-1}}udxds\bigg|.
\]
By \eqref{eq:s-1-int-by-pts-time}, each summand is estimated by (use
$M_{3}^{0,-1}\lesssim K^{0,-1}$ and $K^{\frac{1}{2}-s_{1},-s_{2}+\frac{1}{2}}\lesssim K^{0+,0-}$)
\[
\Big(\prod_{i=2}^{k-1}K_{i}^{1/2}c_{K_{i}}^{(u,\overline{s})}\Big)(N^{-1}+\|u\|_{F_{\lambda}^{\overline{s}}}^{k-1})d_{N}^{(v,\overline{s},-1)}K^{0+,0-}d_{K}^{(v,\overline{s},-1)}d_{N}^{(u,\overline{s})}.
\]
This sums up to \eqref{eq:LipschitzContinuityNegativeSobolevSpaces}.

For the $H_{\lambda}^{s}$-estimate, we apply \eqref{eq:s-int-by-pts-time-1}
and \eqref{eq:s-int-by-pts-time-2} to each summand of the expression
\eqref{eq:s-expansion} by the sum of (use again $M_{3}^{0,-1}\lesssim K^{0,-1}$
and $K^{\frac{1}{2}-s_{1},-s_{2}+\frac{1}{2}}\lesssim K^{0+,0-}$)
\begin{alignat*}{1}
d_{N}^{(u,\overline{s})}\Big(\prod_{i=2}^{k-1}K_{i}^{1/2}c_{K_{i}}^{(u,\overline{s})}\Big)K^{0+,0-} & \bigg\{(N^{-1}+\|u\|_{F_{\lambda}^{\overline{s}}}^{k-1})d_{N}^{(v,\overline{s})}d_{K}^{(v,\overline{s})}\\
 & +\|u\|_{F_{\lambda}^{\overline{s}}}^{k-2}\|v\|_{F_{\lambda}^{\overline{s}}}(d_{N}^{(v,\overline{s})}d_{K}^{(u,\overline{s})}+d_{N}^{(u,\overline{s})}d_{K}^{(v,\overline{s})})\bigg\}
\end{alignat*}
and 
\[
d_{N}^{(w,\overline{s})}\Big(\prod_{i=2}^{k-1}K_{i}^{1/2}c_{K_{i}}^{(w,\overline{s})}\Big)K^{0+,0-}\bigg\{(N^{-1}+\|u\|_{F_{\lambda}^{\overline{s}}}^{k-1})d_{N}^{(v,\overline{s})}+\|u\|_{F_{\lambda}^{\overline{s}}}^{k-2}\|v\|_{F_{\lambda}^{\overline{s}}}d_{N}^{(u,\overline{s})}\bigg\} d_{K}^{(v,\overline{s},-1)}.
\]
Using $\|w\|_{F_{\lambda}^{\overline{s}}}^{k-1}\lesssim\|u_{2}\|_{F_{\lambda}^{\overline{s}+1}}\|u\|_{F_{\lambda}^{\overline{s}}}^{k-2}$,
these sum up to \eqref{eq:HsContinuityBonaSmithRegularities}.

\textbf{Case B:} $N\in\{M_{3},\dots,M_{k+1}\}$.

\underline{Subcase I:} $K\in\{M_{1},M_{2}\}$. Note that $K_{1}\sim K$.
We start with the $H_{\lambda}^{\overline{s}-1}$-estimate. Replacing
$\partial_{x}P_{N}$ by $N$, we read \eqref{eq:s-1-expansion} as
\[
\sum_{N\gtrsim1}\sum_{\substack{K\sim K_{1}\gtrsim N\\
K_{2},\dots,K_{k-1}\lesssim K_{1}
}
}N^{2s-1}\sup_{t\in[0,T]}\bigg|\int_{0}^{t}\int_{\lambda\T}P_{N}vP_{K}vP_{K_{1}}u\dots P_{K_{k-1}}udxds\bigg|.
\]
By \eqref{eq:s-1-int-by-pts-time}, each summand is estimated by (use
$M_{3}^{0,-1}\leq N^{-1}$)
\[
d_{K_{1}}^{(u,\overline{s})}\Big(\prod_{i=2}^{k-1}K_{i}^{1/2}c_{K_{i}}^{(u,\overline{s})}\Big)(K^{-1}+\|u\|_{F_{\lambda}^{\overline{s}}}^{k-1})N^{s_{2}-\frac{1}{2}}K^{-2s_{2}+1}d_{N}^{(v,\overline{s},-1)}d_{K}^{(v,\overline{s},-1)}.
\]
Using $s_{2}>\frac{1}{2}$, this easily sums up to \eqref{eq:LipschitzContinuityNegativeSobolevSpaces}.

For the $H_{\lambda}^{s}$-estimate, we apply \eqref{eq:s-int-by-pts-time-1}
and \eqref{eq:s-int-by-pts-time-2} to each summand of the expression
\eqref{eq:s-expansion} by the sum of (use again $M_{3}^{0,-1}\leq N^{-1}$)
\begin{align*}
d_{K_{1}}^{(u,\overline{s})}\Big(\prod_{i=2}^{k-1}K_{i}^{1/2}c_{K_{i}}^{(u,\overline{s})}\Big)N^{s_{2}-\frac{1}{2}}K^{-2s_{2}+1} & \bigg\{(K^{-1}+\|u\|_{F_{\lambda}^{\overline{s}}}^{k-1})d_{N}^{(v,\overline{s})}d_{K}^{(v,\overline{s})}\\
 & +\|u\|_{F_{\lambda}^{\overline{s}}}^{k-2}\|v\|_{F_{\lambda}^{\overline{s}}}(d_{N}^{(v,\overline{s})}d_{K}^{(u,\overline{s})}+d_{N}^{(u,\overline{s})}d_{K}^{(v,\overline{s})})\bigg\}
\end{align*}
and 
\begin{align*}
 & d_{K_{1}}^{(w,\overline{s})}\Big(\prod_{i=2}^{k-1}K_{i}^{1/2}c_{K_{i}}^{(w,\overline{s})}\Big)N^{s_{2}-\frac{1}{2}}K^{-2s_{2}+1}\\
 & \times\bigg\{(N^{-1}+\|u\|_{F_{\lambda}^{\overline{s}}}^{k-1})d_{N}^{(v,\overline{s})}+\|u\|_{F_{\lambda}^{\overline{s}}}^{k-2}\|v\|_{F_{\lambda}^{\overline{s}}}d_{N}^{(u,\overline{s})}\bigg\} d_{K}^{(v,\overline{s},-1)}.
\end{align*}
Using $s_{2}>\frac{1}{2}$ and $\|w\|_{F_{\lambda}^{\overline{s}}}^{k-1}\lesssim\|u_{2}\|_{F_{\lambda}^{\overline{s}+1}}\|u\|_{F_{\lambda}^{\overline{s}}}^{k-2}$,
these sum up to \eqref{eq:HsContinuityBonaSmithRegularities}.

\underline{Subcase II:} $K\in\{M_{3},\dots,M_{k+1}\}$. Note that
$K_{1}=M_{1}$ and $K_{2}=M_{2}$. We start with the $H_{\lambda}^{\overline{s}-1}$-estimate.
We replace $\partial_{x}P_{N}$ by $N$ to rewrite \eqref{eq:s-1-expansion}
as 
\[
\sum_{N\geq1}\sum_{\substack{K_{1}\sim K_{2}\gtrsim N\\
K,K_{3},\dots,K_{k-1}\lesssim K_{2}
}
}N^{2s_{2}-1}\sup_{t\in[0,T]}\bigg|\int_{0}^{t}\int_{\lambda\T}P_{N}vP_{K}vP_{K_{1}}u\dots P_{K_{k-1}}udxds\bigg|.
\]
By \eqref{eq:s-1-int-by-pts-time}, each summand is estimated by (use
$M_{3}^{0,-1}\lesssim K^{0,-1}$ and $K^{\frac{1}{2}-s_{1},-s_{2}+\frac{1}{2}}\lesssim K^{0+,0-}$)
\[
(K_{1}^{-1}+\|u\|_{F_{\lambda}^{\overline{s}}}^{k-1})N^{s_{2}+\frac{1}{2}}K^{0+,0-}K_{1}^{-2s_{2}}d_{K_{1}}^{(u,\overline{s})}d_{K_{2}}^{(u,\overline{s})}\Big(\prod_{i=3}^{k-1}K_{i}^{1/2}c_{K_{i}}^{(u,\overline{s})}\Big)d_{N}^{(v,\overline{s},-1)}d_{K}^{(v,\overline{s},-1)}.
\]
Using $s_{2}>\frac{1}{2}$, this sums up to \eqref{eq:LipschitzContinuityNegativeSobolevSpaces}.

For the $H_{\lambda}^{s}$-estimate, we apply \eqref{eq:s-int-by-pts-time-1}
and \eqref{eq:s-int-by-pts-time-2} to each summand of the expression
\eqref{eq:s-expansion} by the sum of (use again $M_{3}^{0,-1}\lesssim K^{0,-1}$
and $K^{\frac{1}{2}-s_{1},-s_{2}+\frac{1}{2}}\lesssim K^{0+,0-}$)
\begin{align*}
 & N^{s_{2}+\frac{1}{2}}K^{0+,0-}K_{1}^{-2s_{2}}d_{K_{1}}^{(u,\overline{s})}d_{K_{2}}^{(u,\overline{s})}\Big(\prod_{i=3}^{k-1}K_{i}^{1/2}c_{K_{i}}^{(u,\overline{s})}\Big)\\
 & \times\bigg\{(K_{1}^{-1}+\|u\|_{F_{\lambda}^{\overline{s}}}^{k-1})d_{N}^{(v,\overline{s})}d_{K}^{(v,\overline{s})}+\|u\|_{F_{\lambda}^{\overline{s}}}^{k-1}\|v\|_{F_{\lambda}^{\overline{s}}}(d_{N}^{(v,\overline{s})}d_{K}^{(u,\overline{s})}+d_{N}^{(u,\overline{s})}d_{K}^{(v,\overline{s})})\bigg\}
\end{align*}
and 
\begin{align*}
 & N^{s_{2}+\frac{1}{2}}K^{0+,0-}K_{1}^{-2s_{2}}d_{K_{1}}^{(w,\overline{s})}d_{K_{2}}^{(w,\overline{s})}\Big(\prod_{i=3}^{k-1}K_{i}^{1/2}c_{K_{i}}^{(w,\overline{s})}\Big)\\
 & \times\bigg\{(K_{1}^{-1}+\|u\|_{F_{\lambda}^{\overline{s}}}^{k-1})d_{N}^{(v,\overline{s})}+\|u\|_{F_{\lambda}^{\overline{s}}}^{k-2}\|v\|_{F_{\lambda}^{\overline{s}}}d_{N}^{(u,\overline{s})}\bigg\} d_{K}^{(v,\overline{s},-1)}.
\end{align*}
Using $s_{2}>\frac{1}{2}$ and $\|w\|_{F_{\lambda}^{\overline{s}}}^{k-1}\lesssim\|u_{2}\|_{F_{\lambda}^{\overline{s}+1}}\|u\|_{F_{\lambda}^{\overline{s}}}^{k-2}$,
these easily sum up to \eqref{eq:HsContinuityBonaSmithRegularities}.

Therefore, the proofs of \eqref{eq:LipschitzContinuityNegativeSobolevSpaces}
and \eqref{eq:HsContinuityBonaSmithRegularities} are completed in
case of $M_{1}\sim M_{2}\gtrsim M_{3}\gg M_{4}$.

\subsection{\label{subsec:diff-est-case2}Case $M_{1}\sim M_{2}\gg M_{3}\sim M_{4}$}

We directly estimate \eqref{eq:s-1-expansion} and \eqref{eq:s-expansion}
in this case. We do not integrate by parts in time. We can use two
bilinear Strichartz estimates in the form 
\begin{equation}
\begin{split} & \quad\sup_{t\in[0,T]}\bigg|\int_{0}^{t}\int_{\lambda\T}P_{M_{1}}u_{1}\dots P_{M_{m}}u_{m}dxds\bigg|\\
 & \lesssim M_{3}^{1/2,0}M_{4}^{1/2,0}\Big(\prod_{i=5}^{m}M_{i}^{1/2}\Big)\prod_{i=1}^{m}\|P_{M_{i}}u_{i}\|_{F_{\lambda}^{0}}.
\end{split}
\label{eq:diff-est-two-bilinear}
\end{equation}

\textbf{Case A:} $N\in\{M_{1},M_{2}\}$.

\underline{Subcase I:} $K\in\{M_{1},M_{2}\}$. We start with the
$H_{\lambda}^{\overline{s}-1}$-estimate. In this case, we can perform
integration by parts in space to the expression \eqref{eq:s-1-expansion}
to move the derivative $\partial_{x}$ to $K_{1},\dots,K_{k-1}$.
Replacing $\partial_{x}$ by $K_{1}$, we need to estimate 
\[
\sum_{N\geq1}\sum_{\substack{K_{1}\sim K_{2}\ll N\\
K_{3},\dots,K_{k-1}\lesssim K_{2}
}
}N^{2s_{2}-2}K_{1}\sup_{t\in[0,T]}\bigg|\int_{0}^{t}\int_{\lambda\T}(P_{N}v)^{2}P_{K_{1}}u\dots P_{K_{k-1}}u\,dxds\bigg|.
\]
By \eqref{eq:diff-est-two-bilinear} and $K_{1}\sim K_{2}$, each
summand is estimated by 
\[
\|P_{N}v\|_{F_{\lambda}^{\overline{s}-1}}^{2}\prod_{i=1}^{k-1}\|P_{K_{i}}u\|_{F_{\lambda}^{1/2}}.
\]
This sums up to \eqref{eq:LipschitzContinuityNegativeSobolevSpaces}.

For the $H_{\lambda}^{\overline{s}}$-estimate, we also integrate
by parts in space to the expression \eqref{eq:s-expansion} to assume
that $\partial_{x}$ is applied to $K_{1},\dots,K_{k-1}$. Replacing
$\partial_{x}$ by $K_{1}$, it suffices to estimate 
\[
\sum_{N\geq1}\sum_{\substack{K_{1}\sim K_{2}\ll N\\
K_{3},\dots,K_{k-1}\lesssim K_{2}
}
}N^{2s_{2}}K_{1}\sup_{t\in[0,T]}\bigg|\int_{0}^{t}\int_{\lambda\T}(P_{N}v)^{2}P_{K_{1}}u\dots P_{K_{k-1}}u\,dxds\bigg|.
\]
By \eqref{eq:diff-est-two-bilinear} and $K_{1}\sim K_{2}$, each
summand is estimated by 
\[
\|P_{N}v\|_{F_{\lambda}^{\overline{s}}}^{2}\prod_{i=1}^{k-1}\|P_{K_{i}}u\|_{F_{\lambda}^{1/2}}.
\]
This sums up to \eqref{eq:HsContinuityBonaSmithRegularities}.

\underline{Subcase II:} $K\in\{M_{3},\dots,M_{k+1}\}$. We may assume
$K_{1}=N$. For the $H_{\lambda}^{\overline{s}-1}$-estimate, it suffices
to estimate 
\[
\sum_{N\geq1}\sum_{\substack{M_{3}\sim M_{4}\ll N\\
M_{5},\dots,M_{k+1}\lesssim M_{3}
}
}N^{2s_{2}-1}\sup_{t\in[0,T]}\bigg|\int_{0}^{t}\int_{\lambda\T}P_{N}vP_{K}vP_{N}u\dots P_{K_{k-1}}u\,dxds\bigg|.
\]
Applying \eqref{eq:diff-est-two-bilinear} and $M_{3}^{1/2,0}M_{4}^{1/2,0}\lesssim K^{1/2,-1/2}K_{2}^{1/2}$
(due to $M_{3}\sim M_{4}\sim K_{2}\gtrsim K$), each summand is estimated
by 
\begin{align*}
 & \quad N^{2s_{2}-1}K^{1/2,-1/2}\|P_{N}v\|_{F_{\lambda}^{0}}\|P_{K}v\|_{F_{\lambda}^{0}}\|P_{N}u\|_{F_{\lambda}^{0}}\prod_{i=2}^{k-1}\|P_{K_{i}}u\|_{F_{\lambda}^{1/2}}\\
 & \lesssim\|P_{N}v\|_{F_{\lambda}^{\overline{s}-1}}\|P_{N}u\|_{F_{\lambda}^{\overline{s}}}K^{0+,0-}\|P_{K}v\|_{F_{\lambda}^{\overline{s}-1}}\prod_{i=2}^{k-1}\|P_{K_{i}}u\|_{F_{\lambda}^{1/2}}.
\end{align*}
This sums up to \eqref{eq:LipschitzContinuityNegativeSobolevSpaces}.

For the $H_{\lambda}^{\overline{s}}$-estimate, we apply \eqref{eq:diff-est-two-bilinear}
to \eqref{eq:s-expansion} to estimate each summand by 
\begin{align*}
\|P_{N}v\|_{F_{\lambda}^{\overline{s}}} & \Big(K^{0+,0-}\|P_{K}v\|_{F_{\lambda}^{\overline{s}}}\|P_{N}u\|_{F_{\lambda}^{\overline{s}}}\prod_{i=2}^{k-1}\|P_{K_{i}}u\|_{F_{\lambda}^{1/2}}\\
 & +K^{0+,0-}\|P_{K}v\|_{F_{\lambda}^{\overline{s}-1}}\|P_{N}w\|_{F_{\lambda}^{\overline{s}}}\prod_{i=2}^{k-1}\|P_{K_{i}}w\|_{F_{\lambda}^{1/2}}\Big).
\end{align*}
This sums up to \eqref{eq:HsContinuityBonaSmithRegularities}.

\textbf{Case B:} $N\in\{M_{3},\dots,M_{k+1}\}$.

\underline{Subcase I:} $K\in\{M_{1},M_{2}\}$. Note that $K_{1}\sim K$
and $M_{3}\sim M_{4}\gtrsim N\gtrsim1$. For the $H_{\lambda}^{\overline{s}-1}$-estimate,
it suffices to estimate 
\[
\sum_{\substack{M_{1}\sim M_{2}\gg M_{3}\sim M_{4}\gtrsim1\\
M_{5},\dots,M_{k+1}\lesssim M_{4}
}
}N^{2s_{2}-1}\sup_{t\in[0,T]}\bigg|\int_{0}^{t}\int_{\lambda\T}P_{N}vP_{K}vP_{K_{1}}u\dots P_{K_{k-1}}u\,dxds\bigg|.
\]
Using \eqref{eq:diff-est-two-bilinear}, $K_{1}\sim K$, and $N^{1/2}\lesssim K_{2}^{1/2}$
(due to $K_{2}\in\{M_{3},M_{4}\}$), each summand is estimated by
\begin{align*}
 & \quad N^{2s_{2}-1}\Big(\prod_{i=5}^{k+1}M_{i}^{1/2}\Big)\|P_{N}v\|_{F_{\lambda}^{0}}\|P_{K}v\|_{F_{\lambda}^{0}}\prod_{i=1}^{k-1}\|P_{K_{i}}u\|_{F_{\lambda}^{0}}\\
 & \lesssim N^{s_{2}-\frac{1}{2}}K^{-2s_{2}+1}\|P_{N}v\|_{F_{\lambda}^{\overline{s}-1}}\|P_{K}v\|_{F_{\lambda}^{\overline{s}-1}}\|P_{K_{1}}u\|_{F_{\lambda}^{\overline{s}}}\prod_{i=2}^{k-1}\|P_{K_{i}}u\|_{F_{\lambda}^{1/2}}.
\end{align*}
Using $s_{2}>\frac{1}{2}$ to guarantee $-s_{2}+\frac{1}{2}<0$, this
sums up to \eqref{eq:LipschitzContinuityNegativeSobolevSpaces}.

For the $H_{\lambda}^{\overline{s}}$-estimate, we apply \eqref{eq:diff-est-two-bilinear}
to \eqref{eq:s-expansion} to estimate each summand by 
\begin{align*}
N^{s_{2}-\frac{1}{2}}K^{-2s_{2}+1}\|P_{N}v\|_{F_{\lambda}^{\overline{s}}} & \Big(\|P_{K}v\|_{F_{\lambda}^{\overline{s}}}\|P_{K_{1}}u\|_{F_{\lambda}^{\overline{s}}}\prod_{i=2}^{k-1}\|P_{K_{i}}u\|_{F_{\lambda}^{1/2}}\\
 & +\|P_{K}v\|_{F_{\lambda}^{\overline{s}-1}}\|P_{K_{1}}w\|_{F_{\lambda}^{\overline{s}}}\prod_{i=2}^{k-1}\|P_{K_{i}}w\|_{F_{\lambda}^{1/2}}\Big).
\end{align*}
Using $s_{2}>\frac{1}{2}$ to guarantee $-s_{2}+\frac{1}{2}<0$, this
sums up to \eqref{eq:HsContinuityBonaSmithRegularities}.

\underline{Subcase II:} $K\in\{M_{3},\dots,M_{k+1}\}$. We still
have $M_{3}\sim M_{4}\gtrsim N\gtrsim1$. For the $H_{\lambda}^{\overline{s}-1}$-estimate,
it suffices to estimate 
\[
\sum_{\substack{M_{1}\sim M_{2}\gg M_{3}\sim M_{4}\gtrsim1\\
M_{5},\dots,M_{k+1}\lesssim M_{4}
}
}N^{2s_{2}-1}\sup_{t\in[0,T]}\bigg|\int_{0}^{t}\int_{\lambda\T}P_{N}vP_{K}vP_{K_{1}}u\dots P_{K_{k-1}}u\,dxds\bigg|.
\]
Using \eqref{eq:diff-est-two-bilinear} and $\{K_{1},K_{2}\}=\{M_{1},M_{2}\}$,
each summand is estimated by 
\begin{align*}
 & \quad N^{2s_{2}-1}M_{3}^{1/2,0}M_{4}^{1/2,0}\Big(\prod_{i=5}^{k+1}M_{i}^{1/2}\Big)\|P_{N}v\|_{F_{\lambda}^{0}}\|P_{K}v\|_{F_{\lambda}^{0}}\prod_{i=1}^{k-1}\|P_{K_{i}}u\|_{F_{\lambda}^{0}}\\
 & \lesssim N^{s_{2}}K^{\frac{1}{2}-s_{1},-s_{2}+1}K_{1}^{-2s_{2}}\|P_{N}v\|_{F_{\lambda}^{\overline{s}-1}}\|P_{K}v\|_{F_{\lambda}^{\overline{s}-1}}\|P_{K_{1}}u\|_{F_{\lambda}^{\overline{s}}}\|P_{K_{2}}u\|_{F_{\lambda}^{\overline{s}}}\prod_{i=3}^{k-1}\|P_{K_{i}}u\|_{F_{\lambda}^{1/2}}.
\end{align*}
Using $s_{2}>\frac{1}{2}$ to guarantee $-2s_{2}+1<0$, this sums
up to \eqref{eq:LipschitzContinuityNegativeSobolevSpaces}.

For the $H_{\lambda}^{\overline{s}}$-estimate, we apply \eqref{eq:diff-est-two-bilinear}
to \eqref{eq:s-expansion} to estimate each summand by 
\begin{align*}
N^{s_{2}}K^{\frac{1}{2}-s_{1},-s_{2}+1}K_{1}^{-2s_{2}}\|P_{N}v\|_{F_{\lambda}^{\overline{s}}} & \Big(\|P_{K}v\|_{F_{\lambda}^{\overline{s}}}\|P_{K_{1}}u\|_{F_{\lambda}^{\overline{s}}}\|P_{K_{2}}u\|_{F_{\lambda}^{\overline{s}}}\prod_{i=3}^{k-1}\|P_{K_{i}}u\|_{F_{\lambda}^{1/2}}\\
 & +\|P_{K}v\|_{F_{\lambda}^{\overline{s}-1}}\|P_{K_{1}}w\|_{F_{\lambda}^{\overline{s}}}\|P_{K_{2}}w\|_{F_{\lambda}^{\overline{s}}}\prod_{i=3}^{k-1}\|P_{K_{i}}w\|_{F_{\lambda}^{1/2}}\Big).
\end{align*}
Using $s_{2}>\frac{1}{2}$ to guarantee $-2s_{2}+1<0$, this sums
up to \eqref{eq:HsContinuityBonaSmithRegularities}.

Therefore, the proofs of \eqref{eq:LipschitzContinuityNegativeSobolevSpaces}
and \eqref{eq:HsContinuityBonaSmithRegularities} are completed in
case of $M_{1}\sim M_{2}\gg M_{3}\sim M_{4}$.

\subsection{\label{subsec:diff-est-case3}Case $M_{1}\sim M_{2}\sim M_{3}\sim M_{4}$}

We estimate \eqref{eq:s-1-expansion} and \eqref{eq:s-expansion}
via linear Strichartz estimates: 
\begin{equation}
\sup_{t\in[0,T]}\Big|\int_{0}^{t}\int_{\lambda\T}P_{M_{1}}u_{1}\dots P_{M_{m}}u_{m}dxds\Big|\lesssim M_{1}^{1/2}\Big(\prod_{i=5}^{m}M_{i}^{1/2}\Big)\prod_{i=1}^{m}\|P_{M_{i}}u_{i}\|_{F_{\lambda}^{0}}.\label{eq:diff-est-lin-strichartz}
\end{equation}
Here, our estimates only allow for $s_{2}\geq3/4$ due to the loss
of $M_{1}^{1/2}$ compared to the previous case.

\textbf{Case A:} $N\in\{M_{1},\dots,M_{4}\}$.

\underline{Subcase I:} $K\in\{M_{1},\dots,M_{4}\}$. Note that $N\sim K\sim K_{1}\sim K_{2}$
so we may assume $K=K_{1}=K_{2}=N$. For the $H_{\lambda}^{\overline{s}-1}$-estimate,
we use \eqref{eq:diff-est-lin-strichartz} to estimate each summand
of \eqref{eq:s-1-expansion} by 
\[
N^{2s_{2}-\frac{1}{2}}\|P_{N}v\|_{F_{\lambda}^{0}}^{2}\|P_{N}u\|_{F_{\lambda}^{0}}^{2}\prod_{i=3}^{k-1}\|P_{K_{i}}u\|_{F_{\lambda}^{1/2}}.
\]
We use $s_{2}\geq\frac{3}{4}$ to estimate the above by 
\[
\|P_{N}v\|_{F_{\lambda}^{\overline{s}-1}}^{2}\|P_{N}u\|_{F_{\lambda}^{\overline{s}}}^{2}\prod_{i=3}^{k-1}\|P_{K_{i}}u\|_{F_{\lambda}^{1/2}}.
\]
This sums up to \eqref{eq:LipschitzContinuityNegativeSobolevSpaces}.

For the $H_{\lambda}^{\overline{s}}$-estimate, we replace $\partial_{x}$
of \eqref{eq:s-expansion} (even for $w$ when $w=\partial_{x}u_{2}$)
by $N$ and use \eqref{eq:diff-est-lin-strichartz} to estimate each
summand by 
\[
N^{2s_{2}+\frac{3}{2}}\|P_{N}v\|_{F_{\lambda}^{0}}^{2}\|P_{N}u\|_{F_{\lambda}^{0}}^{2}\prod_{i=3}^{k-1}\|P_{K_{i}}u\|_{F_{\lambda}^{1/2}}.
\]
We use $s_{2}\geq\frac{3}{4}$ to estimate the above by 
\[
\|P_{N}v\|_{F_{\lambda}^{\overline{s}}}^{2}\|P_{N}u\|_{F_{\lambda}^{\overline{s}}}^{2}\prod_{i=3}^{k-1}\|P_{K_{i}}u\|_{F_{\lambda}^{1/2}}.
\]
This sums up to \eqref{eq:HsContinuityBonaSmithRegularities}.

\underline{Subcase II:} $K\in\{M_{5},\dots,M_{k+1}\}$. Note that
$N\sim K_{1}\sim K_{2}\sim K_{3}$, so we may assume $K_{1}=K_{2}=K_{3}=N$.
For the $H_{\lambda}^{\overline{s}-1}$-estimate, we use \eqref{eq:diff-est-lin-strichartz}
to estimate each summand of \eqref{eq:s-1-expansion} by 
\[
N^{-2s_{2}+\frac{1}{2}}K^{\frac{1}{2}-s_{1},-s_{2}+\frac{3}{2}}\|P_{N}v\|_{F_{\lambda}^{\overline{s}-1}}\|P_{K}v\|_{F_{\lambda}^{\overline{s}-1}}\|P_{N}u\|_{F_{\lambda}^{\overline{s}}}^{3}\prod_{i=4}^{k-1}\|P_{K_{i}}u\|_{F_{\lambda}^{1/2}}.
\]
This sums up to \eqref{eq:LipschitzContinuityNegativeSobolevSpaces}
provided that $s_{2}\geq\frac{2}{3}$.

For the $H_{\lambda}^{\overline{s}}$-estimate, we find similarly
\begin{align*}
N^{-2s_{2}+\frac{1}{2}}K^{\frac{1}{2}-s_{1},-s_{2}+\frac{3}{2}}\|P_{N}v\|_{F_{\lambda}^{\overline{s}}} & \Big(\|P_{K}v\|_{F_{\lambda}^{\overline{s}}}\|P_{N}u\|_{F_{\lambda}^{\overline{s}}}^{3}\prod_{i=4}^{k-1}\|P_{K_{i}}u\|_{F_{\lambda}^{1/2}}\\
 & +\|P_{K}v\|_{F_{\lambda}^{\overline{s}-1}}\|P_{N}w\|_{F_{\lambda}^{\overline{s}}}^{3}\prod_{i=4}^{k-1}\|P_{K_{i}}w\|_{F_{\lambda}^{1/2}}\Big).
\end{align*}
This sums up to \eqref{eq:HsContinuityBonaSmithRegularities} provided
that $s_{2}\geq\frac{2}{3}$.

\textbf{Case B:} $N\in\{M_{5},\dots,M_{k+1}\}$.

\underline{Subcase I:} $K\in\{M_{1},\dots,M_{4}\}$. Note that $K\sim K_{1}\sim K_{2}\sim K_{3}$,
so we may assume $K=K_{1}=K_{2}=K_{3}$. For the $H_{\lambda}^{\overline{s}-1}$-estimate,
we use \eqref{eq:diff-est-lin-strichartz} to estimate each summand
of \eqref{eq:s-1-expansion} by 
\begin{align*}
 & \quad N^{2s_{2}-\frac{1}{2}}K^{1/2}\|P_{N}v\|_{F_{\lambda}^{0}}\|P_{K}v\|_{F_{\lambda}^{0}}\|P_{K}u\|_{F_{\lambda}^{0}}^{3}\prod_{i=4}^{k-1}\|P_{K_{i}}u\|_{F_{\lambda}^{1/2}}\\
 & \lesssim N^{s_{2}+\frac{1}{2}}K^{-4s_{2}+\frac{3}{2}}\|P_{N}v\|_{F_{\lambda}^{\overline{s}-1}}\|P_{K}v\|_{F_{\lambda}^{\overline{s}-1}}\|P_{K}u\|_{F_{\lambda}^{\overline{s}}}^{3}\prod_{i=4}^{k-1}\|P_{K_{i}}u\|_{F_{\lambda}^{1/2}}.
\end{align*}
This sums up to \eqref{eq:LipschitzContinuityNegativeSobolevSpaces}
provided that $s_{2}\geq\frac{2}{3}$.

For the $H_{\lambda}^{\overline{s}}$-estimate, we find similarly
\begin{align*}
N^{s_{2}+\frac{1}{2}}K^{-4s_{2}+\frac{3}{2}}\|P_{N}v\|_{F_{\lambda}^{\overline{s}}} & \Big(\|P_{K}v\|_{F_{\lambda}^{\overline{s}}}\|P_{N}u\|_{F_{\lambda}^{\overline{s}}}^{3}\prod_{i=4}^{k-1}\|P_{K_{i}}u\|_{F_{\lambda}^{1/2}}\\
 & +\|P_{K}v\|_{F_{\lambda}^{\overline{s}-1}}\|P_{N}w\|_{F_{\lambda}^{\overline{s}}}^{3}\prod_{i=4}^{k-1}\|P_{K_{i}}w\|_{F_{\lambda}^{1/2}}\Big).
\end{align*}
This sums up to \eqref{eq:HsContinuityBonaSmithRegularities} provided
that $s_{2}\geq\frac{2}{3}$.

\underline{Subcase II:} $K\in\{M_{5},\dots,M_{k+1}\}$. Note that
$K_{1}\sim K_{2}\sim K_{3}\sim K_{4}$, so we may assume $K_{1}=K_{2}=K_{3}=K_{4}$
are the largest four frequencies. Using \eqref{eq:diff-est-lin-strichartz},
each summand of \eqref{eq:s-1-expansion} is estimated by 
\begin{align*}
 & \quad N^{2s_{2}-\frac{1}{2}}K^{1/2}K_{1}^{1/2}\|P_{N}v\|_{F_{\lambda}^{0}}\|P_{K}v\|_{F_{\lambda}^{0}}\|P_{K_{1}}u\|_{F_{\lambda}^{0}}^{4}\prod_{i=5}^{k-1}\|P_{K_{i}}u\|_{F_{\lambda}^{1/2}}\\
 & \lesssim N^{s_{2}+\frac{1}{2}}K^{\frac{1}{2}-s_{1},-s_{2}+\frac{3}{2}}K_{1}^{-4s_{2}+\frac{1}{2}}\|P_{N}v\|_{F_{\lambda}^{\overline{s}-1}}\|P_{K}v\|_{F_{\lambda}^{\overline{s}-1}}\|P_{K_{1}}u\|_{F_{\lambda}^{\overline{s}}}^{4}\prod_{i=5}^{k-1}\|P_{K_{i}}u\|_{F_{\lambda}^{1/2}}.
\end{align*}
This sums up to \eqref{eq:LipschitzContinuityNegativeSobolevSpaces}
provided that $s\geq\frac{5}{8}$.

For the $H_{\lambda}^{\overline{s}}$-estimate, we find similarly
for each summand of \eqref{eq:s-expansion} 
\begin{align*}
N^{s_{2}+\frac{1}{2}}K^{\frac{1}{2}-s_{1},-s_{2}+\frac{3}{2}}K_{1}^{-4s_{2}+\frac{1}{2}}\|P_{N}v\|_{F_{\lambda}^{\overline{s}}} & \Big(\|P_{K}v\|_{F_{\lambda}^{\overline{s}}}\|P_{K_{1}}u\|_{F_{\lambda}^{\overline{s}}}^{4}\prod_{i=5}^{k-1}\|P_{K_{i}}u\|_{F_{\lambda}^{1/2}}\\
 & +\|P_{K}v\|_{F_{\lambda}^{\overline{s}-1}}\|P_{K_{1}}w\|_{F_{\lambda}^{\overline{s}}}^{4}\prod_{i=5}^{k-1}\|P_{K_{i}}w\|_{F_{\lambda}^{1/2}}\Big).
\end{align*}
This sums up to \eqref{eq:HsContinuityBonaSmithRegularities} provided
that $s_{2}\geq\frac{5}{8}$.

Therefore, the proofs of \eqref{eq:LipschitzContinuityNegativeSobolevSpaces}
and \eqref{eq:HsContinuityBonaSmithRegularities} are completed in
case of $M_{1}\sim M_{2}\sim M_{3}\sim M_{4}$.

The proof of Proposition \ref{prop:EnergyEstimatesDifferences} is
finished. \hfill{}$\square$

\section*{Acknowledgements}

The first author is partly supported by NRF-2016K2A9A2A13003815 (Korea)
through the IRTG 2235 and NRF-2018R1D1A1A0908335 (Korea). The second
author is supported by the German Research Foundation (DFG) through
the CRC 1173, Project-ID 258734477. Much of this work was done when
the first author was visiting Bielefeld University. He would like
to thank the Center for Interdisciplinary Research (ZiF) at Bielefeld
University for its kind hospitality.

\end{document}